\documentclass[11pt]{article}
\usepackage{amsmath}
\usepackage{amsthm}
\usepackage{amssymb,amsfonts}
\usepackage{hyperref}
\usepackage{latexsym}
\usepackage{nicefrac}
\usepackage{epsfig}
\usepackage{setspace}
\usepackage{enumerate}
\usepackage[all]{xypic}
\usepackage{bbm,ifpdf,tikz}
\ifpdf
\fi

\newtheorem{theorem}{Theorem}[section]
\newtheorem{lemma}[theorem]{Lemma}
\newtheorem{proposition}[theorem]{Proposition}
\newtheorem{corollary}[theorem]{Corollary}

{
\theoremstyle{definition}
\newtheorem{definition}[theorem]{Definition}
\newtheorem{example}[theorem]{Example}

\newtheorem{remark}[theorem]{Remark}
}

\newenvironment{prooflinked}{\noindent {\bf Proof of Proposition \ref{prop:linked}:}}{\qed}

\newcommand{\excise}[1]{}
\newcommand{\Spec}{\operatorname{Spec}}

\newcommand{\Specm}{\operatorname{Specm}}

\newcommand{\supp}{\operatorname{Supp}}
\newcommand{\Ext}{\operatorname{Ext}}

\newcommand{\id}{\operatorname{id}}

\renewcommand{\dim}{\operatorname{dim}}
\newcommand{\Ann}{\operatorname{Ann}}
\newcommand{\Sym}{\operatorname{Sym}}

\newcommand{\fin}{{\operatorname{fin}}}
\newcommand{\Int}{{\operatorname{int}}}
\newcommand{\D}{\mathbbm{D}}

\renewcommand{\and}{\qquad\text{and}\qquad}

\newcommand{\Z}{\mathbb{Z}}

\newcommand{\N}{\mathbb{N}}
\newcommand{\R}{\mathbb{R}}
\newcommand{\C}{\mathbb{C}}

\newcommand{\del}{\partial}

\newcommand{\la}{\leftarrow}

\newcommand{\M}{\mathfrak{M}}
\newcommand{\mmod}{{-}{\sf mod}}
\newcommand{\modfin}{\mmod_{\fin}}
\newcommand{\udot}{{\scriptscriptstyle \bullet}}
\newcommand{\bdy}{\partial}

\renewcommand{\la}{\lambda}
\newcommand{\vart}{\vartheta}

\renewcommand{\a}{\alpha}
\newcommand{\ol}{\overline}
\renewcommand{\b}{\beta}

\newcommand{\Hom}{\operatorname{Hom}}

\newcommand{\vb}{{\Lambda}}

\newcommand{\IC}{\operatorname{IC}^{\scriptscriptstyle{\bullet}}}

\newcommand{\cF}{\mathcal{F}}
\newcommand{\cP}{\mathcal{P}}

\newcommand{\cO}{\mathcal{O}}
\newcommand{\cI}{\mathcal{I}}

\newcommand{\fS}{\mathfrak{S}}
\renewcommand{\cD}{\mathcal{D}}
\newcommand{\bS}{\mathbb{S}}
\newcommand{\T}{\mathbb{T}}
\newcommand{\gr}{\operatorname{gr}}
\newcommand{\sgr}{{\mathrm{gr}}}
\newcommand{\spol}{{\mathrm{pol}}}

\renewcommand{\t}{\mathfrak{t}}
\newcommand{\cQ}{\mathcal{Q}}

\newcommand{\cB}{\mathcal{B}}

\newcommand{\Loc}{\operatorname{Loc}}

\newcommand{\nc}{\newcommand}
\nc{\cq}{/\!\!/}
\nc{\red}{{U}}
\nc{\zred}{{U_\fz}}
\nc{\sred}{{\cal U}}
\nc{\szred}{{\cal U}_{\fz}}
\nc{\sEnd}{{\cal E}\mathit{nd}}
\nc{\sHom}{{\cal H}\mathit{om}}
\nc{\Pic}{\operatorname{Pic}}
\nc{\cone}[1]{\text{\textswab C}(#1)}
\nc{\coneZ}[1]{\text{\textswab C}_\Z(#1)}
\newcommand{\fg}{{\mathfrak{g}}}
\newcommand{\fk}{{\mathfrak{k}}}

\newcommand{\fh}{{\mathfrak{h}}}
\newcommand{\fz}{{\mathfrak{z}}}

\nc{\Gd}{G'}

\newcommand{\fb}{{\mathfrak{b}}}

\newcommand{\fM}{{\mathfrak{M}}}
\newcommand{\mg}{{\mathfrak{g}}}

\newcommand{\fI}{{\mathfrak{I}}}

\newcommand{\Umod}{{U\mmod}}
\newcommand{\Ulmod}{{U_\la\mmod}}
\newcommand{\proWM}{\Ulmod_{pro}}
\newcommand{\Ulpmod}{{U_{\la'}\mmod}}
\newcommand{\lf}{{\operatorname{lf}}}
\newcommand{\End}{\operatorname{End}}
\newcommand{\UWM}{\Umod_{\lf}}

\newcommand{\WM}{\Ulmod_{\lf}}

\newcommand{\wh}{\widehat}
\newcommand{\Lleq}{\overset{L}{\leq}}
\newcommand{\Rleq}{\overset{R}{\leq}}
\newcommand{\Tleq}{\overset{2}{\leq}}

\nc{\Ree}{\EuScript{R}}
\newcommand{\cM}{{\cal M}}
\nc{\muq}{\boldsymbol{\mu}}
\renewcommand{\cL}{{\cal Y}} 
\newcommand{\cU}{\mathcal{U}}
\newcommand{\cN}{\mathcal{N}}
\newcommand{\mgs}{\operatorname{Mod}^{\operatorname{good}}_{\bS}(\cU_\la)}

\newcommand{\bT}{\mathbb{T}}
\newcommand{\bt}{\mathbbm{t}}

\newcommand{\secs}{\Gamma_\bS}

\newcommand{\tra}{T}
\newcommand{\ctra}{\mathcal{T}}

\renewcommand{\cH}{\mathcal{H}}
\newcommand{\plusminus}{{+,-}}
\newcommand{\X}{\mathfrak{X}}
\renewcommand{\ol}{\tilde}

\newcommand{\Iv}{{I_\qvb}}

\newcommand{\SHZ}{\SH_\Z}
\newcommand{\SHR}{\SH_\R}
\newcommand{\SgrH}{W}
\newcommand{\SgrHZ}{\SgrH_\Z}
\newcommand{\SgrHR}{\SgrH_\R}

\newcommand{\laurenth}{{(\!(\hh)\!)}}
\newcommand{\hh}{\hbar^{\nicefrac{1}{2}}}
\newcommand{\Weyl}{\mathbb{V}}
\newcommand{\rWeyl}{\mathbb{W}}

\newcommand{\grH}{{H}}
\newcommand{\bH}{\mathbf{H}}
\newcommand{\V}{{\mathbf{V}}}
\newcommand{\qvb}{\mathbf{\Lambda}}
\newcommand{\DDelta}{\mathbf{\Delta}}
\newcommand{\SH}{{\mathbf{W}}}
\newcommand{\PA}{X}
\newcommand{\QPA}{{\mathbf{X}}}
\newcommand{\qH}{\text{\boldmath$\cH$}}

\newcommand{\ordk}{{(k)}}
\newcommand{\Pk}{P^{\ordk}}

\newcommand{\htbmc}{H_{2d,\bT}^{B\! M}\!\left(\fM^+; \C\right)}
\newcommand{\hbmc}{H_{2d}^{B\! M}\!\!\left(\fM^+; \C\right)}
\newcommand{\hmid}{H^{2d}_{\bT}(\fM; \C)}
\newcommand{\htbmm}{H_{2d,\bT}^{B\! M}\!\left(\fM; \C\right)}

\newcommand{\mb}{\mathfrak{b}}
\newcommand{\mh}{\mathfrak{h}}

\begin{document}
\spacing{1.2}

\noindent {\Large \bf 
Hypertoric category $\cO$}
\bigskip

\noindent{\bf Tom Braden}\footnote{Supported by NSA grant H98230-08-1-0097.}\\
Department of Mathematics and Statistics, University of Massachusetts,
Amherst, MA 01003\smallskip \\
{\bf Anthony Licata}\\
Department of Mathematics, Stanford University,
Palo Alto, CA 94305\smallskip \\
{\bf Nicholas Proudfoot}\footnote{Supported by NSF grants DMS-0738335 and DMS-0950383.}\\
Department of Mathematics, University of Oregon,
Eugene, OR 97403\smallskip \\
{\bf Ben Webster}\footnote{Supported by an NSF Postdoctoral Research Fellowship and  by  NSA grant H98230-10-1-0199.}\\
Department of Mathematics, University of Oregon,
Eugene, OR 97403
\bigskip

{\small
\begin{quote}
\noindent {\em Abstract.}
We study the representation theory of the invariant subalgebra of the
Weyl algebra under a torus action, which we call a ``hypertoric enveloping algebra.'' 
We define an analogue of BGG category $\cO$ for this algebra, and identify 
it with a certain category of sheaves on a hypertoric variety.  We prove that a regular block of this category is highest weight
and Koszul, identify its Koszul dual, compute its center, and study its cell structure.
We also consider a collection of derived auto-equivalences analogous to the shuffling
and twisting functors for BGG category $\cO$.
\end{quote}
}
\bigskip

\section{Introduction}
\label{sec:intro}
In this paper we study an algebra $U$ 
analogous to the universal enveloping algebra $U(\fg)$ of a semisimple
Lie algebra $\fg$.  Just as the central quotients of $U(\fg)$ are quantizations
of the ring of functions on the cotangent bundle to the flag variety, the central quotients of our
algebra are quantizations of the ring of functions on a hypertoric variety; for this reason, we
call $U$ the {\bf hypertoric enveloping algebra}.
The most important structure from our perspective is a category $\cO$ of $U$-modules 
analogous to the Bernstein-Gelfand-Gelfand (BGG) category $\cO$ of modules over the universal enveloping algebra. 
Our category $\cO$ shares many beautiful structures and properties with the BGG category $\cO$, including 
a Koszul grading, the presence of ``standard objects'' (analogues of Verma modules), 
a ``cell'' partition of the set of simple objects in a block, and two commuting actions (shuffling and twisting) of 
discrete groups by derived auto-equivalences.  

In the first part of the introduction, we will review how these structures arise in Lie theory, 
before describing the analogous phenomena in the hypertoric setting.  
Let $\mg$ be a semisimple Lie algebra with Cartan and Borel subalgebras $\mh\subset\mb\subset\mg$.
{\bf BGG category \boldmath$\cO$} is defined 
to be the category of finitely generated $U(\mg)$-modules on which $U(\mh)$ acts semisimply 
and $U(\mb)$ acts locally finitely.  The center of $U(\mg)$ 
also acts locally
finitely on objects of $\cO$, which implies that $\cO$ decomposes into {\bf infinitesimal blocks} $\cO_\la$ indexed 
by central characters $\la$ of $U(\mg)$.

Each infinitesimal block has finitely many simple objects 
(at most the order of the Weyl group), and may 
decompose further, depending on the existence of nontrivial extensions between the simple objects.  In this manner
we decompose all of $\cO$ into irreducible {\bf blocks}.  If $\la$ is generic, then $\cO_\la$ is semisimple, and breaks
into one block for each simple object.  At the other extreme, when $\la$ is {\bf integral}, 
the infinitesimal block $\cO_\la$ is itself already a block.
An infinitesimal block is called {\bf regular} if the number of (isomorphism classes of) simple objects is equal to the order of the Weyl group.
In the non-integral case, the constituent blocks of a regular infinitesimal block are called regular, as well.  
Each of these blocks is equivalent to a regular integral block for a semi-simple subalgebra of $\mg$,
a phenomenon which is paralleled in the hypertoric setting.

One of the most powerful tools for studying BGG category $\cO$ is the geometry of the flag variety $G/B$.
For any central character $\la$ of $U(\mg)$, one can define $\la$-twisted D-modules on $G/B$.
If the infinitesimal block $\cO_\la$ is regular, then it is equivalent via the Beilinson-Bernstein localization
theorem to the category of finitely generated $\la$-twisted D-modules on $G/B$ with regular singularities
and singular supports in the conormal varieties to the Schubert strata.  Such D-modules may also be regarded
as sheaves on the cotangent bundle $T^*(G/B)$, the ``Springer resolution" of the nilpotent
cone in $\mg$.
This perspective allows one to define a Koszul grading on $\cO_\la$ 
and to understand various algebraic properties of the category in terms of the geometry of the flag variety and
the Springer resolution.

Kazhdan and Lusztig \cite{KL79} define a partition of the simple objects of $\cO$ called the {\bf two-sided
cell} partition, which can be further refined into {\bf left} and {\bf right cell} partitions.  
The two-sided cell partition induces a direct sum decomposition of the Grothendieck group $K(\cO_\la)_\C$, 
in which the subspaces are spanned by the classes of projective covers of simples in a given two-sided cell.
The set of cells has a natural partial order, and this defines for us a filtration
of $K(\cO_\la)_\C$ that we will call the {\bf cell filtration}.  

Two-sided cells are in bijection
with special $G$-orbits in the nilpotent cone, and the partial order is given by inclusions of closures of orbits.
Let $\bT := \C^\times$ act on $G/B$ in such a way so that the associated Bia{\l}ynicki-Birula stratification agrees with the Schubert stratification.  Let $d := \dim G/B$.
If we regard objects of $\cO_\la$ as twisted D-modules on $G/B$, then the singular support map defines
an isomorphism from $K(\cO_\la)_\C$ to $H_\bT^{2d}\big(T^*(G/B); \C\big)$.  
In the case where $\mg=\mathfrak{sl}_N$,
this isomorphism takes
the cell filtration of $K(\cO_\la)_\C$ to the filtration of $H_\bT^{2d}\big(T^*(G/B); \C\big)$ determined
by the Beilinson-Bernstein-Deligne (BBD) decomposition theorem \cite{BBD,CG97}.\footnote{A statement of this form
can be made for arbitrary $\mg$, but it is more complicated due to the existence of non-special nilpotent orbits
which show up in the BBD picture but not in the cell picture.}

We now summarize some of the well known and important properties of regular integral
blocks of BGG category $\cO$,
many of which we have already stated above.

\begin{theorem}\label{first-intro-thm}
Let $\cO_{\la}$ be a regular integral block of BGG category $\cO$.
\begin{enumerate}
\setlength{\itemsep}{-2pt}
  \item $\cO_\la$ is highest weight 
  (equivalent to the representations of a
    quasihereditary algebra) with respect to some partial ordering of the simple objects \cite{CPS88}.
  \item $\cO_\la$ is Koszul (equivalent to the representations of an algebra with a Koszul grading)
    \cite{BGS96}.
    \item $\cO_\la$ is equivalent via the localization functor to a certain
category of sheaves of modules over a quantization
of the structure sheaf of $T^*(G/B)$ \cite{BB}.
\item The center of the Yoneda algebra of $\cO_\la$ is isomorphic to $H^*(G/B; \C)$ \cite{Soe90}.
\item The complexified Grothendieck group $K(\cO_\la)_\C$ is isomorphic to $H_\bT^{2d}(G/B; \C)$,
and if $\mg = \mathfrak{sl}_N$, the cell filtration corresponds to the BBD filtration \cite{BBD,CG97}.
\item There are two collections of derived auto-equivalences of $\cO_\la$ (shuffling and twisting). 
They define commuting actions of the Artin braid group of $\mg$, which is the fundamental group of the quotient by the Weyl group of the complement to the Coxeter
arrangement for $\fg$ \cite{AS, MOS, BBM04}. 
\item $\cO_\la$ is Koszul self-dual \cite{BGS96}.  The corresponding derived auto-equivalence 
exchanges (graded versions of) the shuffling and twisting functors \cite[6.5]{MOS}.
The permutation of the set of simple objects induced by Koszul duality
sends left cells to right cells, right cells to left cells, and two-sided cells to two-sided cells.
It is order-reversing for the highest weight ordering of the simple objects, and also order-reversing
on the set of two-sided cells.
\end{enumerate}
\end{theorem}

Now we turn to the hypertoric enveloping algebra.  
This algebra, which was originally studied by Musson and Van den Bergh \cite{MvdB},
is easy to define.  Start with the ring 
$\C[x_1,\partial_1,\ldots, x_n, \partial_n]$ of polynomial differential operators
on $\C^n$, which is equipped with the action of an algebraic torus $T\cong (\C^\times)^n$.
The algebra $U$ is defined to be the invariant ring with respect to a subtorus $K\subset T$.
Inside of $U$ is the polynomial subalgebra $\bH$ generated by $x_1\partial_1, \ldots, x_n\partial_n$,
which plays the role of the Cartan subalgebra.  The role of the Borel subalgebra is played by 
the subalgebra $U^+ \subset U$ consisting of elements of 
non-negative weight for a certain action of the
multiplicative group $\bT$; this subalgebra always contains $\bH$.
We note that in our situation there is no analogue of the conjugacy of Borel subgroups;  
different choices of $\bT$-action result in non-isomorphic subalgebras $U^+$.

We define {\bf hypertoric category $\cO$} to be the category of finitely generated $U$-modules
with the property that $U^+$ acts locally finitely and the center $Z(U)$ acts semisimply.
Note that this definition differs in a key way from the definition of BGG category $\cO$:
\begin{itemize}
\item in BGG category $\cO$ the Cartan algebra acts semisimply, while the center acts locally finitely;
\item in our definition, the center acts semisimply, but
$\bH$ only acts locally finitely.
\end{itemize}
In fact, the distinction vanishes if we look only at regular blocks or regular infinitesimal blocks.
A theorem of Soergel \cite{Soe86} says that a regular 
infinitesimal block $\cO_\la$ of BGG category $\cO$ is equivalent to the ``reversed" category obtained by allowing the Cartan
subalgebra to act locally finitely but requiring that the center act semisimply with character $\la$.
Indeed, the proof of Part (3) of Theorem \ref{first-intro-thm} goes through this equivalence, 
thus one can argue that even in the Lie-theoretic setting the reversed category is the more fundamental of the two.

The geometric perspective on BGG category $\cO$ begins with the
observation that a central quotient of $U(\fg)$ can be
realized as the ring of global twisted differential operators on
$G/B$, or equivalently as the ring of $\bS$-invariant global sections
of an equivariant quantization of $T^*(G/B)$, where $\bS := \C^\times$ acts by scaling the fibers.
In our setting the analogue of $T^*(G/B)$ is a hypertoric variety
$\fM$.  Though $\fM$ is not itself
a cotangent bundle, it is a symplectic variety that admits a Hamiltonian $\bT$-action
analogous to the induced $\bT$-action on $T^*(G/B)$ as well as a $\bS$-action analogous to
the scaling action on the cotangent fibers.\footnote{The groups $\bT$ and $\bS$ are both copies of the multiplicative group,
but they play very different roles in this paper.  In particular, $\bT$ acts on both $T^*(G/B)$ and $\fM$
preserving the symplectic forms, while $\bS$ does not.}  
In Section \ref{sec:hypertoric} we construct an equivariant
quantization of $\fM$ whose ring of $\bS$-invariant global
sections is isomorphic to a central quotient of $U$ (there is a unique equivariant quantization for each central quotient).  
This quantization has already been studied by Bellamy and Kuwabara \cite{BeKu}, who
prove an analogue of the Beilinson-Bernstein localization theorem in this context.

The data required to construct a hypertoric variety
along with the necessary group actions are encoded by a linear algebraic object called a {\bf polarized arrangement}.
For any polarized arrangement $\PA$, let $\fM(\PA)$ be the associated hypertoric variety.
The data required to construct a block 
of hypertoric category $\cO$ (a subtorus $K\subset T$, a subalgebra $U^+\subset U$,
a central character of $U$, plus a little bit more data\footnote{The ``little bit more data" is needed if and only
if $\vb_0$ fails to be unimodular or the central character fails to be integral (Remark \ref{inf-block=block}).}) 
are encoded by another,
slightly more complicated linear algebraic object called a {\bf quantized polarized arrangement}.  
For any quantized polarized arrangement $\QPA$, let $\cO(\QPA)$ be the associated block.
We define
what it means for $\QPA$ to be {\bf integral} and {\bf regular}, and we show that the regular integral blocks
are exactly those that have the largest possible number of isomorphism classes of simple objects 
(Remark \ref{number of simples}).
Quantized polarized arrangements and polarized arrangements are closely related; 
in Section \ref{sec:linked} we make this precise by defining certain pairs $\PA$ and $\QPA$ to be {\bf linked}.

Every statement in Theorem \ref{first-intro-thm} has a hypertoric equivalent.  For example, we define left, right,
and two-sided cells in the category $\cO(\QPA)$, as well as a support isomorphism from the Grothendieck
group of $\cO(\QPA)$ to the degree $2d$ equivariant cohomology group of $\fM(\PA)$, where $2d$ is the dimension
of $\fM(\PA)$.
The statement of Theorem \ref{first-intro-thm} is simplified by the fact that the category $\cO_\la$ is Koszul dual to itself.
In the hypertoric setting any regular polarized arrangement $\PA$ has a {\bf Gale dual}
$\PA^!$, and we will see that Gale duality of arrangements corresponds to Koszul duality of categories.

\begin{theorem}\label{second-intro-thm}
Suppose that $\QPA$ and $\QPA^!$ are regular, integral, and linked to a dual pair $\PA$ and $\PA^!$.
\begin{enumerate}\setlength{\itemsep}{-2pt}
  \item $\cO(\QPA)$ is highest weight with respect to a partial ordering of the simple objects (Corollary \ref{qhk}).
  \item $\cO(\QPA)$ is Koszul (Corollary \ref{qhk}).
    \item $\cO(\QPA)$ is equivalent to a certain
category of sheaves of modules over a quantization
of the structure sheaf of $\fM(\PA)$ (Corollary \ref{linked-loc}).
\item The center of the Yoneda algebra of $\cO(\QPA)$ is isomorphic to $H^*(\fM(\PA); \C)$ (Theorem \ref{center}).
\item The complexified Grothendieck group $K(\cO(\QPA))_\C$ is isomorphic to $H_\bT^{2d}(\fM(\PA); \C)$,
and the cell filtration corresponds to the BBD filtration (Theorem \ref{CaCo}).
\item There are two collections of derived auto-equivalences of $\cO(\QPA)$, which we call shuffling and twisting
functors. 
They define commuting actions
of the fundamental groups of the quotients of the complements of the discriminantal arrangements for $\PA$ and $\PA^!$ by the actions of certain finite groups of automorphisms of $U$ 
(Corollary \ref{commute}, Theorems \ref{action} and \ref{TwistShuf Koszul}).
\item $\cO(\QPA)$ is Koszul dual to $\cO(\QPA^!)$.  The associated derived equivalence exchanges shuffling and twisting functors.  The induced bijection between sets of simple objects sends left cells to right cells, right cells to left cells,
and two-sided cells to two-sided cells.  It is order-reversing for the highest weight orderings of the simple objects
of $\cO(\QPA)$ and $\cO(\QPA^!)$, and also order-reversing from the set of two-sided cells of $\cO(\QPA)$
to the set of two-sided cells of $\cO(\QPA^!)$
(Corollary \ref{qhk}, Theorems \ref{opposite preorders} and \ref{TwistShuf Koszul}).\footnote{Unlike in the Lie-theoretic setting, one cannot deduce the reversal of either of the highest weight partial order on the simple objects of $\cO(\QPA)$ or the partial order on the set of two-sided cells from reversal of the other, so these are truly independent statements.}
\end{enumerate}
\end{theorem}

\begin{remark}
A consequence of Parts (5) and (7) of Theorem \ref{second-intro-thm} is that the cohomology groups
of $\fM(\PA)$ and $\fM(\PA^!)$ have dual BBD filtrations (Corollary \ref{dual filtrations}).  This fact has
combinatorial implications that we explore in Remark \ref{matroids}.
\end{remark}

\begin{remark}
Several parts of Theorem \ref{second-intro-thm} are proved by 
means of an equivalence between $\cO(\QPA)$ and the category of modules 
over a finite dimensional algebra $A(\PA)$ that we introduced in \cite{GDKD}
(this equivalence is stated in Theorems \ref{alg=comb} and \ref{blpw} of this paper).
In \cite{GDKD} we proved that $A(\PA)$ is quasihereditary and Koszul, computed the center
of its Yoneda algebra, gave a combinatorial construction of shuffling functors,
and showed that $A(\PA)$ is Koszul dual to $A(\PA^!)$, so this equivalence
immediately implies the corresponding statements for $\cO(\QPA)$.
To establish this equivalence we rely heavily on work of Musson and Van den Bergh \cite{MvdB}.
They give an detailed analysis of $U$ and some of its representation categories, but they never consider the 
subalgebra $U^+\subset U$ or our category $\cO(\QPA)$.
Everything involving the geometry of $\fM(\PA)$ or the cell structure of $\cO(\QPA)$ is new to this paper.
\end{remark}

\begin{remark}
This paper is a part of a larger program initiated by the authors, in which $\fM$
will be replaced by an equivariant symplectic resolution of an affine cone,
and $U$ will be replaced by an algebra whose central quotients are quantizations
of the algebra of functions on $\fM$ \cite{BLPWgco}.  When two symplectic resolutions yield
categories that are Koszul dual as in Parts (2) of Theorems \ref{first-intro-thm} and \ref{second-intro-thm}, 
we call those resolutions
a {\bf symplectic dual pair}.  Above we see that $T^*(G/B)$ is self-dual and $\fM(\PA)$ is dual to $\fM(\PA^!)$.
Other conjectural examples of dual pairs include Hilbert schemes on ALE spaces, which we 
expect to be dual to certain moduli spaces of instantons on $\C^2$, and quiver varieties of simply
laced Dynkin type, which we expect to be dual to resolutions of slices to certain subvarieties
of the affine Grassmannian.
We expect further examples to arise from physics as Higgs branches of the moduli
space of vacua for mirror dual 3-dimensional $\mathcal{N}=4$ superconformal field theories,
or as the Higgs and Coulomb branches of a single such theory.  That 
hypertoric varieties occur in mirror dual theories was observed by Kapustin and Strassler in \cite{KS99}.  
\end{remark}

\noindent
{\em Acknowledgments:}
The authors are grateful to the
Mathematisches Forschungsinstitut Oberwolfach for its hospitality
and excellent working conditions during the preparation of this paper.
Thanks are also due to Gwyn Bellamy for pointing out the 
reference \cite{MvdB} and for helpful comments on a draft, and to Carl Mautner and Ethan Kowalenko for pointing out a mistake in the proof of Lemma \ref{Akoszul}. 

\section{Linear algebraic data}\label{sec:lad}
In this section we introduce the basic linear algebraic constructions that will be crucial
to our analysis of hypertoric category $\cO$.
We define \textbf{polarized arrangements}, which are used to construct hypertoric varieties,
and \textbf{quantized polarized arrangements}, whose 
combinatorics control our category in a manner similar to the
way the Weyl group and associated Coxeter arrangement control 
BGG category $\cO$.

\subsection{The polynomial rings}\label{sec:poly}
Fix a positive integer $n$, and consider the ring
 $$\bH:= \C[h_1^{\pm}, \ldots, h_n^{\pm}]\big/\big\langle h_i^+ - h_i^- + 1\mid i,\ldots, n\big\rangle.$$
It is isomorphic to a polynomial ring on $n$ generators, 
but it is naturally filtered by the semigroup $2\mathbb{N}$ rather than graded, where the $(2k)^{\text{th}}$ piece $F_{2k}\bH$ of the 
filtration is the space of polynomials of degree $\le k$ in the $h^+_i$ (or equivalently in the $h^-_i$).\footnote{It will become clear
in Section \ref{sec:weyl} why we have indexed our filtration by the even natural numbers.}
Its associated graded is the polynomial ring
$$\grH := \gr\bH = \C[h_1,\ldots, h_n], \;\;\;\text{where}\; h_i = h^+_i + F_0\bH = h^-_i + F_0\bH.$$
The maximal spectrum $\SH := \Specm \bH$ is an $n$-dimensional complex affine space.
It is naturally a torsor for the vector space $W := \Specm \grH$.

Both $\SH$ and $\SgrH$ have distinguished
integral structures, given by
$$\SHZ := \{v\in\SH \mid h_i^\pm(v)\in\Z\;\text{for all $i$}\}\and\SgrHZ := \{v \in \SgrH \mid h_i(v) \in \Z\;\text{for all $i$}\}.$$
These in turn induce real structures
$$\SHR := \SHZ\otimes_\Z\R\and\SgrHR := \SgrHZ\otimes_\Z\R.$$

\subsection{Polarized arrangements}\label{sec:polarr}
In this section we use only the graded ring $\grH$ and the vector space $\SgrH$ from Section \ref{sec:poly};
the filtered ring $\bH$ and the affine space $\SH$ will be used in the next section.

\begin{definition} A \textbf{polarized arrangement} is a triple $\PA = (\vb_{0}, \eta, \xi)$, where
$\vb_0 \subset \SgrHZ$ is a direct summand,
 $\eta$ is a $\vb_0$-orbit in $\SgrHZ$, and
$\xi \in \vb_0^*$.  
To avoid some degenerate possibilities, we will always assume that $h_i(\vb_0) \ne 0$ for every $i$, so 
$\vb_0$ is not contained in any coordinate hyperplane, and that $\vb_0$ doesn't contain any of the 
coordinate axes.
\end{definition}

Given a polarized arrangement, let
$$V_0 := \C\vb_0\subset\SgrH\and V_{0,\R} := \R\vb_0\subset\SgrHR$$ be the complex and 
real vector spaces spanned by $\vb_0$, and let $$V_{\R} := \eta + V_{0, \R}\subset\SgrHR.$$
For each $i\in\{1,\ldots,n\}$, let $$H_i := \{v \in V_{\R} \mid h_i(v) = 0\},$$ and let $\cH := \{H_i\mid i=1,\ldots,n\}$
be the associated (multi)arrangement.
The assumption that $h_i(\vb_0) \ne 0$ ensures that each $H_i$ is really a hyperplane.

We will also need the corresponding central arrangement $\cH_0$ in $V_{0,\R}$, but to avoid  notational unpleasantness,
we will refrain from giving names to its hyperplanes.  If we identify $V_{\R}$ with $V_{0,\R}$ by choosing
an origin (that is, by choosing a lift of $\eta$ to $\SgrHZ$), then the arrangement $\cH_0$ is obtained
from $\cH$ by translating all of the hyperplanes to the origin.
We say that $\eta$ is \textbf{regular} if the arrangement $\cH$ is simple, which means that 
no point in $V_{\R}$ lies on more than $\dim V$ hyperplanes.  This is a genericity assumption
with respect to the positions (but not the slopes) of the hyperplanes in $\cH$.

A \textbf{flat} of the arrangement $\cH$ or $\cH_0$ is any nonempty intersection of the hyperplanes.
The parameter $\xi$ determines a linear functional on $V_{0,\R}$, and an affine linear functional on $V_\R$
that is only defined up to a constant.
We say that the covector $\xi$ is \textbf{regular} if it is not constant on any 
one-dimensional flat of $\cH_0$, or equivalently on any 
one-dimensional flat of $\cH$.
We say that $\PA = (\vb_0, \eta,\xi)$ is \textbf{regular} if both $\eta$ and $\xi$ are regular.

\begin{remark}
This definition differs slightly from the one in \cite[2.1]{GDKD}.  In that paper $\eta$ and $\xi$ were allowed
to be real rather than integral, but they were always assumed to be regular.
\end{remark}
 
\subsection{Quantized polarized arrangements}\label{sec:qpd}
In this section we introduce objects that are analogous to those considered in Section \ref{sec:polarr},
but with $\grH$ and $\SgrH$ replaced by $\bH$ and $\SH$.

\begin{definition} A \textbf{quantized polarized arrangement} is a triple $\QPA = (\vb_0, \qvb, \xi)$, where
$\vb_0 \subset \SgrHZ$ is a direct summand, 
$\qvb$ is a $\Lambda_0$-orbit in $\SH$, and $\xi \in \Lambda_0^*$.
\end{definition}

Given a quantized polarized arrangement, let
$$\V := \qvb + V_0 = \qvb + \C\vb_0\subset\SH\and \V_\R := \qvb + V_{0, \R} = \qvb + \R\vb_0\subset\SHR.$$

Let $I_\qvb$ be the set of indices $i\in \{1, \dots, n\}$ for which $h_i^+(\qvb) \subset \Z$ 
(or equivalently $h_i^-(\qvb) \subset \Z$).  We say that $\QPA$ and $\qvb$ are \textbf{integral} if
$\qvb \subset \SHZ$, or equivalently if $I_\qvb = \{1, \dots, n\}$.  For each $i \in I_\qvb$, let
\[H_i^+ := \{v \in \V_\R \mid h_i^+(v) = 0\} \and H_i^- := \{v \in \V_\R \mid h_i^-(v) = 0\},\]
and let $\qH := \{H^\pm_i \mid i \in I_\qvb\}$ be the associated (multi)arrangement.

\begin{remark}  The definitions of polarized and quantized polarized arrangements are clearly very close.
One important difference is that the parameter $\eta$ is required to be integral (it is a $\vb_0$-orbit in $\SgrHZ$),
while the parameter $\qvb$ can sit anywhere in the complex vector space $\SH$.
This difference is unavoidable:  to define a hypertoric variety it is necessary for $\eta$ to be integral, but to study
the hypertoric enveloping algebra it is necessary to consider arbitrary $\qvb$.  
\end{remark}

The definition of regularity for the parameter $\qvb$ is more subtle than for $\eta$, because
of the more complicated geometry of the doubled hyperplanes and certain integrality issues.
Ignoring the integrality issues for the moment, we make the following definitions.
A hyperplane arrangement is called {\bf essential} if it has a zero-dimensional flat.  The arrangement
$\cH$ is always essential, but $\qH$ is not, since the index set $I_\qvb$ may be very small or even empty.
For a $\vb_0$-orbit $\qvb$ in $\SH$, let $m$ be the maximum over all points in $\V_\R$ of the number of 
pairs of hyperplanes $H_i^\pm$ in between which the point lies:
\[m := \max_{v \in \V_\R} \left|\{i \in I_\qvb \mid 0 < h_i^-(v) < 1\}\right|.\]
If $\qH$ is essential, then we necessarily have $m\geq\dim V$.  

\begin{definition}\label{def:quasi-reg}
We say that $\qvb$ is {\bf quasi-regular} if $\qH$ is essential and $m = \dim V$.  This means 
that any arrangement obtained by replacing each pair $H_i^\pm$ with a single hyperplane lying strictly between them is simple.
\end{definition}

The definition of regularity of $\qvb$ and $\QPA$ will appear in the next section.

\subsection{Boundedness and feasibility}\label{sec:bf}
Fix a polarized arrangement $\PA = (\vb_0, \eta,\xi)$ 
and a quantized polarized arrangement $\QPA = (\vb_0,\qvb, \xi)$ with 
the same underlying lattice $\vb_0$ and the same covector $\xi \in \vb_0^*$.
The following definitions for $\PA$ are repeated from
\cite{GDKD}; we then adapt them for $\QPA$.

For a sign vector $\a \in \{+, -\}^n$, define the chamber 
$\Delta_\a \subset V_{\R}$ 
to be the polyhedron cut out by the inequalities
\begin{equation*}\label{chamber ineqs}
h_i \ge 0\text{ for all $i$ with }\a(i) = + \and h_i \le 0 \text{ for all $i$ with }\a(i) = -,
\end{equation*}
and let $\Delta_{0,\a}$ be the polyhedral cone in $V_{0,\R}$ cut out 
by the same inequalities.
If $\Delta_\a \ne \emptyset$, we say that $\a$ is \textbf{feasible} for $\PA$;
let $\cF_\PA$ be the set of feasible sign vectors.
If $\a$ is feasible, then $\Delta_{0,\a}$ is the cone of unbounded directions in $\Delta_\a$.
Note, however, that $\Delta_{0,\a}$ is always nonempty, even if $\a$ is infeasible.

We say that $\a$ is \textbf{bounded} for $\PA$ if the restriction
of $\xi$ is proper and bounded above on the cone $\Delta_{0,\a}$.
Note that if $\xi$ is regular, we can
drop the properness hypothesis.  Also note that if $\a$ is feasible,
then $\a$ is bounded if and only if $\xi$ is proper and bounded above
on $\Delta_{\a}$.  Let $\cB_\PA$ be the set of bounded 
sign vectors, and let $\cP_{\PA} = \cF_\PA \cap \cB_\PA$.

\begin{remark}\label{bijection}
When $\PA$ is regular, the set $\cP_\PA$
is in natural bijection with the set of vertices of the 
hyperplane arrangement $\cH$, as each vertex appears 
as the $\xi$-maximal point of
$\Delta_{\a}$ for a unique bounded feasible sign vector $\a$.
\end{remark}

\begin{remark}\label{abuse}
It will often be the case that we will fix a lattice $\vb_0$ and vary the parameters $\eta$ and $\xi$.
Since $\cF_\PA$ depends only on $\vb_0$ and $\eta$, we will often abusively write $\cF_\eta$
rather than $\cF_\PA$.  Likewise, since $\cB_\PA$ depends only on $\vb_0$ and $\xi$, we will
often write $\cB_\xi$ rather than $\cB_\PA$.  We will then write $\cP_{\eta,\xi} = \cF_\eta\cap\cB_\xi =  \cP_\PA$.
This notation coincides with that of \cite{GDKD}.
\end{remark}

\begin{remark}\label{totally bounded}
We have $\Delta_{0,\a} = \{0\}$ if and only if $\a\in\cB_\xi$ for any choice of $\xi$.
We call such a sign vector \textbf{totally bounded}.  If $\a$ is
feasible, then it is totally bounded if and only if the polyhedron $\Delta_{\a}$ is compact.
The set of totally bounded sign vectors depends only on $\vb_0$;
it is independent of both $\eta$ and $\xi$.
\end{remark}

Turning to the quantized polarized arrangement $\QPA$, we 
define the chamber $\DDelta_\a$
corresponding to $\a \in \{+, -\}^{I_\qvb}$
to be the subset of the affine space $\V_\R$ cut out by the inequalities
\begin{equation*}\label{dd}
h^+_i \ge 0\text{ for all $i \in I_\qvb$ with }\a(i) = + \and h^-_i \le 0 \text{ for all $i\in I_\qvb$ with }\a(i) = -.
\end{equation*}
If $\DDelta_\a \cap \qvb$ is nonempty, we say that $\a$ is \textbf{feasible} for $\QPA$, and we
denote the set of feasible sign vectors $\cF_\QPA$.
Since boundedness depends only on the $\vb_0$ and the covector $\xi$, the
definition of boundedness for $\QPA$ is the same as for a polarized arrangement $X$:
we let the set of bounded sign vectors be $\cB_\QPA = \cB_\xi$.
Note that if $\a \in \cF_\QPA$, then $\a$ is bounded if and only if $\xi$ is proper and bounded above
on $\DDelta_{\a}$ (or equivalently on $\DDelta_\a\cap\qvb$).

\begin{remark}\label{quantum abuse}
Following the conventions of Remark \ref{abuse}, we will often use the notation $\cF_\qvb$
in place of $\cF_\QPA$, since this set does not depend on $\xi$.  Unlike in the non-quantized case,
the set $\cB_\QPA$ does depend on $\qvb$ because $I_\qvb$ depends on $\qvb$.  
We will thus write $\cB_{\qvb,\xi}$
in place of $\cB_\QPA$, and $\cP_{\qvb,\xi} = \cF_\qvb\cap \cB_{\qvb,\xi} = \cP_\QPA$.
\end{remark}

It is possible for the polyhedron $\DDelta_\a$ to be nonempty, but so small that it does not
contain an element of $\qvb$.  We are primarily interested in quantized polarized arrangements for which this
does not happen, thus we incorporate this condition into our definition of regularity.

\begin{definition}
We say that $\qvb$ is \textbf{regular} if it
is quasi-regular and $\DDelta_\a \ne \emptyset$ implies $\a\in\cF_\qvb$ for every $\a \in \{+, -\}^n$.
\end{definition}

\begin{remark}\label{rmk:unimodular} We say that $\vb_0$ is \textbf{unimodular} if its image
under the projection $\SgrHZ \cong \Z^n \to \Z^I$ is a direct summand of
$\Z^I$ for every $I\subset \{1, \dots, n\}$.  In this case regularity 
and quasi-regularity are equivalent, since any vertex of a chamber
$\DDelta_\a$ must lie in $\qvb$.
\end{remark}

\colorlet{bounde}{blue!20}

\begin{figure}

\begin{center}
\begin{tikzpicture}
\node at (-4,0){
\begin{tikzpicture}[thick,scale=2,every node/.style={fill=white}]
\fill[bounde] (intersection of -2,-2 -- 2,2 and  -2,-.6 -- 2,-.6) -- (intersection of -2,-2 -- 2,2 and .6,-2 -- .6,2) -- (intersection of -2,-.6 -- 2,-.6 and .6,-2 -- .6,2) -- cycle;

\fill[bounde] (intersection of -2,-2 -- 2,2 and  -2,-.8 -- 2,-.8) -- (-2,-2) -- (.6,-2) -- (intersection of -2,-.8 -- 2,-.8 and .6,-2 -- .6,2) -- cycle;

\fill[bounde] (intersection of -2,-1.8 -- 1.8,2 and  -2,-.8 -- 2,-.8) -- (-2.2, -2) -- (-2.4,-2) -- ( -2.4,-.8) -- cycle;

\foreach \x in {-2.4,-2.2,...,2}
\foreach \y in {-2,-1.8,...,2}
\fill[color=gray] (\x,\y) circle (.5pt);;
\draw[gray] (-2.2,-2) -- (1.8,2) node[at start,anchor=north east]{$h_1^-=0$};
\draw[gray] (-2.5,-.8) -- (2,-.8) node[at start,anchor=10]{$h_2^-=0$};
\draw[gray] (.8,-2) -- (.8,2) node[at start,anchor=145]{$h_3^-=0$};
\draw (-2,-2) -- (2,2) 
node[at start,anchor=145]{$h_1^+=0$};
\draw (-2.5,-.6) -- (2,-.6) 
node[at start, anchor=-10]{$h_2^+=0$}; 
\draw (.6,-2) -- (.6,2) 
node[at start,anchor=35]{$h_3^+=0$};
\node[scale=1.1,inner sep=0 pt ,outer sep=0 pt, fill=bounde] (a) at (.18,-.38) {$\Delta_{(+,+,+)}$};
\node[scale=1.1] at (-1.07,.3) {$\Delta_{(-, +, +)}$};
\node[scale=1.1,fill=bounde] at (-.27,-1.5) {$\Delta_{(+,-,+)}$};
\node[scale=1.1] at (1.53,.3) {$\Delta_{(+,+,-)}$};
\node[scale=1.1,inner sep=0 pt ,outer sep=0 pt, fill=bounde] at (-1.95,-1.2) {$\Delta_{(-,-,+)}$};
\node[scale=1.1] at (1.51,-1.5) {$\Delta_{(+,-,-)}$};
\node[scale=1.1,inner sep=0 pt ,outer sep=0 pt] at (1.28,2) {$\Delta_{(-,+,-)}$};
\fill[white] (-2.5,0) -- (-.4,2.1) -- (-2.5,2.1) -- cycle;
\draw[->,thick] (-1.8,1.2) -- (-1.2,1.8) node[above,midway, fill=none]{$\xi$} ;
\end{tikzpicture}};
\end{tikzpicture}
\end{center}

\caption{A regular quantized polarized arrangement.  There are 7 feasible sign vectors, 3 of which are bounded.
The polyhedra associated to bounded feasible sign vectors are shaded.  
}
\label{quant-fig}
\end{figure}
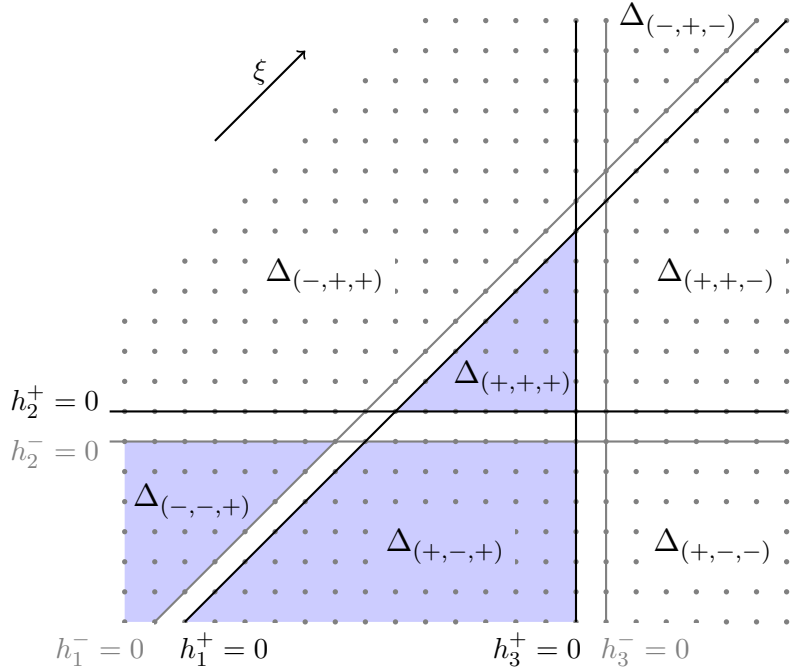

\subsection{Linked arrangements}\label{sec:linked}
In this section we fix a direct summand $\vb_0\subset\SgrHZ$
and an arbitrary element $\xi\in\vb_0^*$.  For any $\vb_0$-orbit $\qvb\subset\SH$,
let $\pi:\{+,-\}^n \to \{+,-\}^{I_\qvb}$ be the projection.

\begin{definition}\label{def:linked}
We say that two regular parameters $\eta$ and $\qvb$ are \textbf{linked} if $\pi(\cF_\eta) = \cF_\qvb$.
In this case we also say that the polarized arrangement $\PA = (\vb_0,\eta,\xi)$ is linked
to the quantized polarized arrangement $\QPA = (\vb_0,\vb,\xi)$.
\end{definition}

\begin{remark}
In general, a polarized arrangement does not determine a quantized polarized arrangement, nor
the other way around.  Linkage is the only property that we will consider that relates these two types of objects.
\end{remark}

\begin{definition}\label{def:equivalence}
We will consider two polarized arrangements (respectively quantized polarized arrangements)
to be {\bf equivalent} if they have the same direct summand $\vb_0\subset\SgrHZ$
and the same sets $\cF_\eta$ and $\cB_\xi$ (respectively $I_\qvb$,
$\cF_\qvb$, and $\cB_{\qvb,\xi}$) of feasible and bounded sign vectors.
\end{definition}

It is clear that every regular $\eta$ is linked to some regular integral $\qvb$ and {\em vice versa}, hence
the concept of linkage provides a bijection between equivalence classes of regular polarized arrangements
and equivalence classes of regular integral quantized polarized arrangements with the same lattice $\vb_0$
and covector $\xi$.
The following proposition will be a key tool for us in Section \ref{sec:gco}.

\begin{proposition}\label{prop:linked}
Two regular parameters
$\eta$ and $\qvb$ are linked if and only if there exists a positive integer $k$ 
such that $\cF_{\qvb} = \cF_{\qvb + r k\eta}$ for all positive integers $r$.
If $\vb_0$ is unimodular, then $k$ may be taken equal to 1.
\end{proposition}

Since we need to consider the polyhedra $\Delta_\a$ and $\DDelta_\a$ and arrangements $\qH$ for
varying choices of $\eta$ and $\qvb$, we add them to the notation 
for this section only, writing $\Delta_{\eta,\a}$, $\DDelta_{\qvb,\a}$, and $\qH_\qvb$.
Let $N_\qvb$ be the number of subsets $I\subset I_\qvb$ such that the composition
$\vb_\R\subset\SgrHR\twoheadrightarrow\R^I$ is surjective; such subsets are commonly
known as {\bf independent sets} of the matroid associated to $\cH_0$.   

\begin{lemma}\label{chambers}
For any $\qvb$, $|\cF_\qvb|\leq N_\qvb$, and equality is attained
if and only if $\qvb$ is regular.
\end{lemma}

\begin{proof}
If $\qvb$ is regular, then it is possible to replace each pair $H_i^\pm$ of hyperplanes
in $\qH_\qvb$ with a single hyperplane lying strictly in between them in such a way so that the new arrangement
of $|I_\qvb|$ hyperplanes is simple, and every chamber of the new arrangement contains a unique non-empty
set of the form $\DDelta_{\qvb,\a}\cap\qvb$.
(To visualize this, see Figure \ref{quant-fig}.)  Thus the feasible sign vectors $\cF_\qvb$ are in bijection with the
chambers of the new arrangement, which is a simplification
of our central arrangement $\cH_0$.

The Orlik-Solomon algebra of any arrangement has dimension equal to the
number of chambers, and if the arrangement is simple, it has a basis indexed by the independent sets
of the matroid associated to the corresponding central arrangement \cite[3.45 \& 5.95]{OT}.  
Since the number of chambers is equal to $|\cF_\qvb|$
and the size of the basis is equal to $N_\qvb$, this proves the equality when $\qvb$ is regular.

If $\qvb$ is not regular, it is still possible to construct a simplification of $\cH_0$ in the manner described
above, but some of the chambers will not contain any nonempty sets of the form $\DDelta_{\qvb,\a}\cap\qvb$.
Thus the number of chambers is strictly greater than $|\cF_\qvb|$, which proves the inequality.
\end{proof}

\begin{prooflinked}
Suppose that $\eta$ and $\qvb$ are linked.
Choose a positive integer $k$ such that $\Delta_{k\eta,\a} \cap \SgrHZ \ne \emptyset$ 
for every $\a\in\cF_\eta$.  
The same condition will hold when
$k$ is replaced with any positive multiple $rk$.  
When $\vb_0$ is unimodular, we can take $k=1$ (see Remark \ref{rmk:unimodular}).

For any positive integer $r$ and any $\a \in \{+,-\}^n$, we have 
$$(\Delta_{rk\eta, \a}\cap \SgrHZ) + (\DDelta_{\qvb, \pi(\a)}\cap\qvb) \subset (\DDelta_{\qvb+ rk\eta,\pi(\a)}\cap\qvb),$$
which tells us that
$\cF_\qvb \subset \cF_{\qvb+ rk\eta}$. Since
$$|\cF_\qvb| = N_\qvb = N_{\qvb+ rk\eta} \geq |\cF_{\qvb+ rk\eta}|,$$ this inclusion must be an equality.

Conversely, suppose that there exists a positive integer $k$ 
such that $\cF_{\qvb} = \cF_{\qvb + r k\eta}$ for all positive integers $r$.  
Choose an arbitrary element of $\SH$; this allows us to identify $\SH$ with $\SgrH$.
Choose an arbitrary sign vector $\a$.
As $r$ approaches $\infty$, the set $\displaystyle{\frac{1}{rk}\DDelta_{\qvb+rk\eta,\pi(\a)}}\subset\SH\cong\SgrH$
approaches a set containing $\Delta_{\eta,\a}\subset\SgrH$.
In particular, when $r$ is very large, one set is non-empty if and only if the other is,
so we have $\cF_\eta = \cF_{\qvb+rk\eta} = \cF_{\qvb}$.
Thus $\eta$ and $\qvb$ are linked.
\end{prooflinked}

\subsection{Gale duality}\label{sec:gd}
We recall the notion of Gale duality for
polarized arrangements from \cite[2.3]{GDKD}.  
Fix a positive integer $n$, and let $\PA = (\vb_{0}, \eta, \xi)$ and $\PA^! = (\vb^!_0, \eta^!,\xi^!)$ be two polarized
arrangements with $n$ hyperplanes.

\begin{definition}\label{def:gd}
We say that $\PA^!$ is {\bf Gale dual} to $\PA$ if 
\begin{itemize}
\item $\vb^!_0$ and $\vb_0$ are complementary with respect to the coordinate inner product on $\SgrH_\Z \cong \Z^n$
\item $\eta^! = - \xi$ and $\xi^! = -\eta$ under the resulting
identifications $\SgrHZ/\vb^!_{0}\cong\vb_0^*$ and $(\vb^!_{0})^*\cong \SgrHZ/\vb_{0}$.
\end{itemize}
\end{definition}

\begin{theorem}[\mbox{\cite[2.4]{GDKD}}]\label{bffb}
If $\PA$ and $\PA^!$ are Gale dual,
then $$\cF_\eta = \cB_{\xi^!}
\qquad\text{and}\qquad
\cF_{\eta^!} = \cB_\xi.$$
In particular, we have $\cP_{\PA} = \cP_{\PA^!}$.
Furthermore, $\eta$ is regular if and only if $\xi^!$ is 
regular, and $\xi$ is regular if and only if $\eta^!$ is
regular.
\end{theorem}

There is no direct way to define Gale duality for
quantized polarized arrangements.  
However, if both $\QPA$ and $\QPA^!$ are regular and integral, 
then we will say that $\QPA$ and $\QPA^!$ are Gale dual if they
are linked to a pair of Gale dual polarized arrangements $\PA$ and 
$\PA^!$.

\section{The hypertoric enveloping algebra}
\label{sec:ring}
Recall that a quantized polarized arrangement consists of a triple $(\vb_0,\qvb,\xi)$.
The hypertoric enveloping algebra itself is determined by the parameter $\vb_0$,
and the affine space $\V = \qvb + \C\vb_0 \subset \SH$ determines a central character of the
algebra (Section \ref{sec:algdef}).  There are of course many different choices of $\qvb$ that
yield the same $\V$; this choice of a lattice $\qvb\subset\V$ determines a certain subcategory
of modules over the algebra with central character given by $\V$ (Section \ref{sec:weight}).  
The entirety of Section \ref{sec:ring} is devoted to understanding
this subcategory, which we describe in combinatorial terms in Section \ref{sec:quiver}.
In Section \ref{translation functors}, we consider translation functors between these subcategories for various
choices of $\V$ and $\qvb$.

The parameter $\xi$ will not enter the picture until Section \ref{hyper O}, where we use it to define
hypertoric category $\cO$.  The intersection of $\cO$ with the category studied in Section \ref{sec:ring} will be a block of $\cO$.

\subsection{The Weyl algebra}\label{sec:weyl}
Fix an integer $n$, and consider the $n$-dimensional torus $T := \Specm\C[W_\Z]$,
where $\C[W_\Z]$ denotes the group ring of the lattice $W_\Z$ that was introduced in Section \ref{sec:poly}.
Thus the Lie algebra $\t$ of $T$ is naturally identified with $\C\{h_1,\ldots,h_n\}$, $\Sym(\t)$ is naturally identified with $H$,
and $\t^*$ is naturally identified with $W$.  The character lattice $\t^*_\Z$ is equal to $W_\Z$, and the cocharacter
lattice $\t_\Z$ is the lattice spanned by $h_1,\ldots,h_n$.

Consider the coordinate vector space $\C^n := \Specm\C[x_1,\ldots,x_n]$.
We let $T$ act on $\C^n$ in a manner such that the induced action of $\t$ on the ring of regular
functions is given by $h_i\cdot x_j = \delta_{ij} x_j$.\footnote{Note that if we use the basis 
$\{h_1,\ldots,h_n\}$ of $\t$ to identify $T$ with the coordinate torus
$(\C^\times)^n$, then the action of $T$ on $\C^n$ is the {\em opposite} of the standard action.}
Let $\D$ be the Weyl algebra of polynomial differential operators on $\C^n$.
The ring $\D$ is generated over $\C[x_1,\ldots,x_n]$ by pairwise commuting elements $\{\del_1,\ldots,\del_n\}$
that satisfy the relations $[\del_i, x_j] = \delta_{ij}$ for all $1 \le i, j \le n$.  
The action of $T$ on $\C^n$ induces an action on $\D$, along with a grading
$$\D = \bigoplus_{z\in W_\Z}\D_z$$
of $\D$ by the character lattice of $T$.  We identify the $0^\textrm{th}$ graded piece
$$\D_0 = \D^T = \C[x_1\del_1,\ldots, x_n\del_n] = \C[\del_1x_1,\ldots,\del_nx_n]$$
with the ring $\bH$ from Section \ref{sec:poly} by sending $x_i\del_i$ to $h_i^+$ and $\del_ix_i$ to $h_i^-$.
This algebra will play a role for us analogous to a Cartan subalgebra of
a semisimple Lie algebra.

We filter $\D$ by the semigroup $\mathbb{N}$ by letting $F_k\D$ be the linear span of all monomials
of total degree $\leq k$.  In particular, we have $F_{2k}\bH = \bH \cap F_{2k}\D$, where $F_{2k}\bH$ is the filtered
piece introduced in Section \ref{sec:poly}.
The associated graded ring $\gr \D$ is canonically identified with
functions on $T^*\C^n$, with its $T$-action induced by the action on $\C^n$.
We have the following link between the algebra structure on $\D$ and the $T$-action: 
let $\text{\boldmath$\sigma$} \in F_2\bH\cong F_2\D_0$, and let $\sigma$ be its image in the degree 2
part of $H = \gr\bH$, which we have identified with $\t$.
Then for all $z\in W_\Z \cong \t^*$ and $a\in \D_z$, we have
\begin{equation}\label{moment map}
[\text{\boldmath$\sigma$},\, a] = z(\sigma)a\in \D_z. 
\end{equation}

\subsection{The hypertoric enveloping algebra}\label{sec:algdef}
In this section we fix a direct summand $\vb_0\subset\SgrHZ$,
and use it to define the hypertoric enveloping algebra.
Let $V_0 := \C\vb_0\subset W\cong\t^*$, and let $\fk := V_0^\perp\subset W^*\cong\t$.
Let $K\subset T$ be the connected subtorus with Lie algebra $\fk$, so that $\vb_0$
may be identified with the character lattice of $T/K$ and $\SgrHZ/\vb_0$ may be identified with the character
lattice of $K$.
Our main object of study is the ring of $K$-invariants
\[U := \D^K = \bigoplus_{z \in \vb_0} \D_z,\]
which we will call the {\bf hypertoric enveloping algebra}
associated to $\vb_0$.

From Equation \eqref{moment map} it is easy to see that 
the center $Z(U)$ is the subalgebra generated by 
all $\text{\boldmath$\sigma$} \in F_2\bH$ whose image $\sigma\in\t$ lies in $\fk$.
In particular, $Z(U)\subset\bH$, and its maximal spectrum $\Specm Z(U)$ is naturally 
the quotient of the affine space $\SH = \Specm \bH$ by the action of $V_0\subset W$.

Let $\V\subset\SH$ be a $V_0$-orbit, and let
$\la\colon Z(U) \to \C$ be the associated central character of $U$.
We stress that $\V$ and $\la$ completely determine each other. 
Let $U_\la := U/U\langle\ker \la\rangle$ be the corresponding
central quotient, and let $\bH_\la := \bH/\bH\langle \ker \la\rangle$ be
the image of $\bH$ in $U_\la$.

\begin{remark}\label{grul} 
There is a unique splitting of the surjection $F_2Z(U) \to \fk$
whose image is the kernel of $\la$.  The induced map
$\mu_\la:\Sym(\fk) \to Z(U) \to \D$ 
is known as a \textbf{quantized moment map} for the $K$-action,
since the associated graded map 
\[\gr \mu_\la: \Sym(\fk) \overset{\sim}{\longrightarrow} \gr Z(U) \to \gr \D \cong \C[T^*\C^n]\]
is given by composition with the classical
moment map $T^*\C^n \to \fk^*$.  The ring $U_\la$ is known as a {\bf noncommutative Hamiltonian reduction}
of $\D$ by $K$ \cite{CBEG}.
The associated graded algebra
\begin{eqnarray*}
\gr U_\la &=& \gr\left(\D^K\Big/\big\langle \mu_\a(\fk)\big\rangle\right)\\
&\cong& \left(\gr\D\right)^K\Big/\big\langle \gr\mu_\a(\fk)\big\rangle\\
&\cong& \C[T^*\C^n]^K\Big/\big\langle \gr\mu_\a(\fk)\big\rangle
\end{eqnarray*}
is isomorphic to the coordinate ring of the symplectic quotient of
$T^*\C^n$ by $K$, which is a hypertoric variety (see Proposition \ref{prop-grul}).
This ring inherits a natural Poisson structure from 
the symplectic structure on the hypertoric variety, and $U_\la$ is
a quantization of this Poisson structure.
Since $\gr U_\la$ is a finitely generated $\C$-algebra, $U_\la$
is finitely generated as well.
\end{remark}

\begin{remark}\label{module Y}
The algebra $U_\la$ can also be realized as the ring of $K$-equivariant endomorphisms of the right
$\D$-module $$Y_\la := \D\,\big/\,\big\langle \ker \la\big\rangle\,\D.$$  The isomorphism is given by sending an endomorphism $\psi$ to $\psi(1)\in Y_\la^K\cong U_\la$.
\end{remark}

\subsection{Weight modules}\label{sec:weight}
Let $\D\mmod$, $\Umod$, and $\Ulmod$ denote the category of finitely generated
left $\D$-modules, $U$-modules, and $U_\la$-modules, respectively.  
(All modules over any ring in this paper will be assumed to be finitely generated.)
We will mainly be interested in modules which
decompose into (generalized) weight spaces for the 
action of $\bH$.  For $v \in \SH = \Specm\bH$, let $\mathfrak{I}_v \subset\bH$ be the associated maximal ideal.
For any module $M \in \Umod$, define the 
\textbf{\boldmath$v$-weight space} of $M$ to be
\[M_v := \{m \in M \mid \mathfrak{I}_v^k m = 0\;\text{for $k \gg 0$}\}.\]  

\begin{remark}
It is more conventional to call $M_v$ a ``generalized weight space", and reserve the term ``weight space"
for the more restrictive $k=1$ condition.   
In this article, however, we will never be interested in weight spaces in the usual sense, 
and will always use ``weight space" to mean the subspace defined above.
\end{remark}

Equation \eqref{moment map} implies that 
$\mathfrak{I}_{z+v}a = a \mathfrak{I}_v$ for all $z\in\vb_0\subset\SgrHZ$, $v\in\SH$, and $a\in\D_z$,
and therefore that $$\D_z\cdot M_v \subset M_{z+v}.$$
For any $U$-module $M$, define its \textbf{support}
by \[ \supp M := \{v \in \SH \mid M_v \neq 0\}.\]
Let $\D\mmod_\lf$, $\UWM$, and $\WM$ be the full subcategories of $\D\mmod$, $\Umod$, and $\Ulmod$
consisting of modules for which $H$ acts locally finitely.
Objects of these categories will be called \textbf{weight modules}; these are exactly the modules
that are isomorphic to the direct sum of their weight spaces.

So far we have chosen an integer $n$, a direct summand $\vb_0\subset\SgrHZ$,
and a central character $\la:Z(U)\to\C$. 
Now choose a $\vb_0$-orbit $\qvb \subset \V$,
and let $\Ulmod_\qvb$ be the full subcategory of $\WM$ consisting of modules
supported in $\qvb$.  This is the category on which we will focus our attention for
the rest of Section \ref{sec:ring}.

\begin{remark} Since a point $v \in \SH$ lies in $\V$ if and only if 
$\ker \la \subset \mathfrak{I}_v$, the category $\WM$ can be thought of as
the subcategory of $\UWM$ consisting of objects that are \emph{scheme-theoretically} supported
on $\V$.  However, an object $M$ of $\UWM$ can have $\supp M \subset \V$
and still not lie in $\WM$.  We will consider the larger category of
weight modules for $U$ with set-theoretic support in $\qvb$ in Section
\ref{sec:Deformed O} when we discuss a deformation of our category $\cO$.
\end{remark}

Define functors
\[(\cdot)^\la\colon \D\mmod_\lf \to \WM\quad \text{and}\quad 
(\cdot)^\qvb\colon \D\mmod_\lf \to \Ulmod_\qvb \]
by 
$$M^\la := \left\{ m \in \bigoplus_{v \in \V} M_v \,\,\Big{|}\,\, (\ker \la) m = 0\right\}
\qquad\text{and}\qquad
M^\qvb := \left\{ m \in \bigoplus_{v \in \qvb} M_v \,\,\Big{|}\,\, (\ker \la) m = 0\right\}.$$
These functors are left exact, and they are also right exact when restricted to the full subcategory of
objects annihilated by $\ker \la$.  They are right adjoint to the functor
$(\D \otimes_U -)$ restricted to $\WM$ or $\Ulmod_\qvb$, respectively.  
Furthermore, the adjunction map $M \to (\D \otimes_U M)^\la$
is an isomorphism for any $M \in \WM$.

Fix a weight $v\in\qvb$.  For any sign vector $\a\in\{+,-\}^\Iv$, consider the simple $\D$-module
$$L_\a := \D\Big/\D \langle \del_i\mid \a(i) = +\rangle 
+ \D\langle x_i\mid \a(i) = -\rangle+\D\langle h^+_i-h^+_i(v)\mid i\notin \Iv\rangle.$$
It is easy to check that the isomorphism class of $L_\a$ does not depend on the choice of $v$ and that
\begin{equation}\label{simple-support}
\supp L_\a^\qvb = \DDelta_{\a}\cap\qvb.
\end{equation}

\begin{proposition}
\label{simple modules}
The modules $\{L_\a^\qvb\mid \a \in \cF_\qvb\}$ give a complete and irredundant set of representatives for the
isomorphism classes of simple objects of $\Ulmod_\qvb$.  
\end{proposition}

\begin{proof}
Since the adjoint action of $\bH$ on $U_\la$ is semisimple and
the resulting weight spaces are all cyclic left $\bH$-modules, 
the theory developed in \cite{MvdB} applies to $U_\la$.
The modules $\{L_\a\mid\a \in \{\plusminus\}^\Iv\}$ are precisely the simple weight modules for $\D$ 
supported in the lattice $\qvb + W_\Z$, 
and the results \cite[4.2.1, 4.3.1, \& 7.2.4]{MvdB}
imply that the isomorphism classes of simples in $\Ulmod_\qvb$ are given by the 
set of all nonzero $L_\a^\qvb$.
Since $L_\a^\qvb$ is a weight module, we have $L_\a^\qvb \ne 0$ if and only if
$\supp L_\a^\qvb \neq\emptyset$.
By Equation \eqref{simple-support}, this is exactly the condition that $\a\in\cF_\qvb$.
\end{proof}

\subsection{Quiver description of weight modules}\label{sec:quiver}
The results of Musson and Van den Bergh \cite{MvdB}
give an equivalence between $\Ulmod_\qvb$ and the category of finite
dimensional modules over a certain algebra,
which we now describe.

Let $Q$ be the path algebra over $\C$ of the quiver 
\[ \begin{tikzpicture}
\node (a) at (0,0){$(-)$};
\node (b) at (3,0){$(+)$};
\draw[->,thick] (a) to[out=20,in=160] (b);

\draw[<-,thick] (a) to[out=-20,in=-160] (b);
\end{tikzpicture}\]
with vertices labeled $+$ and $-$ and one 
arrow in each direction, 
and let $Q_n = Q \otimes_\C Q \otimes_\C \dots \otimes_\C Q$ be
the tensor product of $n$ copies of $Q$.  The algebra $Q_n$ is the path algebra of the
quiver whose vertices are labeled by $\{\plusminus\}^n$,
the set of vertices of an $n$-cube, with an edge connecting $\alpha$
to $\beta$ whenever $\alpha$ and $\beta$ differ in exactly one
position, modulo the relations that whenever $\alpha$ and $\gamma$
differ in exactly two positions, the two paths $\alpha \to \beta \to
\gamma$ and $\alpha \to \beta' \to \gamma$ are equal in $Q_n$.
Consider the grading on $Q_n$ for which a path of length $d$ has degree $d$, and  
let $\wh{Q}_n$ be the completion of $Q_n$ with
respect to the grading. 

For each $\alpha \in \{\plusminus\}^n$ and each $1 \le i \le n$, let
$\theta_{\a, i}$ be the element of ${Q}_n$ represented by the path $\a
\to \b \to \a$, where $\b$ agrees with $\a$ except in the $i^\text{th}$ place,
and put $\theta_i := \sum_{\a \in \{\plusminus\}^n} \theta_{\a,i}$.  
The center $Z(Q_n)$ is a polynomial algebra on the elements
$\{\theta_i\mid 1 \le i \le n\}$, and the center of $Z(\wh{Q}_n)$ is the completion of this
polynomial algebra with respect to the grading.  Let $\vart\colon \t \to Z(Q_n)$ denote
the linear map which sends $h_i$ to $\theta_i$.

Next, let $Q_\qvb \subset Q_n$ and $\wh{Q}_\qvb \subset \wh{Q}_n$ be the centralizers of all 
length $1$ paths $\a \to \b$ where $\a$ and $\b$ agree in every position 
except for $i \notin I_\qvb$.  If $\qvb$ is integral, we have $Q_\qvb = Q_n$ and 
$\wh{Q}_\qvb = \wh{Q}_n$.
Otherwise, $Q_\qvb$ is isomorphic to 
\[Q_{|I_\qvb|} \,\otimes_\C \,\C[\theta_i \mid i\notin I_\qvb],\] and 
$\wh{Q}_\qvb$ is its completion.
The primitive idempotents of $\wh{Q}_\qvb$ are indexed by $\a \in \{+,-\}^{I_\qvb}$.
The idempotent $e_\a$ corresponding to $\a$ is the sum of the primitive idempotents
of $\wh{Q}_n$ for all vertices in the fiber over $\a$ of the projection
$\{+,-\}^n \to \{+,-\}^{I_\qvb}$ forgetting the indices $i \notin I_\qvb$.  
Let $e_\qvb := \sum_{\a \in \cF_\qvb} e_\a$.

\begin{theorem}\label{first-equiv}
There is an equivalence of categories between $\Ulmod_\qvb$ and the category of 
finite dimensional modules over the ring $\left(e^{}_\qvb\, \wh{Q}_\qvb\, e^{}_\qvb\right)\Big{/}\big\langle \vart(x)e^{}_\qvb\mid x\in\fk\big\rangle$.
\end{theorem}

\begin{proof}
\cite[3.5.6 \& 6.3]{MvdB} give an equivalence between 
the category of weight modules for $\D$ with weights
in $\qvb + \SgrHZ$ and the category of finite dimensional
$\wh{Q}_\qvb$-modules.  The result now follows from Proposition \ref{simple modules} and
\cite[4.4.1]{MvdB}.  
\end{proof}

\begin{remark} In the proof of Theorem \ref{Translation and quivers} below we explain in more detail how
the equivalence of Theorem \ref{first-equiv} is constructed.
\end{remark}

\subsection{Translation functors}\label{translation functors}
Let $\la, \la' \colon Z(U) \to \C$ be two central characters, and let $\V, \V'$ be the corresponding
$V_0$-orbits in $\SH$.  Since $\SH$ is a torsor for $\SgrH$, the difference $\V - \V'$ is naturally
a $V_0$-orbit in $\SgrH$.
Let $$\D^{\la - \la'} := \displaystyle\bigoplus_{z\,\in\,\V - \V'} \D_z,$$ 
and consider the $(U_\la, U_{\la'})$-bimodule
\begin{align*}
 {}_\la T_{\la'} & := \D^{\la-\la'}\big/\D^{\la-\la'}\langle \ker \la'\rangle \\
 & = \D^{\la-\la'}\big/\langle \ker \la\rangle\D^{\la-\la'}\\
 & \cong \Hom_K(Y_\la, Y_{\la'}).
\end{align*}
Note that $ {}_\la T_{\la'}$ is nonzero if and only if
$\la - \la'$ is integral, meaning that $\V-\V'$ contains an element of $\SgrHZ$.
We have an associative collection of maps
\begin{equation}\label{bimodule map}
{}_\la T_{\la'} \otimes_{U_{\la'}} {}_{\la'}T_{\la''} \to {}_{\la}T_{\la''}
\end{equation}
given by compositions of homomorphisms.
If $\la = \la'$, then we have ${}_\la T_{\la'} \cong U_\la$ and Equation \eqref{bimodule map} is the obvious isomorphism.
The following proposition is proved in \cite[4.4.4]{MvdB}.

\begin{proposition}\label{trans-equiv}
Assume $\la-\la'$ is integral.
Then to any $\vb_0$-orbit $\qvb'\subset\V'$ there is a unique $\vb_0$-orbit $\qvb\subset\V$
such that $\qvb'-\qvb\subset W_\Z$.
The bimodule map  $${_{\la}T_{\la'}}\otimes_{U_{\la'}}{_{\la'}T_\la}\to {}_\la T_{\la}\cong U_\la$$ is an isomorphism
if and only if $\cF_{\qvb'} = \cF_\qvb$ for all $\vb_0$-orbits $\qvb'\subset \V'$.
\end{proposition}

\begin{remark}
The functor of tensoring with ${}_\la T_{\la'}$ can be considered a kind
of ``translation functor" on weight modules, and on the category $\cO$
which we will define in Section \ref{hyper O}.  The reader should
be warned, however, that it behaves quite differently from 
the usual translation functors on the BGG category $\cO$.  In particular, 
tensoring with ${}_\la T_{\la'}$ is only right exact.  (See 
Remark \ref{trick} below for more about the difference between
our framework and classical BGG category $\cO$.)
These functors will be important 
in Section \ref{sec:Localization} where we use them to
study the localization of $U_\la$-modules to a hypertoric variety, in 
Section \ref{sec:cells} where we define right cells in our category $\cO$, and in Section
\ref{pione}, where we use them to construct certain derived equivalences between different
blocks of our hypertoric category $\cO$.
\end{remark}

The following easy result is useful in describing the effect of tensoring
with ${}_\la T_{\la'}$ on weight modules.

\begin{lemma}\label{factoring translation} There is an equivalence of functors 
\[({}_\la T_{\la'} \otimes_{U_{\la'}} -) \cong (\D \otimes_U -)^\la  \colon U_{\la'}\mmod_{\lf} \to U_{\la}\mmod_{\lf}.\]
\end{lemma}

We have the following explicit formula for tensoring with ${}_\la T_{\la'}$ in terms of the quiver algebra
in Theorem \ref{first-equiv}.  
Let $\qvb\subset\V$ and $\qvb'\subset\V'$ be as in Proposition \ref{trans-equiv}, so that $\qvb'-\qvb\subset W_\Z$.
Let $R$ be the ring $\wh{Q}_\qvb\big{/}\langle \vart(x) \mid x\in\fk\rangle$, 
and put $e = e_{\qvb}$, $e' = e_{\qvb'}$
so that Theorem \ref{first-equiv} gives equivalences
\[F \colon \Ulmod_\qvb \to eRe\modfin \and 
F'\colon \Ulpmod_{\qvb'} \to e'Re'\modfin.\]
(Note that $\qvb'-\qvb\subset W_\Z$ implies that $\Iv = I_{\qvb'}$, so $\wh{Q}_\qvb = \wh{Q}_{\qvb'}$.)

\begin{theorem} \label{Translation and quivers}
The square 
\[\xymatrix{
\Ulmod_\qvb \ar[r]^{F}\ar[d]_{({}_{\la'}T^{}_\la \otimes_{U_\la} -)} & eRe\modfin \ar[d]^{( e'Re \,\otimes_{eRe}\, -)} \\
\Ulpmod_{\qvb'} \ar[r]^{F'}  & e'Re'\modfin
}\]
commutes up to natural isomorphism.
\end{theorem}
\begin{proof}
First we explain more carefully how the functors $F$ and $F'$ are
constructed using the methods of \cite{MvdB}. 
For each $\a\in\{+,-\}^{\Iv}$, choose an element $v_\a\in\qvb+W_\Z$ satisfying the inequalities
\[h^+_i(v_\a) \ge 0\text{ for all $i \in I_\qvb$ with }\a(i) = + \and h^-_i(v_\a) \le 0 \text{ for all $i\in I_\qvb$ with }\a(i) = -.\]
Note that these are exactly the inequalities that cut $\DDelta_\a$ out of $\V_\R$.
Thus when $\a\in\cF_\qvb$, we can and will assume that $v_\a\in\DDelta_\a\cap\qvb$.  

Recall from Section \ref{sec:weight} that $\fI_{v_\a}\subset\bH$ is the vanishing ideal of the point $v_\a\in\SH$.
Musson and van den Bergh \cite[6.3]{MvdB} give an isomorphism \[\varprojlim_k\, \End_\D\left(\bigoplus_{\a\in\{+,-\}^{\Iv}} \D/\D\fI^{k+1}_{v_\a}\right)^{\!\!\text{op}} \cong \,\,\,\,\wh{Q}_\qvb,\]
and we have an equivalence of categories
$$F_\D:\D\mmod_{\qvb + \SgrHZ} \to \wh{Q}_\qvb\modfin$$
given by
$$F_\D(M)\,\, := \,\,\,\varprojlim_k \Hom_\D\left(\bigoplus_{\a\in\{+,-\}^{\Iv}} \D/\D\fI^{k+1}_{v_\a}, \,\,\, M\right).$$
In particular, we have 
\[F_\D\left(\D/\D\fI^{k+1}_{v_\a}\right) \cong \left(\wh{Q}_\qvb/\langle \theta_i\rangle^{k+1}\right)e_\a.\]

Let $M$ be an object of $\D\mmod_{\qvb + \SgrHZ}$.  For any $m\in M$ and $v\in \qvb+\SgrHZ$, let $m_v$
denote the projection of $m$ onto $M_v$, so that $m = \sum m_v$.  For each $1 \le i \le n$, consider the endomorphism
$\Theta_i \in \End_\D(M)$ given by $\Theta_i\big(\sum m_v\big) := \sum \left(h_i^+ - h_i^+(v)\right)m_v$.
Using \cite[6.3]{MvdB} it is easy to check that $\Theta_i$ corresponds
under the equivalence $F_\D$ to multiplication by the central element $\theta_i \in \wh{Q}_\qvb$.
When $M = \D/\D\fI_{v_\a}^{k+1}$, $\Theta_i$ is given by \emph{right} 
multiplication by $\Theta_i(1) = h_i^+ - h_i^+(v_\a)$.
  
For each $\a\in\{+,-\}^{\Iv}$, 
let $\la_\a$ be the unique character of $Z(U)$ for which $\ker \la_\a$ is contained in $\fI_{v_\a}$
(equivalently, the character corresponding to the $V_0$-orbit $v_\a+V_0\subset\SH$), and put
\[\Pk_{\a} := \D \Big{/} \big{(}\D\fI_{v_\a}^{k+1} + \D\langle \ker \la_\a\rangle \big{)}.\]
This module is the quotient of $\D/\D\fI_{v_\a}^{k+1}$ by
the sum of the images of the operators
\[\sum_{i = 1}^n x_i\Theta_i, \;\, x \in \fk,\]
and since $F_{\D}$ is an equivalence, it follows that
\[F_\D(\Pk_\a) \cong \left( \wh{Q}_\qvb \big{/} \langle \theta_i\rangle^{k+1} + \langle \vart(x) \mid x\in\fk\rangle \right) e_\a.\]
This in turn implies that 
\[\varprojlim \End_\D\Big(\bigoplus_{\a \in \{+,-\}^{\Iv}}\Pk_\a\Big) \cong R,\] and the functor $F$ is given by 
$\varprojlim \Hom_{U_\la}\!\left(\bigoplus_{\a\in \cF_\qvb}(\Pk_\a)^\qvb, -\right)$.

The key property of the modules $\Pk_\a$ that we will need is the fact that
for any $\a \in \cF_\qvb$ we have
\begin{equation}\label{translation isomorphism}
\D \otimes_U (\Pk_\a)^\qvb \cong \D \otimes_U \left(U/U(\fI_{v_\a}^{k+1} + \langle \ker \la_\a\rangle)\right) \cong \Pk_\a,
\end{equation}
since $v_\a \in \DDelta_\a \cap \qvb$.
Note that \eqref{translation isomorphism} does \emph{not} hold for $\a \notin \cF_\qvb$.

Now let $M$ be an object of $\Ulmod_\qvb$.  Then if $\a \in \cF_{\qvb}$, we have
\[e_\a F(M) = \Hom_U\!\big((\Pk_\a)^\qvb, M\big) \cong \Hom_U\big((\Pk_\a)^\qvb, (\D \otimes_U M)^\qvb\big)
\cong \Hom_\D(\Pk_\a, \D \otimes_U M),\]
where the second isomorphism comes from \eqref{translation isomorphism} and adjunction. 
On the other hand, if $\b \in \cF_{\qvb'}$, we get
\[e_\b F'({}_{\la'}T_\la \otimes_{U_\la} M) \cong \Hom_U\big((\Pk_\b)^{\qvb'}, (\D \otimes_U M)^{\qvb'}\big)
\cong \Hom_\D(\Pk_\b, \D \otimes_U M).\]
The required natural transformation $\phi_M\colon e'Re \otimes_{eRe} F(M) \to F'({}_{\la'}T_\la \otimes_{U_\la} M)$
therefore comes from taking inverse limits of the composition
\[\Hom_\D(\Pk_\b, \Pk_\a) \otimes \Hom_\D(\Pk_\a, \D\otimes_U M) \to \Hom_\D(\Pk_\b, \D\otimes_U M).\]
Furthermore, it is clear from \eqref{translation isomorphism} 
that $\phi$ is an isomorphism on 
$(\Pk_\a)^\qvb$, $\a \in \cF_\qvb$.  This object is the projective cover of 
$L_\a^\qvb$ in $\Ulmod_\qvb^\ordk$, the full subcategory of $M \in \Ulmod_\qvb$
for which $\fI_{v}^{k+1} M_v = 0$ for all $v \in \qvb$.  This category has
enough projectives, and any $M \in \Ulmod_\qvb$ lies in $\Ulmod_\qvb^\ordk$ for 
some $k$, so we can find an exact sequence $P_1 \to P_0 \to M \to 0$ where
$\phi_{P_0}$ and $\phi_{P_1}$ are isomorphisms.
Since both the source and target of $\phi$ are right exact functors, $\phi_M$
is an isomorphism for all $M$.
\end{proof}

\section{Hypertoric category \texorpdfstring{$\cO$}{O}}\label{hyper O}
Fix a quantized polarized arrangement $\QPA = (\vb_0, \qvb, \xi)$.  In Section \ref{sec:ring} we explained
how $\vb_0$ determines an algebra $U$, how $\V = \qvb + \C\vb_0\subset\SH$ determines a central character
$\la:Z(U)\to\C$, and how the lattice $\qvb\subset\V$ determines a subcategory $\Ulmod_\qvb$ of the category
of weight modules over $U_\la$.  In this section we use the parameter $\xi\in\vb_0^*$ to restrict our categories
even further, and thus obtain the hypertoric category $\cO$.

\subsection{Definition of the category}\label{o-def}
Recall that $U$ has the 
decomposition $U = \displaystyle\bigoplus_{z \in \Lambda_0} U_z$, where
$U_z = \D_z$ is the $z$-isotypic piece of $\D$.
For any $k\in\Z$, put
$$U^k := \bigoplus_{\xi(z) =k} U_{z}.$$
Then put
$$U^+ := \bigoplus_{k\geq 0} U^k
\and
U^- := \bigoplus_{k\leq 0} U^k.$$
The algebra $U^+$ will play a role for us similar to 
the role played by the enveloping algebra
of a Borel in the definition BGG category $\cO$.
Note that the analogy is not exact; in particular, we
have a surjection
\[U^+ \otimes_{U^0} U^{-} \to U,\]
but it is not an isomorphism.

These subalgebras and subspaces  induce corresponding subalgebras 
and subspaces of the 
central quotient $U_\la$: we
let $U_{\la, z}$, $U_\la^k$, $U_\la^+$, and $U_\la^-$ be
the images of $U_z$, $U^k$, $U^+$, and $U^-$, respectively, under the
quotient map $U \to U_\la$.

\begin{definition}\label{defn of O}
We define {\bf hypertoric category \boldmath$\cO$} to be
the full subcategory of $U\mmod$ consisting of modules that are $U^+$-locally finite and 
semisimple over the center $Z(U)$.
We define $\cO_\la$ to be the full subcategory of $\cO$ consisting of modules on which $U$ acts with
central character $\la$; equivalently, it is the full subcategory of $\Ulmod$ consisting of modules that are $U_\la^+$-locally finite.
We define $\cO(\QPA)$ to be the full subcategory of $\cO_\la$ consisting of modules supported in $\qvb$;
equivalently, it is the full subcategory of $\Ulmod_\qvb$ consisting of modules that are $U_\la^+$-locally finite.
\end{definition}

We have a direct sum decomposition $$\cO = \bigoplus_{\qvb \in \SH/\vb_0} \cO(\vb_0, \qvb, \xi).$$
It follows from Theorem \ref{alg=comb} below that these summands of $\cO$ are blocks, that is, they
are the smallest possible direct summands.
In accordance with the terminology in Lie theory, we will call $$\cO_\la := \bigoplus_{\qvb \in \V/\vb_0} \cO(\vb_0, \qvb, \xi)$$
an {\bf infinitesimal block} of $\cO$.
We will call the block $\cO(\QPA)$ {\bf regular} (respectively {\bf integral}) if $\QPA$
is regular (respectively integral) as defined in Section \ref{sec:qpd}; we call an infinitesimal block
regular if its constituent blocks are all regular.

\begin{remark}\label{inf-block=block}
If $\vb_0$ is unimodular (Remark \ref{rmk:unimodular}) and $\QPA$ is integral, then $I_{\qvb'}=\emptyset$
for all $I_{\qvb'}\subset\V$ different from $\qvb$, and therefore $\cO(\QPA) = \cO_\la$. 
If $\vb_0$ is not unimodular, however, then the regular infinitesimal blocks of $\cO$ are never themselves blocks.
\end{remark}

\begin{remark}\label{trick}
We warn the reader of a subtle but important
difference between hypertoric category $\cO$ and the classical
BGG category $\cO$.  Let $\fg$ be a semisimple Lie algebra
with $\fh\subset\fb\subset\fg$ a Cartan and Borel subalgebra.
BGG category $\cO$ is the full subcategory of finitely generated
$U(\fg)$-modules for which 
$U(\fb)$ acts locally finitely and $U(\fh)$ acts semisimply, while
hypertoric category $\cO$ is the full subcategory of finitely
generated $U$-modules for which $U^+$ acts locally finitely and $Z(U)$ acts semisimply.
The analogy is imprecise because $U(\fh)$ is not the center of $U(\fg)$.

In the case of a regular infinitesimal block, we are rescued by a theorem of Soergel:
a regular infinitesimal block of BGG category $\cO$ is equivalent
to a regular infinitesimal block of the category obtained by requiring the center of $U(\fg)$, rather than the Cartan $U(\fh)$,
to act semisimply \cite{Soe86}.  Thus, it is reasonable to regard regular infinitesimal blocks (or blocks)
of hypertoric category $\cO$ as analogues of regular infinitesimal blocks (or blocks) of BGG category $\cO$.  
Furthermore, Theorems \ref{first-intro-thm} and \ref{second-intro-thm} demonstrate
that regular integral blocks of BGG category $\cO$ and hypertoric category $\cO$ have many properties in common.
\end{remark}

Consider the natural projection $\SgrHZ^*\to\vb_0^*$, and choose any lift $\tilde\xi\in\SgrHZ^*$ of $\xi$.  
Recall from Section \ref{sec:weyl} that we have identified $\SgrH$ with $\t^*$
and $\t$ with the degree 2 part of $H$, thus we may regard $\tilde\xi$ as a linear combination of $\{h_1,\ldots,h_n\}$.
Lift it further to an element $\hat\xi\in F_2\bH$, that is, to a linear combination of $\{h_1^\pm,\ldots,h_n^\pm\}$.  Note that,
by Equation \eqref{moment map}, the subspace $U^k\subset U$ is exactly
the space on which the conjugation operator $\operatorname{ad}(\hat\xi)$ acts with eigenvalue $k$.

\begin{lemma}\label{gamma-finite}
Let $M$ be a finitely generated $U_\la$-module.  Then the following are equivalent:
\begin{enumerate}
\item $M$ is an object of $\cO$.
\item $M$ is generated by a finite-dimensional $U^+$-invariant subspace $S$.
\item $\hat\xi\in\bH\subset U^0$ acts locally finitely on $M$ with finite
dimensional generalized eigenspaces and the eigenvalues that appear are bounded above.
\end{enumerate} 
\end{lemma}
\begin{proof}
\noindent $(1)\Rightarrow(2)$:
Given a finite generating set for $M$, we can apply $U^+$ to obtain a finite-dimensional $U^+$-invariant
generating set.

\noindent $(2)\Rightarrow(3)$:
Suppose that $S$ is a finite-dimensional $U^+$-invariant generating set for $M$.   Choose finitely many  
algebra generators $u_1, \dots, u_\ell$ for $U$ which are $\bT$-weight vectors,
and let $k_i\in\Z$ be the weight of $u_i$.  Let $U'\subset U$ be the subalgebra generated by those $u_i$ for which $k_i<0$.
Then we have $U'U^+ = U$, since this holds after taking the associated graded.  
This in turn implies that $M=U'S$.  Since $U' \cap U^k$ is finite dimensional for all $k$ and is zero for $k > 0$, 
$M$ has finite dimensional $\hat{\xi}$-weight spaces, all of which have no higher weight than the highest occurring in $S$.

\noindent $(3)\Rightarrow(1)$: If the generalized $\hat \xi$ eigenspaces of $M$ are both finite dimensional and bounded above, then for all $m\in M$, $U^+m$ is a subspace of finitely many generalized eigenspaces and thus finite dimensional.
\end{proof}

Recall from Section \ref{sec:bf} that $\cF_\qvb$ is defined to be the set of all $\a\in\{+,-\}^{\Iv}$ such that $\DDelta_\a\cap\qvb$
is nonempty, and $\cP_{\qvb,\xi}\subset\cF_\qvb$ is the subset for which $\xi$ is proper and bounded above on $\DDelta_\a\cap\qvb$.
Thus Lemma \ref{gamma-finite} has the following immediate corollary.

\begin{corollary}\label{simples}
Suppose that $\a \in \cF_\vb$, so $L_\a^\qvb$ is
a simple object of $\Ulmod_\qvb$.  Then $L_\a^\qvb \in \cO(\QPA)$
if and only if $\a\in\cP_{\qvb,\xi}$.
\end{corollary}

\begin{remark}\label{number of simples}
Corollary \ref{simples} tells us that the block $\cO(\QPA)$ has $|\cP_{\qvb,\xi}|$ isomorphism classes of simple objects.
This number is always less than or equal to the number of bases for the matroid associated to $\cH_0$,
or equivalently the number of vertices of $\cH_\eta$ for a regular value of $\eta$ (Remark \ref{bijection}).
Equality is achieved if and only if $\QPA$ is both regular and integral.  If $\QPA$ fails to be integral, then there
will be too few hyperplanes in $\qH$, and therefore too few chambers $\DDelta_\a$.  If $\QPA$ is integral but fails
to be quasi-regular, then there will again be too few chambers.  If $\QPA$ is integral and quasi-regular
but not regular, then we will have the right number of chambers, but some of them will not contain any points
of the lattice $\qvb$.
\end{remark}

\subsection{Quiver description of \texorpdfstring{$\cO(\QPA)$}{O(X)}}\label{sec:quiver-description}
Theorem \ref{first-equiv} and Corollary \ref{simples} combine to give us the following result.
Let \[e_\xi\,\,\,\, := \sum_{\alpha \notin \cB_{\qvb,\xi}}\!\!\! e_\a\,\, .\]

\begin{theorem}\label{alg=comb}
The category $\cO(\QPA)$ is equivalent to the category of finite dimensional 
modules over
\[A(\QPA) := \left(e^{}_\qvb\, \wh{Q}_\qvb\, e^{}_\qvb\right)\Big{/}\big\langle e_\xi e_\qvb\big\rangle+
            \big\langle \vart(x)e^{}_\qvb\mid x\in\fk\cap \t^\Iv\big\rangle.\]
\end{theorem}

In \cite[3.1]{GDKD} we constructed an algebra $A(\PA)$ from a
regular polarized arrangement $\PA$.

\begin{theorem} \label{blpw}
If $\QPA$ is integral and regular and linked to $\PA$, then $A(\QPA) \cong A(\PA)$.
\end{theorem}

\begin{proof}
The linkage of of $\PA$ and $\QPA$ implies that $\cP_{\qvb,\xi} = \cP_{\eta,\xi}$,
and knowing this, it is straightforward to check that the algebra
defined in \cite[3.1]{GDKD} is equal to the ``polynomial part''
of $A(\QPA)$, namely the uncompleted algebra
$$\left(e^{}_\qvb\, Q_{\qvb}\, e^{}_\qvb\right)\Big{/}\big\langle e_\xi e_\qvb\big\rangle + \big\langle \vart(x)e^{}_\qvb\mid x\in\fk\big\rangle.$$
But \cite[4.14]{GDKD} implies that this algebra is finite dimensional,
so it is isomorphic to its completion.
\end{proof}

This result implies that if $\QPA$ is
a regular integral quantized polarized arrangement
then, up to equivalence,
the category $\cO(\QPA)$ only depends on
the equivalence class of $\QPA$ (Definition \ref{def:equivalence}).  More precisely, if $\QPA$ and $\QPA'$ are regular integral polarized arrangements in the same equivalence class, then
Theorem \ref{Translation and quivers} implies that 
 tensoring with a bimodule ${}_{\la'}T_\la$ gives an equivalence of categories from
 $\cO(\QPA)$ to $\cO(\QPA')$.  

When the quantized polarized arrangement $\QPA$ is not integral, there are two
possibilities.  First, the hyperplane arrangement $\qH$ associated to $\QPA$ could be inessential (Section \ref{sec:qpd}).  In this case the arrangement $\qH$ has no vertices; this happens, for example, when $\Iv=\emptyset$.  
For any $\a\in\cP_{\qvb,\xi}$ the polyhedron $\DDelta_\a$ has a vertex of $\qH$ as a $\xi$-maximal point.
Thus, when $\qH$ is inessential, $\cP_{\qvb,\xi}$ must be empty, 
so there are no nonzero objects in $\cO(\QPA)$.

Suppose on the other hand that $\qH$ is essential.
Then we can define an integral quantized polarized arrangement 
$\QPA' = (\vb_0',\qvb, \xi')$ with $n' := |\Iv|$ hyperplanes 
along with an isomorphism $\V_\R\cong\V'_\R$ that takes $\qH$ to $\qH'$.
To define $\QPA'$, let
$\SgrH'$ and $\SH'$ be the vector space and affine space defined in 
Section \ref{sec:polarr}, but with $n$ replaced by $n'$.
More precisely, they are the spectra of the subrings of $H$ and $\bH$
generated by $h_i^{\pm}$ and $h_i$ for all $i \in I_\qvb$.
Then let $\vb_0'$ and $\qvb'$ be the images of $\vb_0$ and $\qvb$ under the natural
projections $\SgrH\to\SgrH'$ and $\SH\to\SH'$.
The condition that $\qH$ has a vertex implies that the projections
induce isomorphisms $\vb_0' \cong \vb_0$ and $\qvb'\cong \qvb$.
Letting $\xi'$ be the composition of $\xi$ with the
isomorphism $\vb_0' \cong \vb_0$, we obtain equalities
$$\cF_{\qvb'} = \cF_\qvb, \qquad\cB_{\qvb',\xi'}=\cB_{\qvb,\xi},\and\cP_{\qvb',\xi'}=\cP_{\qvb,\xi}$$
of subsets of $\{\plusminus\}^\Iv$.
The following result now allows us to reduce the study of arbitrary blocks of hypertoric category
$\cO$ to the integral case.

\begin{theorem}\label{reduced-thm}
If $\qH$ is essential, then $\cO(\QPA)$ is equivalent to $\cO(\QPA')$ and $A(\QPA)\cong A(\QPA')$.  
\end{theorem}

\begin{proof} As noted in Section \ref{sec:quiver}, there is an isomorphism
$\wh{Q}_\qvb \cong \wh{Q}_{|I_\qvb|}\,\wh\otimes_\C\, \C[[\theta_i \mid i\notin I_\qvb]]$, and
the map $\vart$ gives an isomorphism from $\t \cong \C^n$ to the degree two part
of $Z(\wh{Q}_\qvb)$.  Let $\t^{I_\qvb} \subset \t$ be the space of vectors whose $i^\mathrm{th}$
coordinates vanish for all $i\notin I_\qvb$, so $\vart(\t^{I_\qvb}) = Z(\wh{Q}_{|I_\qvb|})_2 \otimes \C$.
Then the assumption that $\qH$ is essential implies that the composition
\[\fk \hookrightarrow \t \to \t/\t^{I_\qvb}\]
is a surjection.  Thus \[\wh{Q}_\qvb/\langle \vart(\fk)\rangle \cong 
\wh{Q}_{|I_\qvb|}/\langle \vart'(\fk \cap \t^{I_\qvb})\rangle,\]
where $\vart'\colon \t^{I_\qvb} \to Z(\wh{Q}_{|I_\qvb|})$ is the induced map.
The isomorphism $A(\QPA)\cong A(\QPA')$ follows, since $\fk\cap \t^{I_\qvb}$ is the annihilator of $\vb_0'$
under the natural isomorphism $\t^{I_\qvb} \cong (\SgrH')^*$.  The equivalence of $\cO(\QPA)$ and $\cO(\QPA')$
now follows from Theorem \ref{alg=comb}.
\end{proof}

As a corollary we obtain the first two parts of Theorem \ref{second-intro-thm}, along with the beginning of the seventh part.

\begin{corollary}\label{qhk}
If $\QPA$ is regular, then the category $\cO(\QPA)$ is highest weight and Koszul.
If $\QPA$ and $\QPA^!$ are regular, integral, and Gale dual to each other, then the rings $A(\QPA)$ and $A(\QPA^!)$
are Koszul dual.
\end{corollary}

\begin{proof}
These facts follow from Theorems \ref{blpw} and \ref{reduced-thm} and \cite[3.11, 5.23, \& 5.24]{GDKD}.
\end{proof}

\begin{example}\label{big example}
We illustrate the preceding results with a simple example.  Let $K \subset T$ be
the 1-dimensional torus acting diagonally on $\C^n$, so the corresponding $(n-1)$-dimensional lattice
$\vb_0\subset\SgrHZ$ is the vanishing locus of $h_1+\ldots+h_n$.
For $c\in \Z$, let $\qvb_c := \{v \in \SHZ \mid \sum_{i=1}^n h^+_i(v) = c\}$;
it is regular if and only if $c \ge 0$ or $c \le -n$.
The quantized polarized arrangement $\qH$ consists of $n$ pairs of parallel hyperplanes
in a vector space of dimension $n-1$, with each pair in general position with respect to each other pair.
Take any covector $\xi \in \vb_0^*$ which is
positive on the weight of $x_i\partial_j$ if and only if $i < j$. 
These weights are exactly the directions of the $1$-dimensional flats, so $\xi$ is
regular.  Let $\QPA_c = (\vb_0, \qvb_c, \xi)$.
An example with $n=3$ is illustrated in Figure \ref{quant-fig}.  We warn the reader that
even when $\qvb_c$ is regular it is possible
to obtain a picture in which three hyperplanes pass through a single point (see Definition \ref{def:quasi-reg}).
In fact, when $c=0$, the $n$ hyperplanes $\{H_1^+,\ldots,H_n^+\}$ all pass through a single point.

Let us examine the case $c = 0$ in more detail.  We have
\[\cP_{\qvb_0,\xi} = \{(+,+,\dots, +), (-,+, \dots, +), (-, -, +, \dots, +), \dots,  (-, -, \dots, -, +)\}.\]
The algebra $A(\QPA_0)$ is the quotient of the path algebra of the quiver 
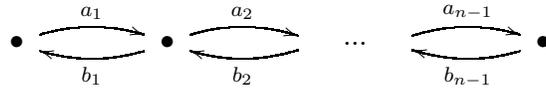
\[\xy
(-35,0)*++{\bullet}="1"; (-15,0)*++{\bullet}="2";(5,0)*++{\;}="3l";(10,0)*++{...}="3";(15,0)*++{\;}="3r";(35, 0)*++{\bullet}="4";
{\ar@/^/^{a_1} "1";"2"}; {\ar@/^/^{b_1} "2";"1"};
{\ar@/^/^{a_2} "2";"3l"}; {\ar@/^/^{b_2} "3l";"2"};
{\ar@/^/^{a_{n-1}} "3r";"4"}; {\ar@/^/^{b_{n-1}} "4";"3r"};
\endxy\]
by the relations $b_1a_1 = 0$, and $a_{i+1}a_i = 0$, $b_ib_{i+1} = 0$ and $a_ib_i = b_{i+1}a_{i+1}$ for all $1 \le i \le n-2$.
The corresponding block $\cO(\QPA_0)$ of our hypertoric category $\cO$ is equivalent to a parabolic block of
classical BGG category $\cO$ for $\fg := \mathfrak{gl}_n$ in the following way.  
There is a surjective homomorphism $U(\fg) \to U$  
which takes the elementary matrix $E_{ij}$ to $x_i\partial_j$ and sends $U(\fb)$ to $U^+$.
Pulling back by this homomorphism is a full and faithful
functor from hypertoric category $\cO$ into the category $\cO'(\fg)$ of
$U(\fg)$-modules on which $U(\fb)$ acts locally finitely and
$Z(U(\fg))$ acts semisimply.  The image of this functor is the subcategory generated by the simple modules with highest weights
$w \cdot 0 = w(\rho) - \rho$, where $\rho$ is half the sum of the positive roots of $\fg$ and 
$w$ runs over all smallest representatives in $S_n$ for the left cosets $S_n/(1\times S_{n-1})$.

The equivalence of \cite{Soe86} gives an equivalence between this category and the subcategory
of BGG category $\cO(\fg)$ generated by simple modules with highest weights $w^{-1}\cdot 0$, where
$w$ runs over the same set; this subcategory is a block of the parabolic category $\cO^{\mathfrak{p}}(\fg)$ consisting of $\mathfrak{p}$-locally finite modules in $\cO(\fg)$,
where $\mathfrak{p} \subset \fg$ is the set of block 
matrices preserving the first basis vector of $\C^n$.
\end{example}

\begin{example}\label{next example}
We now give an example that shows how Corollary \ref{qhk} can fail when $\QPA$ is not regular.
Let $K\subset T$ be the $(n-1)$-dimensional subtorus that acts on $\C^n$ with determinant 1,
so that $\vb_0 \subset\SgrHZ$ is the 1-dimensional lattice cut out by the equations $h_i=h_j$ for all $i,j\leq n$.
For appropriate choices of regular $\qvb$ and $\xi$, the resulting 
quantized polarized arrangement $\QPA = (\vb_0,\qvb, \xi)$ will be Gale dual to the one from 
example \ref{big example}.
On the other hand, if either $\qvb$ or $\xi$ is not regular, then $\cO(\QPA)$ can fail to be highest weight and Koszul.

Consider first the example where $\qvb = \{v \in \SHZ \mid h^+_i(v) = h^+_j(v)\;\text{for all $i, j\leq n$}\}$
(which is not regular) and $\xi\ne 0$.
Then $\cF_\qvb = \{(-, -, \dots, -), (+, +, \dots, +)\}$
and $\cP_{\qvb, \xi}$ contains exactly one of these sign vectors (which one depends on our choice of $\xi$.) 
Theorem \ref{first-equiv} tells us that
$A(\QPA)\cong \C[\theta]/\langle \theta^n\rangle$, which is not quasihereditary when
$n > 1$ and is not Koszul when $n > 2$.

On the other hand, suppose that $\qvb$ is integral and regular, 
so $\cF_\qvb$ contains $n+1$ sign vectors.
If $\xi = 0$, then $\cP_{\qvb,\xi}$ contains exactly those $n-1$ sign vectors 
whose associated polyhedra are compact.  
The resulting algebra
$A(\QPA)$ is the quotient of the path algebra of the quiver
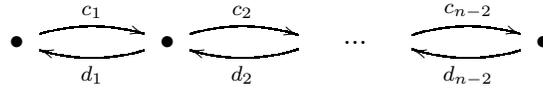
\[\xy
(-35,0)*++{\bullet}="1"; (-15,0)*++{\bullet}="2";(5,0)*++{\;}="3l";(10,0)*++{...}="3";(15,0)*++{\;}="3r";(35, 0)*++{\bullet}="4";
{\ar@/^/^{c_1} "1";"2"}; {\ar@/^/^{d_1} "2";"1"};
{\ar@/^/^{c_2} "2";"3l"}; {\ar@/^/^{d_2} "3l";"2"};
{\ar@/^/^{c_{n-2}} "3r";"4"}; {\ar@/^/^{d_{n-2}} "4";"3r"};
\endxy\]
by the relations $d_1c_1 = 0$, $c_{n-2}d_{n-2} = 0$, and $c_id_i = d_{i+1}c_{i+1}$ for $1 \le i \le n-3$.
This algebra is not quasihereditary when
$n > 1$ and is not Koszul when $n > 2$.
\end{example}

\subsection{Projective objects in \texorpdfstring{$\cO(\QPA)$}{O(X)} and \texorpdfstring{$\xi$}{xi}-truncation}\label{truncation}
Let $\QPA = (\vb_0, \qvb, \xi)$ be a quantized polarized arrangement, and let
$\la$ be the corresponding character of $Z(U)$.  The inclusion functor
$\iota\colon \cO(\QPA) \to \Ulmod_\qvb$ has a left adjoint $\pi_\xi$, which we
call $\xi$-truncation.  To define it, take $M \in \WM$ 
and let $\pi_\xi(M)$ be the quotient of
$M$ by the submodule generated by all weight spaces $M_v$ with 
$v \notin \bigcup_{\a \in \cP_{\qvb,\xi}} \DDelta_\a$.  
It is evident that this functor is right exact, and that
$\pi_\xi \circ \iota$ is the identity functor on $\cO(\QPA)$.

Equivalently, $\pi_\xi(M)$ is the largest quotient of $M$ which 
is $U^+$-locally finite.  This leads immediately to 
a description of truncation in terms of the equivalent
quiver categories.
We use the notation of Theorems \ref{Translation and quivers} and \ref{alg=comb},
so $e = e_\qvb$, and $A(\QPA)$ is the $e R e$ algebra $eRe/ eR e_\xi e R e$.  
\begin{proposition}
The diagram
\[\xymatrix{
\Ulmod_\qvb \ar[r] \ar[d]_{\pi_\xi} & eRe\modfin \ar[d]^{( A(\QPA) \,\otimes_{eRe}\, -)} \\
\cO(\QPA) \ar[r]  & A(\QPA)\mmod
}\]
commutes up to natural equivalence of functors, where the horizontal functors are
the equivalences of Theorems \ref{first-equiv} and  \ref{alg=comb}.
\end{proposition}

We use the $\xi$-truncation functor to construct the projective objects in our category as follows.
Take $\a \in \cP_{\qvb,\xi}$, and recall the $U_\la$-module $(\Pk_\a)^\qvb$ from the
proof of Theorem \ref{Translation and quivers}.  This module, which is projective
in the subcategory $\Ulmod_\qvb^\ordk \subset \Ulmod_\qvb$, represents
the functor $M \mapsto M_v$ for any $v \in \DDelta_\a \cap \qvb$.  It
follows that for $k$ large enough (say $k > \dim_\C A(\QPA)$), the
truncated module $\pi_\xi (\Pk_\a)^\qvb$ is independent of $k$, and represents
the $v$-weight space functor on $\cO(\QPA)$.  It is thus the projective
cover of $L_\a^\qvb$.

\subsection{Standard modules of \texorpdfstring{$\cO(\QPA)$}{O(X)}}\label{sec:std modules}
Let us assume that $\QPA = (\vb_0,\qvb,\xi)$ is regular, so that by Corollary \ref{qhk} the category $\cO(\QPA)$ is highest weight.   Being highest weight means that $\cO(\QPA)$ contains certain objects called {\bf standard modules},
which are analogues of the Verma modules of BGG category $\cO$. 
The definition of standard modules involves a partial order on the set $\cP_{\qvb, \xi}$ that indexes
the simple objects of $\cO(\QPA)$; this partial order was introduced in \cite[\S 2.6]{GDKD}.
For $\a\in\cP_{\qvb,\xi}$, the standard module $S^\qvb_\a$ is defined to be the projective cover
of $L_\a^\qvb$ in the full subcategory $\cO(\QPA)_{\le \a} \subset \cO(\QPA)$ 
consisting of modules whose composition factors $L^\qvb_\beta$ satisfy $\b\leq \a$.
In this section we construct the standard modules explicitly.

Fix an element $\a \in \cP_{\qvb,\xi}$, and let $v_\xi$ be the $\xi$-maximal point of 
$\DDelta_{\a}\cap\qvb$.  Consider the set $$Q := \{h_i\mid v_\xi\in H_{i}^+\} \cup \{-h_i\mid v_\xi\in H_{i}^-\}.$$
By quasi-regularity, there exists a unique subset $Q_0\subset Q$ of order $\dim V$ such that $\xi$ is 
equal to the restriction to $\Lambda_0$ of a negative linear combination
of the elements of $Q_0$.
Fix any weight $v\in\DDelta_{\a}\cap\qvb$, 
and let $J_\a\subset\D$ be the left ideal generated by the following elements:\footnote{In the published
version of this paper, the first and second items below had $Q_0$ replaced with $Q$.  This was a mistake,
as Propositions \ref{KL} and \ref{projective cover} did not hold in full generality with the old definition.}
\medskip
\begin{itemize}
\item $\partial_i$ for all $i$ such that $h_{i}\in Q_0$
\item $x_i$ for all $i$ such that $-h_{i}\in Q_0$
\item $h_i^{\a(i)}$ for all $i\in\Iv$
\item $h_i^+-h_i^+(v)$ for all $i\notin\Iv$.
\end{itemize}
Let $S_\a := \D/J_\a$.  This module lies in $\D\mmod_{\qvb + \SgrHZ}$, 
and there is a natural surjection $S_\a \to L_\a$ of $\D$-modules.
Each simple object $L_\b$ in $\D\mmod_{\qvb + \SgrHZ}$ appears in the composition series of $S_\a$
with multiplicity either $0$ or $1$, with the multiplicity being $1$ 
if and only if $\b(i) = +$ whenever $h_i \in Q_0$ and $\b(i) = -$ whenever $-h_i \in Q_0$ (in particular, $\b(i) = \a(i)$ for these values of $i$).  
Since all the weight spaces are $1$-dimensional, $S_\a$ is semisimple over $Z(U)$, 
and so  $S_\a^\qvb = \bigoplus_{v \in \qvb} (S_\a)_v$.  It follows that
that the simple object $L_\b^\qvb$ in $U_\la\mmod_\qvb$ 
appears in the composition series of $S_\a^\qvb$ if and only if $\b\in\cF_\qvb$
and $\b(i) = \a(i)$ for all $i$ such that $h_i$ or $-h_i$ is in $Q_0$.  
Corollary \ref{simples} then shows that $S_\a^\qvb\in\cO(\QPA)$, and from \cite[2.11]{GDKD} we get
that $S_\a^\qvb\in\cO(\QPA)_{\le \a}$.  

\begin{remark} The definition of $S^\qvb_\a$ is made more complicated by the need to address two possibilities: first, the point $v_\xi$ might not lie in the lattice $\qvb$, and second, $v_\xi$ might lie on more than $\dim V$ of the hyperplanes $H_i^\pm$.  The first possibility will not occur if the arrangement is unimodular.  In that case we can take $v = v_\xi$ and the 
third and fourth sets of generators of $J_\a$ can be replaced by 
$h_i^+ - h_i^+(v)$ for all $i$.
The second issue will not occur if the set $\DDelta_\a$ has full dimension $\dim V$, since then regularity implies that it is a simple polyhedron and the linear forms in $Q$ cut out its facets which contain $v_\xi$, 
so $Q = Q_0$. 
\end{remark}

For use in Section \ref{sec:grothendieck}, it will be convenient if we restate these observations 
for integral $\QPA$ in terms of a polarized
arrangement $\PA$ linked to $\QPA$.

\begin{proposition}\label{KL}
Suppose that $\QPA$ is integral and linked to $\PA$.  Choose any $\a\in\cP_{\qvb,\xi}$, and let $\Sigma_\a\subset V_\R$
be the unique cone that coincides with $\Delta_\a$ in a neighborhood of the $\xi$-maximal point of $\Delta_\a$.
The simple module $L_\b^\qvb$ appears in the composition series of $S_\a^\qvb$ if and only if 
$\Delta_\b\subset\Sigma_\a$, in which case it appears exactly once.
\end{proposition}

\begin{proposition}\label{projective cover} 
The module $S_\a^{\qvb}$ is a projective cover of $L_\a^\qvb$ in the category $\cO(\QPA)_{\le \a}$.
\end{proposition}

\begin{proof}
The $v$-weight space of $S_\a$ is $1$-dimensional,
and $S_\a^\qvb$ is a cyclic $U_\la$-module generated by any nonzero element $x$ of this weight space.
In particular, $S_\a^\qvb$ is indecomposable, and it remains only to show that it is projective in $\cO(\QPA)_{\le \a}$.

To do this, we will show that for any object $M$ of $\cO(\QPA)_{\le \a}$, the map
\begin{equation*}\label{restrict to weight space}
\Hom_{U_\la}(S_\a^\qvb, M) \to M_v, \ \ \phi \mapsto \phi(x)
\end{equation*}
is surjective.  Since this map is obviously injective, it will follow
that $S_\a^\qvb$ represents the exact functor $M \mapsto M_v$, and therefore
that $S_\a^\qvb$ is projective in $\cO(\QPA)_{\le \a}$.

To see surjectivity, let $\ol{M} := \D \otimes_U M$,
and consider the analogous map for $\D$-modules
\begin{equation*}\label{D map}
\Hom_\D(S_\a, \ol{M}) \to \ol{M}_v, \ \ \psi \mapsto \psi(x).
\end{equation*}
It suffices to show that this map is surjective since it factors
$$\Hom_\D(S_\a, \ol{M})\overset{(\cdot)^\qvb}{\longrightarrow}
\Hom_{U_\la}(S_\a^\qvb, M) \to M_v\cong \ol{M}_v,$$
and $(\cdot)^\qvb$ is surjective.

Let $v_\a$ be the weight of the element $1\in S_\a$; the associated
ideal $\mathfrak{I}_{v_\a} \subset \bH$ is spanned by 
$\{h^{\a(i)}_i\mid i\in\Iv\}\cup\{h_i^+ - h_i^+(v)\mid i\notin\Iv\}$.
Let $m$ be a monomial which generates $\D_{v-v_\a}$ as a left or right $\bH$-module.
For any $\D$-module $N$, multiplication by $m$ gives an isomorphism $N_{v_\a} \to N_{v}$.
Thus it is sufficient to show that the map
$$\Hom_\D(S_\a, \ol{M})\to \ol{M}_{v_\a}$$ taking $\psi$ to $\psi(1)$ is surjective.
In other words, we must show that the defining ideal $J_\a$ of $S_\a$
annihilates $\ol{M}_{v_\a}$.

The ideal $J_\a$ has four types of generators; we treat them one at a time.
If $v_\xi\in H_{i}^+$, then $\partial_i$ takes $\ol{M}_{v_\a}$ to $\ol{M}_{v_\b}$,
where $\b$ differs from $\a$ only in the $i^\text{th}$ coordinate.  But such a $\b$ cannot
satisfy $\b\leq\a$, thus $M \in \cO(\QPA)_{\le \a}$ implies that $\ol{M}_{v_\b}=0$.
Hence $\partial_i$ annihilates $\ol{M}_{v_\a}$.  By similar reasoning,
if $v_\xi\in H_{i}^-$, then $x_i$ annihilates $\ol{M}_{v_\a}$.
Now consider the third and fourth set of generators in the list.
These lie in the ideal
$\mathfrak{I}_{v_\a}$, and we know (by definition) that some power of $\mathfrak{I}_{v_\a}$ annihilates
$\ol{M}_{v_\a}$.  We only need to show that the first power does the trick.

Let $\la_\a\colon Z(U) \to \C$ be the character for which $\ker \la_\a \subset \fI_{v_\a}$.
Since $M$ is semisimple over $Z(U)$, so is $\ol{M}$, and so $(\ker \la_\a)\ol{M}_{v_\a} = 0$.
Furthermore, if $v_\xi\in H_{i}^{\a(i)}$, then $h^{\a(i)}_i$ annihilates $\ol{M}_{v_\a}$ 
by our analysis of the action of the first two types of generators of $J_\a$.  
Now we observe that $\mathfrak{I}_{v_\a}$ is generated by 
$$\{h^{\a(i)}_i\mid v_\xi\in H_{i}^{\a(i)}\}\cup \ker \la_\a,$$
and we are done.
\end{proof}

\subsection{Deformed category \texorpdfstring{$\cO$}{O(X)}}\label{sec:Deformed O}
In \cite{kosdef} we described a universal flat graded deformation of any 
Koszul graded algebra, and we studied its properties in the special case
when the original module category and its Koszul dual are both highest weight.
Let $\QPA = (\vb_0,\qvb,\xi)$ be a regular 
quantized polarized arrangement, and for simplicity assume that it is integral.  We now
explain how to interpret modules over the universal deformation of the 
Koszul algebra $A(\QPA)$ as modules over the hypertoric enveloping 
algebra $U$.  The results of this section will not be used later
in the paper.

Let $\tilde{A}(\QPA)$ be the universal deformation of the 
Koszul algebra $A(\QPA)$, as defined in \cite[4.1]{kosdef}.
It is a flat graded algebra over a certain
polynomial ring $S$, and its specialization
at $0 \in \Specm S$ is isomorphic to $A(\QPA)$ as a graded algebra.
The following result relates $\tilde{A}(\QPA)$ to the other quiver
algebras we have considered in this paper.

\begin{proposition} There are isomorphisms
\[\tilde{A}(\QPA) \cong e_\qvb Q_\qvb e_\qvb / \langle e_\xi e_\qvb\rangle\and S\cong \Sym(\fk),\]
and under these isomorphisms the $S$-algebra structure on $\tilde{A}(\QPA)$ is induced by 
$\vartheta\colon \fk \to Q_\qvb$.
\end{proposition}

\begin{proof}
Let $\PA$ be a polarized arrangement linked to $\QPA$,
so $A(\QPA) = A(\PA)$, the algebra defined in \cite{GDKD}.
Then \cite[4.14]{GDKD} tells us that
$A(\PA) \cong B(\PA^!)$ and 
$e_\qvb Q_\qvb e_\qvb / \langle e_\xi e_\qvb \rangle \cong \tilde{B}(\PA^!)$,
where $\PA^!$ is Gale dual to $\PA$ and the algebras
$B(\PA^!)$ and $\tilde{B}(\PA^!)$ are defined in \cite{GDKD}.
Thus the proposition is reduced to showing that the $\Sym(\fk)$-algebra
$\tilde{B}(\PA^!)$ is the universal deformation of $B(\PA^!)$; this is proven in \cite[8.7]{kosdef}.
\end{proof}

Let $\widehat{S}$ be the completion of $\Sym(\fk)$ at the 
maximal ideal generated by $\fk$, and define
$$\widehat{A}(\QPA) := \widehat{S}\otimes_S \tilde{A}(\QPA);$$ it is the 
completion of $\tilde{A}(\QPA)$ with respect to the grading.
Consider the categories $\widehat{A}(\QPA)\modfin$ 
and $\widehat{A}(\QPA)\mmod$ of finite dimensional and
finitely-generated modules over $\wh{A}(\QPA)$.  Our task will be to understand
these two categories in terms of the hypertoric enveloping algebra.

Let $\wh{\cO}_\fin(\QPA)$ be the full subcategory of $\Umod_\lf$
consisting of weight modules whose support lies in $\qvb$ and
which are locally finite over $U^+$.  Modules of this category
differ from modules in $\cO(\QPA)$ in that the center acts
by a generalized character rather than an honest character.
Let $$\widehat{U}_\la := U \otimes_Z \widehat{Z}_\la,$$ where
$\widehat{Z}_\la$ is the completion of the center $Z(U)$
at the maximal ideal $J_\la := \ker \la$.  
Set $$\widehat \bH_\la := \bH \otimes_Z \widehat{Z}_\la\and
\widehat{U}^+_\la := U^+ \otimes_Z \widehat{Z}_\la,$$
and define $\widehat{\cO}(\QPA)$  
to be the category of finitely generated
$\widehat{U}_\la$-modules $M$ so that
\begin{itemize}
\item $M = \bigoplus_{v \in \qvb} M_v$, that is, $M$ is \emph{set-theoretically} supported
on $\qvb\subset\SH \cong \Specm \widehat \bH_\la$, and
\item for any $m \in M$, $\widehat{U}^+_\la \cdot m$ is 
a finitely generated $\widehat Z_\la$-module.
\end{itemize}

\begin{theorem}\label{deformed category} There are equivalences
\[\widehat{\cO}(\QPA)_\fin \cong \widehat{A}(\QPA)\modfin\and
\widehat{\cO}(\QPA) \cong \widehat{A}(\QPA)\mmod.\]
For any $x\in\fk$, the action of $\mu_\la(x) \in Z(U)$ corresponds 
to multiplication by $\vart(x)e_\qvb \in Z(\widehat{A}(\QPA))$.
\end{theorem}

\begin{proof} The first equivalence and the identification of the action of 
$\mu_\la(x)$ on the left category follows from \cite[3.5.6, 4.3.1, \& 6.3]{MvdB}.
For the second equivalence, note that the 
methods of \cite{MvdB} do not directly
apply to $\widehat{U}_\la$, since $\widehat{H}_\la$ is not a 
quotient of a polynomial ring.  However, any object
$M$ of $\widehat{\cO}(\QPA)$ is isomorphic to the limit 
$\varprojlim M/(J_\la)^kM$, and each $M/(J_\la)^kM$ lies in
$\widehat{\cO}_\fin(\QPA)$, or more precisely in the subcategory 
$\cO_k$ of modules annihilated by $(J_\la)^k$.
This gives an equivalence between $\widehat{\cO}(\QPA)$ and
the category of inverse systems $\dots \to M_k \to M_{k-1} \to \dots \to M_1$,
where each $M_k$ lies in $\cO_k$, and each map 
$M_k \to M_{k-1}$ induces an isomorphism $M_k/(J_\la)^{k-1}M_k \cong M_{k-1}$.
The result now follows from the equivalence between $\cO_k$ and
the category of finitely-generated (hence finite dimensional)
modules over $\widehat{A}(\QPA)/\langle \vart(\fk)e_\qvb\rangle^k$.
\end{proof}

\section{Hypertoric varieties}\label{sec:hypertoric}
In this section we recall the definition of a hypertoric variety associated to a regular polarized arrangement
$\PA = (\vb_0, \eta, \xi)$, along with a quantization of its structure sheaf.  The
ring of equivariant sections of this quantization coincides with a central quotient of the hypertoric enveloping algebra determined by $\vb_0$.

\subsection{The variety defined}
Consider the moment map $\Phi:T^*\C^n\to\fk^*$ for the action of $K$ on $T^*\C^n$.  As noted in 
Remark \ref{grul}, one way to define $\Phi$ is by taking the associated graded of the quantized
moment map.  More concretely, for any $(z,w)\in T^*\C^n$
and $k_1h_1+\ldots+k_nh_n\in\fk\subset\t$, $$\Phi(z,w)(k_1h_1+\ldots+k_nh_n) = k_1z_1w_1+\ldots+k_nz_nw_n.$$
The parameter $\eta$ lies in $\SgrHZ/\vb_0$, which is naturally identified with the character lattice of $K$
(Section \ref{sec:algdef}).  Consider the subset
$\X\subset T^*\C^n$ of $\eta$-semistable points in the sense of geometric invariant theory.  
Then the {\bf hypertoric variety} $\fM = \fM(\PA)$ is defined as the categorical quotient 
of $\X\cap\Phi^{-1}(0)$ by $K$.  Let $\pi:\X\cap\Phi^{-1}(0)\to \fM$ be the quotient map.
Since $\eta$ is regular, 
the action of $K$ on $\X\cap\Phi^{-1}(0)$
is locally free, the quotient is geometric (that is, the fibers of $\pi$ coincide with the
orbits of $K$), and $\fM$ is a symplectic orbifold.
It is smooth if and only if $\vb_0$ is unimodular \cite[3.3]{BD}.

Let $\fM_0$ be the categorical quotient of $\Phi^{-1}(0)$ by $K$.
The inclusion of $\X\cap\Phi^{-1}(0)$ into $\Phi^{-1}(0)$ induces a projective map
$\nu:\fM\to \fM_0$.
Moreover, $\fM$ comes equipped with a line bundle 
$$\mathfrak{L}_\eta := \X\cap\Phi^{-1}(0)\times_K\C_\eta,$$
where $K$ acts on $\C_\eta$ via the character $\eta$.
This bundle is ample, and very ample relative to $\fM_0$.
If $\fM$ is smooth, then $\nu$ is a resolution, and the
Grauert-Riemenschneider theorem tells us that the higher 
cohomology of the holomorphic structure sheaf\footnote{Here
and elsewhere we use the unconventional notation $\fS$ for the structure sheaf
of a variety, since the standard symbol $\cO$ is being used for the category.} $\fS_{\fM}$ vanishes \cite[2.1]{Kal00}.
This is the first hint that $\fM$ is morally close to being affine.

\begin{example}\label{hypertoric examples}
If $K\subset T$ is the one-dimensional diagonal subtorus as in Example \ref{big example},
then $\fM \cong T^*\C P^{n-1}$ for any choice of regular $\eta$, and $\fM_0$ is obtained from $\fM$
by collapsing the zero section.  If $K\subset T$ is the $(n-1)$-dimensional determinant 1 subtorus 
as in Example \ref{next example}, then $\fM_0$ is isomorphic to the quotient of $\C^2$
by the symplectic action of $\Z/n\Z$, and for any choice of regular $\eta$, $\fM$ is a minimal resolution
(all of which are isomorphic).
\end{example}

Let $\bS:=\C^\times$ act on $T^*\C^n$ by inverse scalar multiplication;
that is, $s\cdot (z,w) := (s^{-1}z, s^{-1}w)$.  
This induces an action on both
$\fM$ and $\fM_0$, and the map $\nu$ is $\bS$-equivariant.
This action does not preserve the symplectic form; rather we have
$s\cdot\omega = s^2\omega$ for all $s\in\bS$.
Let $o\in\fM_0$ be the point represented by $0\in T^*\C^n$, so that
\begin{equation*}\label{goodness}
\displaystyle\lim_{\bS\ni s\to\infty}s\cdot p_0 = o
\end{equation*}
for all $p_0\in\fM_0$.
It follows that, for all $p\in\fM$, the limit $\displaystyle\lim_{\bS\ni s\to\infty}s\cdot p$ exists and lies in
$\nu^{-1}(o)$.  

The hypertoric variety $\fM$ admits a Hamiltonian action of $T/K$, 
induced by the action of $T$ on $\C^n$, which commutes with the action of $\bS$.
Let $\bT:=\C^\times$ act on $\fM$ via $\xi$, 
which is naturally a cocharacter of $T/K$ (Section \ref{sec:algdef}).
The following proposition is straightforward.

\begin{proposition}\label{prop-grul}
The hypertoric variety $\fM_0$ is affine, and for any central character $\la$
of the hypertoric enveloping algebra $U$
there is a natural isomorphism $$\gr U_\la\cong\C[\fM_0]\cong\C[\fM].$$
The $\Z$-grading on $U_\la$ given in Section \ref{o-def} descends
to the grading on $\C[\fM]$ induced by the action of $\bT$.  The $\N$-grading
on $\gr U_\la$ induced by the associated graded construction coincides with
the grading on $\C[\fM]$ induced by the action of $\bS$.
\end{proposition}

For any quantized polarized arrangement $\QPA = (\vb_0, \qvb, \xi)$, let $L := \oplus_{a\in\cP_{\qvb,\xi}}L_\a^\qvb$ 
be the direct sum of the simple objects, and consider the {\bf Yoneda algebra} $$\bigoplus_{i \ge 0} \Ext^i(L, L).$$
(This is the algebra whose module category is Koszul dual to $\cO(\QPA)$,
as explained in Section \ref{sec:koszul}.)  The following result can be taken as a first piece of
evidence that there is a strong relationship between $\cO(\QPA)$ and the geometry of hypertoric varieties.

\begin{theorem}\label{center}
If $\QPA$ is integral and regular and linked to $\PA$, 
then the center of the Yoneda algebra of $\cO(\QPA)$ is isomorphic as a graded algebra
to the cohomology of $\fM(\PA)$.
\end{theorem}

\begin{proof}
By Theorem \ref{blpw}, $\cO(\QPA)$ is equivalent to the module category of the ring $A(\PA)$
introduced in \cite{GDKD}.  By \cite[5.24]{GDKD}, the Yoneda algebra of $\cO(\QPA)$ is isomorphic
to the ring $B(\PA)$ introduced in \cite{GDKD}.  By \cite[4.7 \& 4.16]{GDKD}, the center of $B(\PA)$
is isomorphic to the cohomology ring of $\fM(\PA)$.
\end{proof}

\begin{remark}
By Theorem \ref{reduced-thm}, every regular block $\cO(\QPA)$ of hypertoric category $\cO$ is equivalent
to a regular integral block $\cO(\QPA')$ for some different quantized polarized arrangement $\QPA'$.
Thus Theorem \ref{center} gives a characterization of the center of the Yoneda algebra of an arbitrary regular block.
\end{remark}

\subsection{The relative core}\label{core}
Consider the reduced Lagrangian subvariety
$$\fM^+:= \{p\in \fM\mid \lim_{\bT\ni t\to 0}t\cdot p\,\,\text{exists}\}.$$
It contains the {\bf core} $\nu^{-1}(o)$ (because the $\bS$ and $\bT$ actions commute and
$\nu$ is projective) but is strictly bigger.
We call $\fM^+$ the {\bf relative core} of $\fM$.

Another Lagrangian subvariety of $\fM$ that has appeared in the literature is the {\bf extended core},
which is equal to the zero set of the moment map for the $T/K$-action.  The extended
core contains the relative core; in fact, it is equal to the union of all of the (finitely many) 
possible relative cores, as $\xi$ varies among generic parameters.

\begin{proposition}[\mbox{\rm \cite[6.5]{BD}}]\label{cores}
The extended core of $\fM$ is isomorphic to a union of toric varieties $C_{\a}$ given by the 
polyhedra $\Delta_{\a}$ for $\a\in\cF_{\eta}$
glued along toric subvarieties as prescribed by the incidences
of the polyhedra.
The relative core is the subvariety of the extended core corresponding to those $\a$ in
$\cP_{\eta,\xi}\subset \cF_{\eta}$, and the core is the subvariety corresponding to
those $\a$ that are totally bounded.
\end{proposition}

As in Section \ref{sec:weyl}, let $x_1,\ldots, x_n$ be coordinates on $\C^n$, 
and let $y_1,\ldots, y_n$ be dual coordinates,
so that $$\C[T^*\C^n]=\C[x_1,y_1,\ldots, x_n,y_n].$$  The subvariety $C_{\a} \subset \fM$ of Proposition \ref{cores}
is cut out by the equations
\begin{equation}\label{Walpha}\text{$y_i = 0$ if $\a(i) = +$}
\qquad\text{and}\qquad
\text{$x_i=0$ if $\a(i) = -$.}
\end{equation}
If $\a\notin \cF_{\eta}$, then these equations will have no solutions in $\fM$.
Let $C_{0,\a}\subset\fM_0$ be defined by the same equations.
If $\a$ is feasible, then $C_{0,\a}$ is the image of $C_{\a}$
along the map $\nu:\fM\to\fM_0$.  However $C_{0,\a}$ is nonempty
even if $\a$ is infeasible; in particular, it always contains the point $o$.

If $M$ is a finitely-generated $U_\la$-module, we can put a filtration on $M$ which is
compatible with the filtration on $U_\la$ by choosing a generating set for $M$.  We can then consider
the support $\supp(\gr M) \subset \M_0$, where we think of $\gr M$
as a sheaf on $\gr U_\la \cong \C[\M_0]$.  Standard arguments show that the support of this sheaf is 
independent of the choice of filtration.\footnote{The reader is cautioned not to confuse the
support of $\gr M$, which is a subvariety of $\fM_0$, with the support of $M$ as defined 
in Section \ref{sec:weight}, which is a subset of $\V$.}
 
\begin{lemma}\label{support}
If $\qvb\subset\SH$ is an integral $\vb_0$-orbit and $\a\in\cF_\qvb$, then $C_{0,\a} = \supp\Big(\gr L_\a^\qvb\Big)$. 
\end{lemma}

\begin{proof} Let $W_\a\subset T^*\C^n$ be the Lagrangian subspace
defined by the equations \eqref{Walpha}, so that 
$\C[C_{0,\a}] = \C[W_\a]^K$.
If we put the filtration on $L_\a$ induced by the image of 
$1 \in \D$, then we get an isomorphism
 $\gr L_\a \cong \C[W_\a]$ of modules over $\gr \D \cong \C[T^*\C^n]$.
This induces a filtration on $L^\qvb_\a$, making $\gr L^\qvb_\a$ 
a $\gr U_\la$-submodule of $\gr L_\a$.  It follows immediately
that $\supp(\gr L^\qvb_\a) \subset C_{0,\a}$. 

For the other inclusion, note that $\gr L^\qvb_\a$ will be 
isomorphic to some nonzero $K$-isotypic component $\C[W_\a]^{K,\theta}$
of $\C[W_\a]$.  There is an injection 
$\C[W_\a]^K \hookrightarrow \C[W_\a]^{K,\theta}$ given by multiplication
by any nonzero polynomial in $\C[W_\a]^{K,\theta}$, so 
we have $C_{0,\a} \subset \supp(\gr L^\qvb_\a)$.
\end{proof}

We end this section with two technical lemmas about $\fM^+$.
Let $\mathfrak{J}$ be the ideal generated by functions on $\fM$ which are
eigenvectors for the
$\bS\times \bT$-action of positive $\bS$-weight and non-negative $\T$-weight.

\begin{lemma}\label{vanishing}
The relative core $\fM^+\subset\fM$ is the vanishing locus of $\mathfrak{J}$.
\end{lemma}

\begin{proof}
Let $f\in\mathfrak{J}$ be an eigenvector of positive $\bS$-weight and non-negative $\T$-weight.
Then $f$ vanishes on $\bS$-fixed points, and the core (being projective) contains at least one such
point.  Thus $f$ vanishes on the entire core.  
For any $p\in\fM^+$, 
$$f(p) = \lim_{\bT\ni t\to 0}(t\cdot f)(t\cdot p) = 0,$$
since $t\cdot f$ is approaching either $f$ or $0$, and $t\cdot p$ is approaching an element
of the core.
Thus $f$ vanishes on all of $\fM^+$, so
$\fM^+$ is contained in the vanishing locus of $\mathfrak{J}$.
 
For the other inclusion, consider the ideal $\mathfrak{J'}$ generated by functions of
\emph{positive} $\T$-weight.  Since all points of $\fM_0$ limit to $o$ under the $\bS$-action, all functions on $\fM$
have non-negative $\bS$-weight, and the only functions with $\bS$-weight zero are the constant functions, so we have
$\mathfrak{J'}\subset \mathfrak{J}$.  Suppose that $p\in\fM$ lies in the vanishing locus of $\mathfrak{J}$,
and therefore of $\mathfrak{J'}$.  Then for every function $f\in\C[\fM]$, the limit
$$f\left(\lim_{t \to 0} t \cdot p\right) = \lim_{t\to 0}(t^{-1}\cdot f)(p)$$ exists.
(If $f$ has non-positive $\bT$-weight the right-hand side clearly exists, and if $f$ has positive $\bT$-weight
then $t^{-1}\cdot f\in\mathfrak{J}$, so the right-hand side is equal to zero.)
Since $\fM$ is projective over the affine variety $\fM_0$, this implies that $\displaystyle\lim_{t \to 0} t \cdot p$ exists,
and therefore that $p \in \M^+$.
\end{proof}

Let $f_\bT:\fM\to\bt^*\cong\C$ be the moment map for the $\bT$-action on $\fM$, 
where $\bt$ is the Lie algebra of $\bT$.  We say a coherent
sheaf $F$ on $\fM$ is {\bf $\bt$-equivariant} if it is equipped with an endomorphism 
$d:F\to F$ such that, for all meromorphic sections $\gamma$ and
functions $f$, we have $d(f\gamma)=\{f_{\bT},f\}\gamma+fd(\gamma)$.  Here
$\{\, ,\}$ is the Poisson bracket induced by the symplectic form on 
$\fM$, so $\{f_{\bT}, f\} = \frac{d}{dt}(t\cdot f)|_{t = 0}$.

\begin{lemma}\label{coherent-fd}
  For any $\bt$-equivariant coherent sheaf $F$ on $\fM$ which is
  set-theoretically supported on $\fM^+$, the $\bt$-action on $\Gamma(\fM;F)$ 
  is locally finite, the generalized eigenspaces 
  are finite dimensional, and the eigenvalues that appear are
  bounded above.
\end{lemma}

\begin{proof}
  The conclusion holds for a sheaf if it holds for the
  successive quotients of a filtration of the sheaf, so we may
  assume that $F$ is scheme-theoretically supported on one of the
  components $C_{\a}$ of $\fM^+$.
  Now consider the pushforward $\nu_*F$  to $\fM_0$.  This is a coherent sheaf
  scheme-theoretically supported on $C_{0,\a} = \nu(C_{\a})$, and so there is a
  surjective $\bt$-equivariant morphism $N\otimes \fS_{C_{0,\a}}\to
  \nu_*F$ for some finite dimensional $\bt$-representation $N$.
Since $\fM_0$ is affine, the sections functor is exact, so we need only prove the result for 
the $\bt$-action on the ring of functions on $C_{0,\a}$.  This follows 
from the fact that $\displaystyle\lim_{\bT\ni t\to 0}t\cdot p_0 = o$ for all $p_0\in C_{0,\a}$.
\end{proof}

\subsection{Quantizing the hypertoric variety}\label{quantizing}
In this section we explain how the algebra $U_\la$ arises as sections of 
a quantization of the structure sheaf of $\M$.  
Let $$\cD := \fS_{T^*\C^n}\laurenth\and \cD(0) := \fS_{T^*\C^n}[[\hh]]\subset\cD,$$
where again $\fS_{T^*\C^n}$ denotes the holomorphic structure sheaf of $T^*\C^n$ and $\hh$ is a formal parameter.
Let $\chi$ be the Poisson bivector on $T^*\C^n$.  The Moyal product on $\cD$ is defined by 
\begin{equation*}
  f\star g=m\circ e^{\hbar\chi/2}(f\otimes g),
\end{equation*}
where $m:\fS_{T^*\C^n}\otimes_\C\fS_{T^*\C^n}\to \fS_{T^*\C^n}$ is the multiplication map.
Thus $\cD$ is a sheaf of associative $K\times\bS$-equivariant $\C\laurenth$-algebras, 
where $\bS$ acts on $\hh$ with weight 1, and therefore on $\hbar$ with weight 2.
(This reflects the fact that $\bS$ acts on the symplectic form of $T^*\C^n$ with weight 2.)
This sheaf is flat over $\C[[\hh]]$,
and we have natural isomorphisms 
$$\cD(0)/\hh\cD(0)\cong\fS_{T^*\C^n}
\and
\Gamma_\bS(\cD)\cong \D,$$
where $\Gamma_\bS$ is the functor that takes $\bS$-invariant global sections \cite[2.6]{BeKu}.
The second isomorphism is given by sending the element $\hbar^{\nicefrac{-1}{2}}x_i\in\Gamma_\bS(\cD)$ to 
$x_i \in \D$ and $\hbar^{\nicefrac{-1}{2}}y_i\in\Gamma_\bS(\cD)$ to $\partial_i$.

The reduction procedure we applied to obtain $U_\la$ from $\D$ can be ``sheafified'' as follows.  
Let $\la\colon Z(U) \to \C$ be any central character.  Since 
$\ker\la\subset Z(U) \subset U \subset \D$,
we can regard elements of $\ker \la$ as $\bS$-invariant sections of $\cD$.  Let
$$\cL_\la := \cD|_{\X}\Big/\ker \la \cdot \cD|_{\X}
\and \cL_\la(0) := \cD(0)|_{\X}\Big/\Big(\cD(0)\cap \ker \la \cdot \cD|_{\X}\Big).$$
Thus $\cL_\la$ is a $\bS \times K$-equivariant sheaf of right $\cD$-modules on $\X$,
and $\cL_\la(0)$ is a $\bS \times K$-equivariant $\cD(0)$-lattice.
Both are supported on $\X\cap\Phi^{-1}(0)$,
thus 
$$\cU_\la:= \left(\pi_*\sEnd_\cD(\cL_\la)\right)^K\and
\cU_\la(0):= \left(\pi_*\sEnd_{\cD(0)}(\cL_\la(0))\right)^K$$
are $\bS$-equivariant sheaves of algebras on $\fM$.
Consider also the sheaf of $(\cU_\la, \cU_{\la'})$-bimodules 
$${_{\la}\ctra_{\la'}}:=\left(\pi_*\sHom_\cD(\cL_\la,\cL_{\la'})\right)^K$$
and its $(\cU_\la(0), \cU_{\la'}(0))$-lattice
$${_{\la}\ctra_{\la'}}(0):=\left(\pi_*\sHom_{\cD(0)}(\cL_\la(0),\cL_{\la'}(0))\right)^K.$$
Just as $\cU_\la(0)/\hh\cU_\la(0)\cong \fS_\fM$, we have that 
$${_{\la}\ctra_{\la'}}(0)/\hh \cdot {_{\la}\ctra_{\la'}}(0)\cong \mathfrak{L}_{\la'-\la}
=\X\cap\Phi^{-1}(0)\times_K\C_{\la-\la'}.$$

\begin{proposition}
We have isomorphisms $\Gamma_\bS(\cU_\la)\cong U_\la$ and $\Gamma_\bS({_{\la}\ctra_{\la'}})\cong{_{\la}\tra_{\la'}}$.
\end{proposition}

\begin{proof}
The isomorphism $\D \cong \Gamma_\bS(\cD)$ induces a map $Y_\la \to \Gamma(\X, \cL_\la)^\bS$, where 
$Y_\la = \D/\langle\ker\la\rangle\D$ is the $\D$-module introduced in Remark \ref{module Y}.  Since the complement
of $\X$ in $T^*\C^n$ has codimension at least $2$, this map is an isomorphism.  
Since $1$ is an $\bS$-invariant section of $\cL_\la$, 
evaluating at $1$ gives a map $\Gamma_\bS(\cU_\la) \to \End_\D(Y_\la) = U_\la$, which is easily seen to be an isomorphism.
\end{proof}

Following \cite{KR}, we call a $\cU_\la(0)$-module $\cN(0)$ {\bf coherent} if it is
locally finitely generated or equivalently by Nakayama's lemma, if
$\cN(0)/h\cN(0)$ is a coherent sheaf and call a
$\cU_\la$-module
{\bf good} if it admits a coherent $\cU_\la(0)$-lattice (and thus is
itself coherent).  
Let $\mgs$ denote the category of good $\bS$-equivariant $\cU_\la$-modules.

\begin{remark}
  There are heuristic reasons to treat $\mgs$ as a version of the
  Fukaya category of $\fM$. 
  One justification comes from the physical theory of A-branes, 
  which the Fukaya category 
  attempts to formalize.  Kapustin and Witten \cite{KW07} suggest that
  on a hyperk\"ahler manifold such as $\fM$ there are objects in an
  enlargement of the Fukaya category which correspond not just to
  Lagrangian submanifolds, but higher dimensional coisotropic
  submanifolds.  In particular, there is a object in this category
  supported on all of $\fM$ called the {\bf canonical coisotropic
    brane}.  Following the prescription of Kapustin and Witten further
  shows that endomorphisms of this object are exactly the algebra
  $U_\la$, or if interpreted sheaf theoretically, $\cU_\la$.
  This leads us to conjecture that there is a natural equivalence
  between $\mgs$ and the Fukaya category of $\fM$, twisted by the
  B-field $H^2(\fM;\C^{\times})$ determined\footnote{The group
  $H^2(\fM;\C^{\times})$ is isomorphic to $H^2(\fM;\C)/H^2(\fM;\Z)$.
  The vector space $H^2(\fM;\C)$ is isomorphic to $\fk^*\cong W/V_0$ via the
  Kirwan map.  The parameter $\la$ determines a $V_0$-orbit in $\SH$,
  which determines a $V_0$-orbit in $\SgrH$ up to translation by the lattice $W_\Z/\vb_0\cong H^2(\fM;\Z)$.} by $\la$.
 When $\fM$ is replaced by
  the cotangent bundle of an arbitrary real analytic manifold  an analogous statement is proven by Nadler and Zaslow \cite{NZ}. 
\end{remark}

\section{Localization}\label{sec:Localization}
As in Section \ref{sec:hypertoric}, let $\PA = (\vb_0,\eta,\xi)$ be a regular polarized arrangement,
let $\fM$ be the hypertoric variety associated to $\PA$, let $U$ be the hypertoric enveloping algebra
associated to $\vb_0$.  Fix a central character $\la$ of $U$, and let $\V\subset\SH$ be the corresponding $V_0$-orbit. 
In this section we show that the infinitesimal block $\cO_\la$ of hypertoric category $\cO$ (Definition \ref{defn of O})
can be ``localized" to a certain category of sheaves of modules over the quantized structure sheaf $\cU_\la$,
supported on the relative core $\fM^+\subset\fM$.  
When $\vb_0$ is unimodular and $\QPA$ an integral quantized polarized arrangement linked to $\PA$, we use localization to obtain a topological interpretation of the Grothendieck group
of $\cO(\QPA)$.

\subsection{The localization functor}
We define the functor $$\Loc:U_\la\mmod \to \mgs$$ by putting
$\Loc(M) := \cU_\la\otimes_{U_\la} M$
for any $U_\la$-module $M$.  We note that any good filtration of $M$ induces a $\cU_\la(0)$-lattice for $\Loc(M)$, 
namely $\cU_\la(0)\otimes_{R(U_\la)} R(M)$, where $R(U_\la)$ and $R(M)$ are the Rees algebra and Rees module
of $U_\la$ and $M$, respectively.
The functor $\Loc$ is adjoint to the
$\bS$-invariant global sections functor $$\Gamma_\bS:\mgs\to U_\la\mmod.$$
Just as a choice of a good filtration of $M$ induces a particular lattice in $\Loc(M)$,
a choice of lattice $\cM(0)\subset\cM\in\mgs$ induces a good filtration of $\Gamma_\bS(\cM)$
with $\gr \Gamma_\bS(\cM)\cong\Gamma\big(\cM(0)/\hh \cM(0)\big)$.

The following theorem, based on the results of Kashiwara and Rouquier \cite{KR} and their adaptation
to the hypertoric case by Bellamy and Kuwabara \cite{BeKu}, is an analogue of the localization theorem of Beilinson and Bernstein \cite{BB}.

\begin{theorem}\label{localization}
If $\vb_0$ is unimodular and $\cF_\qvb = \cF_{\qvb+r\eta}$ for all $\vb_0$-orbits $\qvb\subset\V$
and integers $r\geq 0$, then the functors $\secs$ and $\Loc$ are quasi-inverse equivalences.
\end{theorem}

\begin{proof}
This follows from \cite[3.5]{BeKu} and Proposition \ref{trans-equiv}.  
Note that a very similar theorem with slightly different hypotheses is proven in \cite[5.8]{BeKu}. 
\end{proof}

\subsection{The category \texorpdfstring{$\cQ_\la$}{Q lambda}}\label{sec:gco}
Our next task is to determine which objects of $\mgs$ are sent by the sections functor $\secs$
to $\cO_\la$, and conversely to determine the image of the localization functor restricted to $\cO_\la$.  
Note that we do not assume that $\secs$ and $\Loc$ are equivalences in this section.

\begin{definition}\label{def:ql}
Let $\cQ_\la\subset\mgs$ be the full subcategory consisting of objects $\cM$ satisfying the following 
two additional conditions:
\begin{itemize}
\item there exists a $\cU_\la(0)$-lattice $\cM(0)\subset\cM$ that is preserved by the action of 
$\hat\xi\in U_\la$\footnote{Note that $h\hat\xi$ is contained in $U_\la(0)$, but $\hat\xi$ is not,
so this condition is not vacuous.}
\item $\cM$ is supported on the relative core $\fM^+\subset\fM$.
\end{itemize}
\end{definition}

\begin{theorem}\label{o-sections}
If $\cM\in\cQ_\la$, 
then $\Gamma_{\bS}(\cM)\in\cO_\la$.
\end{theorem}

\begin{proof}
Choose a $\cU_\la(0)$-lattice $\cM(0)$ that is preserved by the action of $\hat\xi$,
and let $F:=\cM(0)/\hh \cM(0)$.  The action of $\hat\xi$ on $\cM(0)$ induces a $\bt$-equivariant
structure on $F$; Lemma \ref{coherent-fd} tells us that $\Gamma(\fM, F)$ decomposes into
finite dimensional generalized eigenspaces and that the eigenvalues that appear are bounded above.
Since $\Gamma(\fM, F)\cong \gr\secs(\cM)$ and the $\bt$-equivariant structure on $F$ is induced by
the action of $\hat\xi$ on $\cM$, this implies that $\secs(\cM)$ decomposes into finite dimensional
generalized eigenspaces for the action of $\hat\xi$ and that the eigenvalues that appear are bounded above.
Thus $\secs(\cM)$ lies in $\cO_\la$ by Lemma \ref{gamma-finite}.
\end{proof}

\begin{theorem}\label{o-localization}
If $M\in\cO_\la$, then $\Loc(M)\in\cQ_\la$.
\end{theorem}

\begin{proof}
We need to verify that $\Loc(M)$ satisfies the two conditions of Definition \ref{def:ql}.
The first condition is easy; if we choose a filtration of $M$ generated by weight vectors,
then the associated lattice $\Loc(M)(0)\subset \Loc(M)$ will be preserved by the action of $\hat\xi$.
Thus we only need to show that $\Loc(M)$ is supported on the relative core.

Let $f$ be a global function on
$\fM$ which is an $\bS\times \bT$-weight vector of non-negative
$\T$-weight and positive $\bS$-weight. 
Let $\fM_{f}$ be the subset of $\fM$ where $f$ is not zero,
and $i_{\! f}:\fM_{f}\to \fM$ be the inclusion.
Then $i_{\! f}^{-1}\Loc(M)(0)$ is a sheaf of flat $\C[[\hh ]]$-algebras with fiber
$\fS_{\fM_{f}}\otimes_{\C[\fM]}\gr(M)$ at $\hh = 0$.
Lemma \ref{gamma-finite} tells us the $\bT$-weight spaces of $\gr(M)$
are finite dimensional and the weights that appear are bounded above.
If $f$ has positive $\bT$-weight, then it acts nilpotently on $\gr(M)$ by the boundedness of the weights that appear.
If $f$ has $\bT$-weight 0, then it still acts nilpotently by the finite dimensionality of the $\bT$-weight spaces and the fact
that $f$ has positive $\bS$-weight.
This implies that $\fS_{\fM_{f}}\otimes_{\C[\fM]}\gr(M)=0$, and therefore that
$i_{\! f}^{-1}\Loc(M)(0)=0$.  Thus the support of $\Loc(M)$ is disjoint from $\fM_{f}$.  
The theorem then follows from Lemma \ref{vanishing}.
\end{proof}

\begin{corollary}\label{linked-loc}
Suppose that $\vb_0$ is unimodular and $\qvb\subset\V$ is integral.
If $\QPA := (\vb_0,\qvb,\xi)$ is linked to $\PA$, then the functors $\secs$ and $\Loc$ induce
quasi-inverse equivalences between $\cO(\QPA)$ and $\cQ_\la$.
\end{corollary}

\begin{proof}
By Proposition \ref{prop:linked}, we have $\cF_\qvb = \cF_{\qvb+r\eta}$ for all positive integers $r$.
By unimodularity, we have $I_{\qvb'} = \emptyset = I_{\qvb'+r\eta}$ for all $\vb_0$-orbits $\qvb'\subset\V$
different from $\qvb$.  By Remark \ref{inf-block=block}, we have $\cO(\QPA) = \cO_\la$.
The result then follows from Theorems \ref{localization}, \ref{o-sections}, and \ref{o-localization}.
\end{proof}

\begin{remark}
We have assumed unimodularity of $\vb_0$ (which is equivalent to the statement that $\fM$ is a manifold rather than an orbifold) so that we can use the results of \cite{KR} and \cite{BeKu}.
We expect that this condition is not essential.
\end{remark}

\begin{example}\label{kazhdan-needed}
Consider the example where $n=2$ and $K\subset T$ is the diagonal subtorus (see Examples \ref{big example} and 
\ref{hypertoric examples}), so that $\fM \cong T^*\C P^1$.  In this case it is possible to choose $\la$ such that
\[U_\la = \C\langle x_1\del_1,x_2\del_1,x_1\del_2,x_2\del_2\rangle/(x_1\del_1+x_2\del_2=0)\]
is isomorphic to the ring of polynomial differential operators on $\C P^1$ and 
$\mgs$ is equivalent to the category of D-modules on $\C P^1$.  
The second condition of Definition \ref{def:ql} requires the singular support to lie in the relative core,
which in this case consists of the zero section along with the fiber at a single point.
The first condition is a regularity assumption that is needed
for Theorem \ref{o-sections} to hold.

To see this, consider the 
$\cU_\la$-module $\cU_\la/\cU_\la(x_1\del_2-1)$.  Restricted to the open subset where $x_1\neq 0$, this is a non-singular connection on the trivial vector bundle; if $z=x_2/x_1$, then it is simply $\frac{d}{dz}-1$.  Thus this module satisfies
the second condition of Definition \ref{def:ql}.  Its $\bS$-equivariant sections, however, are isomorphic to
$U_\la/U_\la(x_1\del_2-1)$, and the fact that the image of $1\in U_\la$ is fixed by $x_1\del_2$ implies that this
is not a weight module.  This reflects the fact that the corresponding D-module has a non-regular singularity at $z=\infty$.
\end{example}

\subsection{The Grothendieck group}\label{sec:grothendieck}
Let $d := \dim V_0$. 
Consider the homomorphism
$$\supp: K(\cQ_\la)_\C\to \hmid$$
given by taking an object $\cM$ to the support cycle of the coherent
sheaf $\cM(0)/\hh \cM(0)$.  The fact that this does not depend on the choice of $\cU_\la(0)$-lattice
$\cM(0)\subset\cM$ is a sheafified version
of \cite[1.1.2]{Gin86}.
The vector space $H^{2d}_\bT(\fM; \C)$ has a natural nondegenerate pairing given by integrating the product
of two classes.  Though $\fM$ is not compact, this integral can be defined by formally applying the
Atiyah-Bott-Berline-Vergne localization formula; see \cite[\S 1]{HP05} for details.  We will refer to this pairing
as the {\bf integration pairing} on $\hmid$.

\begin{remark} 
It may seem more intuitive for the support homomorphism to take values in the group $\hbmc$, 
since the support cycle of $\cM(0)/\hh \cM(0)$ is always a sum of relative core components. 
In fact, these groups are canonically isomorphic via the isomorphisms 
$$\hbmc\to\htbmc\to\htbmm\to\hmid.$$ 
Geometrically, this isomorphism takes the class $[C_\a]$ to $[C_\a]$.  The main reason that we choose 
to work with $\hmid$ is that it is easier to understand the integration pairing on this space; it is also useful 
for the analysis below of the supports of localizations of standard modules. 
\end{remark} 

Suppose now that the hypotheses of Corollary \ref{linked-loc} are satisfied.
Then it is easy to check that $\supp[\Loc(L_\a^\qvb)] = [C_\a]$ for all $\a\in\cP_{\qvb,\xi}$.
Since we know the multiplicities of each simple module in the standard module 
$S_\a^\qvb$, we can compute its image under $\supp$ as follows.  We will 
use the localization isomorphism $H^{2d}_\bT(\fM; \C) \to H^{2d}_\bT(\fM^\bT; \C)$.
A proof of the following result in a more general setting will appear
in \cite{BLPWgco}.
\begin{proposition}
The class $\supp[\Loc(S_\a^\qvb)]$
restricts to the element of 
\[H^{2d}_\bT(\fM^\bT; \C)\cong\bigoplus_{\a\in\cP_{\qvb,\xi}}\Sym(\bt^*)\] which 
is supported at the fixed point $p_\a$ and whose value at $p_\a$ is the product of the negative weights of the action
of $\bT$ on the tangent space $T_{p_\a}\fM$.  
\end{proposition}
\begin{proof}
By Proposition \ref{KL}, $\supp[\Loc(S_\a^\qvb)]$ is equal to the sum of $[C_\b]$ for all $\b$ such that 
$\Delta_\b\subset\Sigma_\a$. 
For any $\gamma\in\cP_{\qvb,\xi}$, let $w_\gamma\in V_\R$ be the vertex of $\cH$
corresponding to $\gamma$ by Remark \ref{bijection}.  Then the restriction of $[C_\b]$ to $p_\gamma$
is equal to the product
of the projections onto $\bt^*$ of the primitive vectors in the directions of the one-dimensional flats
of $\cH$ passing through $w_\gamma$; the signs of these vectors are determined by requiring that they point
away from $\Delta_\b$.  Alternatively, this may be described as the product of the weights of
the normal bundle to $C_\b\subset\fM$ at $p_\gamma$.
For any $\gamma\neq\a$ such that $w_\gamma\in\Sigma_\a$, the contributions
of the various $[C_\b]$ to $p_\gamma$ will cancel.  Thus the restriction of $\supp[\Loc(S_\a^\qvb)]$ to $p_\a$
coincides with the restriction of $[C_\a]$.  Since $T_{p_\a}C_\a$ is equal to the sum of the positive weight spaces
of $T_{p_\a}\fM$, the fiber of the normal bundle at $p_\a$ is isomorphic to the sum of the negative weight spaces,
thus the restriction of $[C_\a]$ to $p_\a$ is equal to the product of the negative weights.  
\end{proof}

\begin{proposition}\label{forms}
Suppose that $\vb_0$ is unimodular and $\QPA$ is integral and linked to $\PA$, so that $\cO(\QPA)$ is equivalent to $\cQ_\la$.
The isomorphism $\supp:K(\cQ_\la)_\C\to \hmid$ takes takes the Euler form to $(-1)^d$ times the integration pairing.
\end{proposition}

\begin{proof}
The classes $\{[\Loc(S_\a^\qvb)]\mid\a\in\cP_{\qvb,\xi}\}$ form an orthonormal basis for $K(\cQ_\la)_\C$.
The fact that the integration pairing of $\supp[\Loc(S_\a^\qvb)]$ with $\supp[\Loc(S_\b^\qvb)]$ is zero for $\a\neq\b$
follows from the fact that the restriction of $\supp[\Loc(S_\a^\qvb)]$ to the fixed point set is supported at $p_\a$.
The integration pairing of $\supp[\Loc(S_\a^\qvb)]$ with itself is equal to the square of the product of the negative weights
of the action of $\bT$ on $T_{p_\a}\fM$ divided by the product of all of the weights.
Since the action is symplectic, the weights come in $d$ pairs that each add to zero, so this quotient is equal to $(-1)^d$.
\end{proof}

\section{Cells}
\label{sec:cells}
In this section we define and study the hypertoric analogues of Kazhdan-Lusztig cells in BGG category 
$\cO$.  We fix quantized polarized arrangement $\QPA = (\vb_0, \qvb, \xi)$, 
which we assume to be both regular and integral.  Recall that the 
assumption of integrality does not actually lose any generality,
since Theorem \ref{reduced-thm}
tells us that $\cO(\QPA)$ is either trivial (if $\qH$ is inessential) or equivalent
to $\cO(\QPA')$ for some integral quantized polarized arrangement $\QPA'$.

\subsection{Left cells}\label{sec:left}
Consider a pair of feasible sign vectors $\a,\b\in\cF_{\qvb}$ 
along with the corresponding simple objects of $L^\qvb_\a$ and $L^\qvb_\b$ of $\Ulmod_\qvb$.

\begin{definition}
We say that $\a\Lleq \b$ if and only if $\Ann L_\b^\qvb\subset \Ann L_\a^\qvb$.  We say that $\a$ and $\b$ are in the same {\bf left cell} 
of $\cF_\qvb$ if $\a\Lleq \b$ and $\b\Lleq \a$.
\end{definition}

In order to give a characterization of the left cells of $\cF_\qvb$, we introduce some basic notation and
constructions for hyperplane arrangements.  If $F\subset V_0$ is a flat of the hyperplane arrangement $\cH_0$
(Section \ref{sec:polarr}), we let $$\cI_F:= \{i\mid F\subset H_{0,i}\}\subset\{1,\ldots,n\}$$
be the set indexing the hyperplanes that contain $F$.
We define the {\bf localization}\footnote{There is no relationship between this notion of localization
and the one that is the topic of Section \ref{sec:Localization}.} 
of $\cH_0$ at $F$ to be the hyperplane arrangement
$$\cH_0^F := \{H_{0,i}/F\mid i\in\cI_F\}$$
in the vector space $V_{0,\R}/F$.  Similarly, we define the localization of $\qH$ at $F$ to be the hyperplane arrangement
$$\qH^F := \{H^\pm_{i}/F\mid i\in\cI_F\}$$
in the affine space $\V_\R/F$.

Any sign vector $\a\in\{+,-\}^{\cI_F}$ determines a polyhedron $\DDelta_\a^F\subset\V_\R/F$
given by the equations
$$h^+_i \ge 0\text{ for all $i \in \cI_F$ with }\a(i) = + \and h^-_i \le 0 \text{ for all $i\in\cI_F$ with }\a(i) = -.$$
Our assumption that $\QPA$ is regular implies that every non-empty polyhedron $\DDelta_\a^F$ will contain
a point in the lattice $\qvb/(F\cap\vb_0)$.
We call a sign vector $\a\in\{+,-\}^{\cI_F}$ {\bf compact} if $\DDelta_\a^F$ is non-empty and compact.\footnote{This
definition is closely related to the notion of feasibility and total boundedness in Remark \ref{totally bounded},
but that had to do with ordinary arrangements rather than ``doubled" arrangements, which we have here.}
We say that $F$ is {\bf coloop-free} if there exists a compact sign vector
in $\{+,-\}^{I_F}$ for the localized arrangement $\qH^F$.  Our assumption of regularity
ensures that this definition agrees with the standard notion of coloop-free flats.
In \cite[\S 2]{PW07}, the authors show that the symplectic leaves of $\fM_0$
are indexed by coloop-free flats; we will denote by $\fM_0^F$ the leaf indexed by $F$.

For any sign vector $\a\in\{+,-\}^n$, let $F_\a\subset V_{0,\R}$ be the linear span of the cone
$\Delta_{0,\a}$ defined in Section \ref{sec:bf}, and let $\cI_\a := \cI_{F_\a}$.  If $\a$ is feasible, then the restriction of $\a$
to $\cI_\a$ is a compact sign vector for the localized arrangement $\qH^{F_\a}$.
In particular, $F_\a$ is coloop-free, and every coloop-free flat arises in this manner.

Fix a polarized arrangement linked to $\QPA$, and let $\fM$ be the associated hypertoric variety.
Let $E_\a\subset\fM$ be the closure of the
unique component of $\nu^{-1}(\fM_0^{F_\a})$ 
that contains the relative core component $C_{\a}$.
Then $E_\a$ is cut out of $\fM$ by the equations 
\begin{equation}\label{Ea}
z_i = 0 \text{ for all  $i\in \cI_\a$ with $\a(i) = -$}
\and 
w_i = 0\text{ for all $i\in \cI_\a$ with $\a(i) = +$.}
\end{equation}
For example, if $F_\a = V_{0,\R}$, then $\cI_\a = \emptyset$, $\fM_0^{F_\a}$ is the dense leaf, 
and $E_\a = \fM$.
At the other extreme, if $F_\a = \{0\}$, then $\cI_\a = \{1,\ldots, n\}$, $\fM_0^{F_\a} = \{o\}$,
and $E_\a = C_{\a}$.

\begin{proposition}\label{left cells}
For all $\a,\b\in\cF_\qvb$, the following are equivalent:
\begin{enumerate}
\item $\a\Lleq \b$
\item $\cI_\b\subset \cI_\a$ (equivalently $F_\a\subset F_\b$)
and $\a$ and $\b$ agree on $\cI_\b$
\item $\overline{\DDelta_\a\cap\qvb}\subset\overline{\DDelta_\b\cap\qvb}$, where the bar denotes Zariski closure
\item $E_\a\subset E_\b$.
\end{enumerate}
\end{proposition}

\begin{proof}
The equivalence of (1) and (3) is proven 
by Musson and Van den Bergh \cite[7.3.1]{MvdB}.
The equivalence of (2) and (4) is manifest from Equation \eqref{Ea}.
To see that (2) and (3) are equivalent, note that 
$\overline{\DDelta_\a\cap\qvb}$ is equal to the preimage of the image of $\DDelta_\a\cap\qvb$
in $\V_\R/F_\a$, and the condition (2) is equivalent
to the condition that $\DDelta_\a\cap\qvb$ and $\DDelta_{\b}\cap\qvb$
have the same image in $\V_\R/F_\b$.
\end{proof}

\begin{corollary}
There is a bijection between the set of left cells of $\cF_{\qvb}$ and the set of 
compact sign vectors for various localizations of $\qH$.
\end{corollary}

\begin{proof}
A flat $F$ along with a compact sign vector $\b\in\{+,-\}^{\cI_F}$ indexes the left cell consisting 
of all $\a$ such that $F_\a = F$ and the restriction of $\a$ to $\cI_F$ is equal to $\b$.
\end{proof}

By a theorem of Ginzburg \cite[2.1]{Giprim}, the zero set in $\fM_0$
of the associated graded ideal 
$\gr \Ann L_\a^\qvb\subset \C[\fM_0]$ is the closure of a single symplectic leaf.  

\begin{proposition} [\mbox{\rm \cite[7.4.1]{MvdB}}]\label{goldie}
For any $\a\in\cF_\qvb$,
the zero set of $\gr\Ann L_\a^\qvb$ is equal to the closure of the leaf $\fM_0^{F_\a}$.
Furthermore, $U_\la/\Ann L^\qvb_\a$ is a quantization of the ring of square matrices with coefficients
in $\C[\fM_0^{F_\a}]$, and the
size of the matrices (the Goldie rank of $L_\a^\qvb$) 
is equal to the number of components of $\overline{\DDelta_\a\cap\qvb}$.
\end{proposition}

\begin{remark}
The number of components in the last part of Proposition \ref{goldie} may also be described
as the number of lattice points in the polytope $\DDelta_\b^{F_\a}$, where $\b\in\{+,-\}^{\cI_\a}$ 
is the compact sign vector that indexes the left cell in which $\a$ lies; indeed, $\overline{\DDelta_\a\cap\qvb}$
is equal to the preimage in $\V$ of this finite set of lattice points.
Thus, the Goldie rank polynomial of $L_\a^\qvb$ 
(which has been studied quite deeply in the Lie theoretic case; see, 
for example \cite{BBM89}) is equal to the Erhart polynomial of $\DDelta_\b^{F_\a}$.
\end{remark}

\subsection{Right cells}\label{sec:right cells}
Suppose that $\QPA = (\vb_0, \qvb, \xi)$ and $\QPA' = (\vb_0, \qvb', \xi)$ are two regular, integral,
quantized polarized arrangements that differ only in that 
$\qvb\neq \qvb'$.  Because $\QPA$ and $\QPA'$ are both integral, we have
$\cB_{\qvb,\xi} = \cB_{\qvb',\xi}$; we will denote this set as $\cB_\xi$ (see Remark \ref{quantum abuse}).
In this section we define right cells using a preorder on $\cB_\xi$.
Let $\la$, $\la'\colon Z(U) \to \C$ be the characters corresponding to $\qvb$,$\qvb'$, respectively.

\begin{definition}\label{HC}
We call a $(U_{\la'}, U_{\la})$-bimodule $B$ {\bf Harish-Chandra} if it is finitely generated 
over both $U_{\la'}$ and $U_{\la}$,
it has a filtration such that 
$\gr B$ is supported on the diagonal of $\fM_0\times\fM_0$, and it is the sum of its generalized weight spaces for the adjoint action of $\hat\xi$.
\end{definition}

We note that if $M$ is a finitely-generated $U$-module which is the 
sum of its weight spaces, then its weight spaces are finite dimensional.  
However, the analogous statement does not necessarily hold for bimodules 
considered with the adjoint action.  For example, 
the $0$-weight space of $U_\la$, regarded as a bimodule over itself, is $\bH_\la$.

\begin{example}
The bimodule ${_{\la'}T_{\la}}$ from Section \ref{translation functors}
is Harish-Chandra.
\end{example}

\begin{lemma}\label{translation preserves O} If $B$ is a Harish-Chandra $(U_{\la'}, U_{\la})$-bimodule, then
the functor $B\otimes_{U_\la} -$ sends objects of $\cO_\la$ to objects of $\cO_{\la'}$.
\end{lemma}

\begin{proof}
Let $M^{>k}$ denote the finite-dimensional subspace of $M$ consisting of $\hat{\xi}$-weight vectors of weight greater than $k$.  
Fix a $k$ such that $M^{>k}$ generates $M$ (such a $k$ must exist by finite generation of $M$).  
Let $\{b_1,\ldots,b_r\}\subset B$ be a finite set of adjoint weight vectors which generate $B$ both as a left and a right module,  let $p_i$ be the weight of $b_i$, and let $p=\min\{p_1,\ldots,p_r\}$.  
Let $$S:=\sum b_i\otimes M^{>k-p_i+p}\subset B\otimes_{U_\la}M.$$  
The set $S$ contains the tensor product of a left-generating
set for $B$ with a generating set for $M$, and therefore generates $B\otimes_{U_\la}M$.  Since the weight spaces of $M$
are finite dimensional and the weights that appear are bounded above (Lemma \ref{gamma-finite}), $S$ is finite dimensional.
If $u\in U^+$, then for all $i$, there exist $\{u_1,\ldots,u_r\}$ such that $u_j\in U^{\geq p_i-p_j}$ and $ub_i=\sum b_ju_j$.  
Thus, if $m\in  M^{>k-p_i+p}$, \[u\cdot b_i\otimes m=\sum b_j\otimes (u_jm)\in S.\]  Thus $S$ is $U^+$-invariant, 
and $B\otimes_{U_\la}M\in \cO_\la$ by Lemma \ref{gamma-finite}.
\end{proof}

Let $\a,\b\in\cB_{\xi}$ be bounded sign vectors.

\begin{definition}\label{right}
We write $\a\Rleq\b$ if $\qvb$ and $\qvb'$ may be chosen such that
$\a\in\cF_{\qvb'}$, $\b\in\cF_\qvb$, and there exists a Harish-Chandra $(U_{\la'}, U_{\la})$-bimodule
$B$ such that the simple $U_{\la'}$-module $L_\a^{\qvb'}$ is a composition factor in $B\otimes_{U_\la} L_\b^{\qvb}$.
We say that $\a$ and $\b$ lie in the same
{\bf right cell} of $\cB_\xi$ if $\a\Rleq \b$ and $\b\Rleq \a$.
\end{definition}

In order to give a characterization of right cells in $\cB_\xi$, we introduce and study a related class of $\D$-modules.
Given any $\a \in \cF_\qvb$ and any subset $A\subset I_\qvb = \{1, \dots, n\}$, define
$S_{\a, A}$ to be the quotient of $\D$ by the left ideal generated by
\begin{itemize}
\item $\bdy_i$ for all $i \in A$ with $\a(i) = +$,
\item $x_i$ for all $i \in A$ with $\a(i) = -$, and 
\item $h_i^{\a(i)}$ for all $1 \le i \le n$.  
\end{itemize} 
If $A = \{1, \dots, n\}$
we have $S_{\a,A} = L_\a$.  Moreover, following the notation of Section \ref{sec:std modules},
if $A$ is the set of $i$ for which $v_\xi \in H^{\a(i)}_i$, then $S_{\a, A}$ is
the module $S_\a$ defined in that section.  

Let $\DDelta_{\a,A} \subset \V_\R$
be the polyhedron defined by the same inequalities that cut out $\DDelta_\a$, but only for the
indices in $A$:
$$h^+_i \ge 0\text{ for all $i\in A$ with }\a(i) = + \and h^-_i \le 0 \text{ for all $i\in A$ with }\a(i) = -.$$
Then we have $\supp\left((S_{\a, A})^\qvb\right) = \DDelta_{\a,A} \cap \qvb$.
All of the weight spaces of $S_{\a,A}$ are one-dimensional.  It follows that
$S_{\a,A}$ is semisimple over $Z(U)$, and so 
$(S_{\a,A})^\qvb = \bigoplus_{v \in \qvb} (S_{\a, A})_v$.
For any $\b\in \cF_\qvb$, the simple module $L_\b^\qvb$ appears 
with non-zero multiplicity in $(S_{\a,A})^\qvb$ if and only if
$\DDelta_\b \subset \DDelta_{\a, A}$, which is equivalent to the condition $\b|_A = \a|_A$.

We say that an index $i \in A$ is
\textbf{active} for $\a$ and $A$ if there exists $\b\in \cF_\qvb$ with 
$\b(i) \ne \a(i)$ and $\b|_{A \setminus i} = \a|_{A \setminus i}$. 
In other words, $i$ is active if and only if the 
$i^\text{th}$ inequality is actually necessary to cut out 
the polyhedron $\DDelta_{\a,A}$.  
Let $A_\a\subset A$ be the set of indices that are active for $\a$ and $A$.
Thus we have $\DDelta_{\a,A_\a} = \DDelta_{\a,A}$,
and $A_\a$ is the smallest subset of $A$ with this property.

\begin{proposition}\label{translation}
There is an isomorphism
\[\D \otimes_U (S_{\a,A})^\qvb \cong S_{\a,A_\a}.\]
In particular, $L_\b$ is a composition factor of $\D \otimes_U (S_{\a,A})^\qvb$
if and only if $\b|_{A_\a} = \a|_{A_\a}$.
\end{proposition} 
\begin{proof}
The natural surjection $S_{\a,A_\a} \to S_{\a,A}$ induces a surjection
$(S_{\a,A_\a})^\qvb \to (S_{\a,A})^\qvb$.  In fact, this surjection is an isomorphism, 
since both sides have one-dimensional weight spaces and the same support.
Thus we need only show that the adjunction map
\[M :=\D \otimes_U (S_{\a,A_\a})^\qvb \to S_{\a,A_\a}\]
is an isomorphism.  It is certainly surjective, since $S_{\a,A_\a}$ 
is generated by its $v$-weight space for any $v$ in $\DDelta_\a \cap \qvb$,
and this set is nonempty since $\a\in \cF_\qvb$.  

Take
some element $\phi \otimes x \in M$ which 
maps to $1 \in S_{\a, A_\a}$;  we can assume that 
$x$ lies in a weight space
$(S_{\a,A_\a})_v$ for $v \in \DDelta_\a \cap \qvb$ and 
$\phi \in \D_{v_\a - v}$, where $v_\a$ is the weight of $1 \in S_{\a,A_\a}$.
It is easy to see that $\phi \otimes x$ generates $M$,
so we only need to show that it is killed by the defining ideal of $S_{\a,A_\a}$.

The generators $h^{\a(i)}_i$ are in the ideal $\mathfrak{I}_{v_\a}$, so
some power of these generators will kill $\phi\otimes x$.  
For the other two classes of generators, let
$i \in A_\a$.  We will assume that $\a(i) = +$ (the argument for $\a(i) = -$ is similar).
The weight of $\bdy_i(\phi \otimes v)$ is
$v_\b$, where $\b(j) = \a(j)$ for all $j\neq i$.  Suppose that $\bdy_i(\phi \otimes v)$ is
nonzero; then it follows that the simple module $L_\b$ appears in  
$M$ with non-zero multiplicity.  Since $i$ is an active index, we have $\b \in \cF_\vb$, 
and so $L_\b^\qvb$ has non-zero multiplicity in $M^\qvb = (S_{\a,A_\a})^\qvb$, contradicting the fact
that $i\in A_\a$.
\end{proof}

Next we give several equivalent formulations of the right preorder on $\cB_\xi$.
As in the case of the left preorder, we have four characterizations:  one from the definition,
another in terms of flats, a third in terms of polyhedra, and a fourth in terms of subvarieties
of a hypertoric variety.  (In this section we only need to consider the affine hypertoric variety $\fM_0$,
which is completely determined by $\vb_0$, thus we do not need to choose a polarized arrangement.)

\begin{proposition}\label{right-cells}
For all $\a,\b\in\cB_\xi$, the following are equivalent:
\begin{enumerate}
\item $\a\Rleq \b$
\item $\cI_\b\subset \cI_\a$ (equivalently $F_\a\subset F_\b$)
and $\a$ and $\b$ agree on the complement of $\cI_\a$
\item $\Delta_{0,\a}\subset \Delta_{0,\b}$
\item $C_{0,\a}\subset C_{0,\b}$.
\end{enumerate}
\end{proposition}

\begin{proof}
The fact that (4) implies (3)
follows from the fact that $\Delta_{0,\a}$ is the image of $C_{0,\a}$
along the real moment map introduced in \cite{BD}.
The fact that (1) implies (4) is immediate from Lemma \ref{support} and Definition \ref{HC}.
The equivalence of (2) and (3) is clear from the definition of the flat $F_\a$.
Thus we only need to show that (3) implies (1).  

Take $\a, \b \in \{+,-\}^n$, and 
suppose that $\Delta_{0,\a}\subset \Delta_{0,\b}$.
Let $J = \{i \mid \a(i) = \b(i)\}$.  We define an integral
character $\la$ of $Z(U)$ as follows.  Recall that, by choosing a parameter $\eta\in\SgrHZ/\vb_0$,
we obtain a hyperplane arrangement $\cH_\eta$ in the affine space $V_{\eta,\R}$.  If we identify $V_{\eta,\R}$
with $V_{0,\R}$ by choosing an origin (this is equivalent to choosing a lift of $\eta$ to $\SgrHZ\cong\Z^n$), 
then $\cH_\eta$ is obtained from $\cH_0$ by translating the $i^\text{th}$ hyperplane away from the origin by the 
$i^\text{th}$ coordinate of the lift of $\eta$.
Removing the inequalities indexed by
$i \notin J$ does not change the polyhedron $\Delta_{0,\b}$;
this means that a regular parameter $\eta$ may be chosen such that
$\Delta_{\eta,\b} = \Delta_{0,\b}$ but
$\Delta_{\eta,\gamma}$ is empty for any $\gamma \ne \b$
with $\gamma|_J = \b|_J$.  Choose a regular integral $\qvb$ which is linked to $\eta$, and let
$\la$ be the corresponding character of $Z(U)$.

Every index which is active for $\b$ and $\{1, \dots, n\}$
lies in $J$.  It follows from Proposition \ref{translation} that 
$L_\a$ is a composition factor of $\D \otimes_U L^\qvb_\b$, so 
$L^{\qvb'}_\a$ is a composition factor of $(\D \otimes_U L^\qvb_\b)^{\qvb'}$
for any integral $\qvb'$ with $\a \in \cF_{\qvb'}$.
By Lemma \ref{factoring translation} we have
$$(\D \otimes_U L^\qvb_\b)^{\qvb'} = (\D \otimes_U L^\qvb_\b)^{\la'} = {}_{\la'}T_\la \otimes_{U_\la} L^\qvb_\b,$$
where $\la'$ is the central character determined by $\qvb'$.  Thus $\a\Rleq \b$, as desired.
\end{proof}

\begin{corollary}
The right cells of $\cB_\xi$ are in bijection with the $\xi$-bounded faces of $\cH_0$.
\end{corollary}

\begin{proof}
The $\xi$-bounded faces of $\cH_0$ are exactly the polyhedra $\{\Delta_{\a,0}\mid\a\in\cB_{\xi}\}$.
\end{proof}

\subsection{Two-sided cells}
\begin{definition}
We define a third preorder on $\cP_{\qvb,\xi} = \cF_\qvb\cap\cB_\xi$ 
by putting $\a\Tleq\b$ if $\a\Lleq\b$ or $\a\Rleq\b$,
and then taking the transitive closure.  If $\a\Tleq\b$ and $\b\Tleq\a$, we say that $\a$
and $\b$ lie in the same {\bf two-sided cell} of $\cP_{\qvb,\xi}$.
\end{definition}

\begin{proposition}\label{two-sided cells}
Two bounded feasible sign vectors $\a,\b\in\cP_{\qvb,\xi}$ lie in the same two-sided cell
if and only if $\cI_\a = \cI_\b$ (equivalently $F_\a = F_\b$).
\end{proposition}

\begin{proof}
The only if statement is an immediate consequence of Propositions \ref{left cells}
and \ref{right-cells}.  To prove the if statement, we will show that the intersection
of the left cell containing $\a$ and the right cell containing $\b$ is nonempty.

Let $\cI = \cI_\a = \cI_\b$, and
let $\gamma$ be the sign vector that agrees with $\a$ on $\cI$ and with $\b$ on $\cI^c$.
Propositions \ref{left cells} and \ref{right-cells} tell us that $\gamma$ is in the same left cell
as $\a$ and the same right cell as $\b$ provided that we can show that $\gamma\in\cP_{\qvb,\xi}$.
That $\gamma$ is $\xi$-bounded follows from the fact that 
$\Delta_{0,\gamma} = \Delta_{0,\b}$.

Feasibility follows from Gale duality together with the same argument.  Let 
$\PA = (\vb_0,\eta,\xi)$ be a 
polarized arrangement linked to $\QPA$, and let $\PA^! = (\vb_0,\eta^!,\xi^!)$ be its Gale dual, as defined 
in Section \ref{sec:gd}.
We have $\a \in \cF_\qvb = \cF_\eta = \cB_{\xi^!}$, where the last equality is by Theorem \ref{bffb}.  
Then the boundedness argument  
from the previous paragraph allows us to conclude that $\gamma \in \cB_{\xi}$, and therefore 
$\gamma \in \cF_{\qvb}$, as required.
\end{proof}

We conclude the section by showing that Gale duality exchanges left cells with right cells.

\begin{theorem}\label{opposite preorders}
Suppose that $\QPA$ and $\QPA^!$ are Gale dual regular integral quantized polarized
arrangements. 
Then the two preorders on $\cF_\qvb = \cB_{\xi^!}$ are opposite to each other,
as are the preorders on $\cF_{\qvb^!} = \cB_\xi$ and on $\cP_{\qvb,\xi} = \cP_{\qvb^!,\xi^!}$.
\end{theorem}

\begin{proof}
This follows from Propositions \ref{left cells}, \ref{right-cells}, and \ref{two-sided cells},
along with the fact that Gale duality induces an order-reversing bijection of coloop-free flats.
\end{proof}

As a corollary, we find that left cells in $\cF_\qvb$ coincide with
right cells in $\cB_{\xi^!}$, and the partial order on the set of left cells is the opposite of the partial order on the set
of right cells, with similar statements holding for $\cF_{\qvb^!} = \cB_\xi$ and $\cP_{\qvb,\xi} = \cP_{\qvb^!,\xi^!}$.
We state this corollary explicitly for two-sided cells, since we will use it in the next section.

\begin{corollary}\label{ordered cells}
Let $\cP := \cP_{\qvb,\xi} = \cP_{\qvb^!,\xi^!}$.
The two-sided cells of $\cP$ induced by $\QPA$ are the same as the two-sided cells of $\cP$
induced by $\QPA^!$, and the partial order on the set of cells induced by $\QPA$ is opposite to the partial order
on the set of cells induced by $\QPA^!$.
\end{corollary}

\subsection{Cells and the BBD filtration}
Let $\PA$ be a polarized arrangement linked to $\QPA$.  Assume that 
$\vb_0$ is unimodular, so that Corollary \ref{linked-loc} provides an equivalence between the algebraic category $\cO(\QPA)$
and the geometric category $\cQ_\la$ of sheaves on the hypertoric variety $\fM = \fM(\PA)$.
In this section we use the two-sided cells to define a natural filtration of the Grothendieck group of $\cO(\QPA)$.  This filtration will be shown to coincide with the
Beilinson-Bernstein-Deligne (BBD) filtration of $\hmid$ via the cycle map of Section \ref{sec:grothendieck}.

Recall the projective morphism $\nu:\fM\longrightarrow \fM_0$ from $\fM$ to the affine hypertoric variety $\fM_0$.  
Since $\nu$ is semismall, the BBD decomposition theorem gives a canonical isomorphism
\begin{equation}\label{bbd}
\nu_*\C_\fM\,\,\cong\,\, \bigoplus_{F}\,\,\IC\!\left({\fM_0^F};\,\Pi_F\right),
\end{equation}
where  $\nu_*$ is the derived pushforward, $F$ ranges over all coloop-free flats of $\cH_0$,
$\Pi_F$ is the local system whose fiber over a point $x$ is the
top nonvanishing cohomology group of $\nu^{-1}(x)$,
and $\IC({\fM_0^F};\Pi_F)$ is the intersection cohomology sheaf of $\fM_0^F$ with coefficients
in $\Pi_F$  \cite[8.9.3]{CG97}.
(In fact, as shown in \cite[5.2]{PW07}, the local system $\Pi_F$ is always trivial.)
By applying the functor $H^{2d}_\bT(\cdot)$ to both sides of Equation \eqref{bbd},
we obtain the decomposition 
$$\hmid \,\,\cong\,\,\bigoplus_F\,\, I\! H^{2d}_\bT(\fM_0^F; \Pi_F).$$

We will be interested not in the full direct sum decomposition, but rather
in one of the two associated filtrations.  
More precisely, let
$$D_F := \bigoplus_{F\subset F'} I\! H^{2d}_\bT(\fM_0^{F'}; \Pi_{F'})\subset \hmid
\and
E_F := \bigoplus_{F\subsetneq F'} I\! H^{2d}_\bT(\fM_0^{F'}; \Pi_{F'})\subset D_F.$$
This filtration has a topological interpretation, which we prove (for arbitrary equivariant symplectic resolutions)
in \cite{BLPWgco}.

\begin{theorem}\label{filtration}
For each coloop-free flat $F$ of $\cH_0$, $D_F$
is equal to the intersection 
$$\bigcap_{\{\a:F\not\subset F_\a\}} [C_\a]^\perp,$$ 
where $[C_\a]^\perp$ is the perpendicular space to $[C_\a]$ 
with respect to the integration 
pairing on $\hmid$.
\end{theorem}

For all $\a\in\cP_{\qvb,\xi}$, let $P_\a^\qvb$ be the projective cover of $L_\a^\qvb$ in $\cO(\QPA)$.

\begin{theorem}\label{CaCo}
We have
$$D_F = \C\left\{\,\supp[\Loc(P_\a^\qvb)]\mid F\subset F_\a\right\}
\and E_F = \C\left\{\,\supp[\Loc(P_\a^\qvb)]\mid F\subsetneq F_\a\right\}.$$
In other words, the BBD filtration of $\hmid$ corresponds, via the support
isomorphism, to the filtration of $K(\cO(\QPA))_\C$ defined by the basis of projective objects and 
the order on two-sided cells.
\end{theorem}

\begin{proof}
We have 
$$\C\left\{\,\supp[\Loc(P_\a^\qvb)]\mid F\subset F_\a\right\} = 
\C\left\{\,\supp[\Loc(L_\a^\qvb)]\mid F\not\subset F_\a\right\}^\perp
= \C\left\{\,[C_\a]\mid F\not\subset F_\a\right\}^\perp$$ by Proposition \ref{forms}
and the fact that the classes of simple and projective objects are dual with respect to the Euler form.
The first statement of the theorem then follows
from Theorem \ref{filtration}.  
The statement about $E_F$ is proven in the same way.  The interpretation in terms of two-sided
cells follows from Proposition \ref{two-sided cells}.
\end{proof}

For all $\a\in\{\plusminus\}^n$, let $d_\a := |\{i\mid \a(i)=-1\}|$.
Let $\PA^!$ be the Gale dual of $\PA$, and let $\QPA$ be a regular, integral, quantized polarized arrangement
linked to $\PA^!$.  Consider the pairing $$K(\cO(\QPA))_\C\otimes K(\cO(\QPA^!))_\C\to\C$$
given by the formula $$\left\langle [P^\qvb_\a], [P^{\qvb^!}_{\b}]\right\rangle = (-1)^{d_\a}\delta_{\a\b}$$
for all $\a,\b\in\cP_{\qvb,\xi} = \cP_{\qvb^!,\xi^!}$.
Using the support isomorphism from Section \ref{sec:grothendieck}, we may interpret this pairing as a pairing between
$H^{2d}_{\bT}\!\left(\fM(\PA); \C\right)$ and $H^{2d^!}_{\bT}\!\!\left(\fM(\PA^!); \C\right)$, where $d^! := \dim V_0^!$.
For each coloop-free flat $F$ of $\cH_0$, $F^c$ is a coloop-free flat of $\cH_0^!$, and we denote by
$$E^!_{F^c}\subset D^!_{F^c}\subset H^{2d^!}_{\bT}\!\!\left(\fM(\PA^!); \C\right)$$
the corresponding pieces of the BBD filtration.
Theorem \ref{CaCo} has the following immediate corollary.

\begin{corollary}\label{dual filtrations}
The BBD filtrations of $H^{2d}_{\bT}\!\left(\fM(\PA); \C\right)$ and $H^{2d^!}_{\bT}\!\!\left(\fM(\PA^!); \C\right)$
are dual to each other under the above pairing.  More precisely,
for any coloop-free flat $F$ of $\cH_0$,
$E^!_{F^c}$ is the perpendicular space to $D_F$ and $D^!_{F^c}$ is the perpendicular space to $E_F$.
In particular, the pairing induces a canonical duality between
$$I\! H^{2d}_\bT\!\left(\fM_0(\PA)^F; \Pi_F\right)\cong D_F/E_F \and
I\! H^{2d^!}_\bT\!\!\left(\fM_0(\PA^!)^{F^c}; \Pi^!_{F^c}\right)\cong D^!_{F^c}/E^!_{F^c}.$$
\end{corollary}

\begin{remark}
Corollary \ref{dual filtrations} would still be true if we did not include the twist of $(-1)^{d_\a}$ in the definition
of our pairing.  We regard the twisted pairing as more natural than the untwisted pairing because the twisted pairing has
the property that
$$\left\langle [L^\qvb_\a], [L^{\qvb^!}_{\b}]\right\rangle =
\left\langle [S^\qvb_\a], [S^{\qvb^!}_{\b}]\right\rangle =
\left\langle [P^\qvb_\a], [P^{\qvb^!}_{\b}]\right\rangle = (-1)^{d_\a}\delta_{\a\b}$$
for all $\a,\b\in\cP_{\qvb,\xi} = \cP_{\qvb^!,\xi^!}$.
If the twist were not included, then the pairing would not be well-behaved with respect to the simple or standard bases.
\end{remark}

\begin{remark}\label{matroids}
For any coloop-free flat $F$, the dimension of $I\! H^{2d}_\bT(\fM_0^F; \Pi_F)$ is equal
to the dimension of $I\! H^*(\fM_0^F; \Pi_F)$.
In fact, the full equivariant cohomology group $I\! H^{*}_\bT(\fM_0^F; \Pi_F)$ is a flat family over 
$\Spec H^*_{\bT}(pt)\cong\bt$ with general fiber isomorphic to $I\! H^{2d}_\bT(\fM_0^F; \Pi_F)$
and special fiber isomorphic to $I\! H^*(\fM_0^F; \Pi_F)$.  Thus Corollary \ref{dual filtrations}
tells us that for every coloop-free flat $F$, a deformation of the vector space $I\! H^{*}\!\!\left(\fM_0(\PA)^F; \Pi_F\right)$
is canonically dual to a deformation of
$I\! H^{*}\!\!\left(\fM_0(\PA^!)^{F^c}; \Pi^!_{F^c}\right)$.

If we take $F=\{0\}$ to be the minimal flat, this says that a deformation of
$$I\! H^{*}\!\!\left(o; \Pi_{\{0\}}\right) = H^{2d}\!\left(\nu^{-1}(o); \C\right) = H^{2d}(\fM; \C)$$
is canonically dual to a deformation of $I\! H^{*}\!\!\left(\fM_0(\PA^!); \C\right)$.
In light of the combinatorial interpretation of the ordinary and intersection Betti numbers
of a hypertoric variety \cite[3.5 \& 4.3]{PW07},
this is a geometric interpretation of the well-known combinatorial statement that the top $h$-number of a matroid
is equal to the sum of the $h$-numbers of the broken circuit complex of the dual matroid.
\end{remark}

\section{Shuffling and twisting functors}\label{pione}
In this section we define functors between the derived categories of regular 
integral blocks of $\cO$ obtained by fixing $\vb_0$ and 
varying $\qvb$ and $\xi$.  These functors are analogous to the 
shuffling and twisting functors on BGG category $\cO$, and like
those functors, we show that they are equivalences, and that they 
produce a categorical action of the fundamental groupoid of the complement
of a certain complex hyperplane arrangement, analogous to the
braid group actions on the derived category of 
BGG category $\cO$ provided by shuffling and twisting functors.
Finally we show that Koszul duality interchanges shuffling functors
with twisting functors, again in analogy with the BGG category $\cO$.

\subsection{Weyl group symmetries}\label{sec:weyl-group-symm}
Fix $n$ and a direct summand $\vb_0$ of the lattice $W_\Z$.  As in Section
\ref{sec:algdef}, this determines a subtorus $K \subset T$ with Lie algebra $\fk = (\C \cdot \vb_0)^\bot$.
We begin by defining two groups of symmetries of the hypertoric enveloping algebra $U = \D^K$
which are analogous to the Weyl group in the study of the BGG category $\cO$ for a semisimple Lie algebra.

Consider the natural action of $G = Sp(2n,\C)$ on the symplectic vector space $T^*\C^n$ and the
induced action on its quantization $\D$.  The actions of $K$ and $T$ on $T^*\C^n$ and $\D$ 
factor through the action of $G$, giving an identification of 
$K$ and $T$ with subgroups of $G$, under which $T$ is a maximal torus.

If $g \in G$ normalizes $K$, then the action of $g$ on $\D$ preserves $U$, and pushing forward by this action gives
an endofunctor $U\mmod \to U\mmod$.  If in addition $g$ normalizes $T$, then this restricts to an endofunctor of 
the category $\UWM$ of weight modules.  In other words, we have an action of 
$N_G(K) \cap N_G(T)$ on $\UWM$.  Moreover, the subgroup $T \subset N_G(K) \cap N_G(T)$ acts trivially, 
in the sense that the pushforward by the endomorphism of $U$ induced by $t\in T$ is naturally isomorphic 
to the identity functor $\UWM \to \UWM$.  
To see this, note that since $t$ acts trivially on $\bH=\D_0$, it preserves each summand
 $\Ulmod_\qvb$ of $\UWM$.  For $M \in \Ulmod_\qvb$, 
the required natural isomorphism between $M = \bigoplus_{v \in \qvb} M_v$ and the $t$-twist of $M$
can be given by multiplication by $v(t)$ on the summand $M_v$, where we choose any isomorphism
$\qvb \cong \vb_0$ of $\vb_0$-torsors, and identify elements of $\vb_0 \subset W_\Z$ with characters
of $T$ via the exponential map.
 
Thus we obtain a weak action of the quotient $(N_G(K) \cap N_G(T))/T$ on the category of 
weight modules.  This group is a
subgroup of the Weyl group $S_n \ltimes \{\pm 1\}^n$ of $G$; more concretely it is the 
group of elements which preserve $\vb_0$ under the obvious action on $W_\Z \cong \Z^n$.
It contains two natural subgroups: the subgroup $\rWeyl$ of elements which 
fix $\vb_0$ pointwise and the subgroup $\Weyl$ of elements which act trivially
on $W_\Z/\vb_0$.  

The group $\Weyl$ acts trivially on the center of $U$, so it acts on the category
of weight modules for any central quotient $U_\la$.  It acts nontrivially on
$\vb_0$, however, so it goes between our hypertoric category $\cO$ for different 
parameters: $s \in \Weyl$ gives an equivalence between $\cO(\vb_0, \qvb, \xi)$ and
$\cO(\vb_0, \qvb, s\cdot \xi)$.  In contrast, an element $s \in \rWeyl$ acts trivially on
$\vb_0$ but may act non-trivially on the space $\SH/V_0$ of central characters of 
$U$.  It therefore gives an equivalence between $\cO_\lambda$ and $\cO_{s \cdot \lambda}$ 
for any central character $\lambda$, and, restricting further to infinitesimal 
blocks, an equivalence between $\cO(\vb_0,\qvb,\xi)$ and $\cO(\vb_0, s\cdot \qvb, \xi)$.

It is clear that these two symmetries are interchanged by Gale duality on 
regular integral blocks of category $\cO$: we have canonical isomorphisms 
$\Weyl(\QPA) \cong \rWeyl(\QPA^!)$ and $\rWeyl(\QPA) \cong \Weyl(\QPA^!)$.

We can describe the structure of the groups $\Weyl$ and $\rWeyl$ more 
explicitly as follows.  Recall that we are assuming that $\vb_0$ doesn't
contain any coordinate axis.  This implies that the weights of $K$ acting on $T^*\C^n$ are
all nonzero and appear in pairs $\pm \lambda$.  Then we have
\[\Weyl \cong \prod_{\{\pm\lambda\} \in \Xi} S_{m_\lambda},  \]
where $\Xi$ is the quotient of the set of nonzero elements of the character lattice
$W_\Z/\vb_0$ of $K$ by the action of $\{\pm 1\}$, and $m_\lambda$ is the
multiplicity of the character $\la$ in the action of $K$ on $T^*\C^n$.

On the other hand, our assumption that $h_i(\vb_0) \ne 0$ for all $i$
implies that $\rWeyl$ is also a product of symmetric groups.  
Elements of $\rWeyl$ correspond to permutations $s$ such that
$h_i|_{\vb_0} = \pm h_{s(i)}|_{\vb_0}$ for all $i$.  

\subsection{Defining the functors}\label{sec:defining}
We continue to fix a direct summand $\vb_0\subset\SgrHZ$.
For any pair of parameters $\qvb$ and $\xi$, let $$\QPA_{\qvb,\xi} :=
(\vb_0,\qvb,\xi)$$ be the corresponding
quantized polarized arrangement.  We will denote by $D(\QPA_{\qvb,\xi})$ the bounded derived category of the associated block $\cO(\QPA_{\qvb,\xi})$ of hypertoric category $\cO$.
We introduce subscripts in the notation because we intend to vary the parameters $\qvb$ and $\xi$,
restricting our attention to those parameters for which $\QPA_{\qvb,\xi}$ is both regular and integral.

We begin by varying only $\qvb$.
Consider a pair of regular integral parameters $\qvb$ and $\qvb'$, and let $\la$ and $\la'$ be the associated central characters
of the hypertoric enveloping algebra.  Fix a regular covector $\xi$,
and consider the functor
$$\Phi_{\qvb}^{\qvb'}:D(\QPA_{\qvb,\xi})\to D(\QPA_{\qvb',\xi})$$
obtained as the derived functor of the translation functor ${}_{\la'} T_{\la} \otimes_{U_{\la}} -$ in Lemma \ref{factoring translation}.

Furthermore, for $s\in \rWeyl$, let $\Phi^s\colon
D(\QPA_{\qvb,\xi})\to D(\QPA_{s\cdot \qvb,\xi})$ denote the 
equivalence induced by pushing forward by the action of $s$ 
as described in the last section.   As
we noted in that section, this functor is well-defined up to a 
natural isomorphism.

\begin{definition}\label{def:twisting}
A {\bf pure twisting functor} is any functor
$D(\QPA_{\qvb,\xi})\to D(\QPA_{\qvb',\xi})$ which is obtained as a composition
$$\Phi_{\qvb_r}^{\qvb'}\circ\Phi_{\qvb_{r-1}}^{\qvb_r}\circ\ldots\circ\Phi_{\qvb_1}^{\qvb_2}\circ\Phi_{\qvb}^{\qvb_1}$$
for any finite sequence $\qvb_1,\ldots,\qvb_r$ of regular integral
parameters.  A {\bf twisting functor} is a functor given by a
composition of pure twisting functors and functors $\Phi^s$.
\end{definition}

The following lemma tells us that pure twists within an equivalence class are trivial.

\begin{lemma}\label{twist-equiv}
Suppose that $\qvb$, $\qvb'$, and $\qvb''$ are chosen such that the quantized polarized arrangements
$\QPA_{\qvb,\xi}$, $\QPA_{\qvb',\xi}$, and $\QPA_{\qvb'',\xi}$ are all equivalent in the sense of Definition
\ref{def:equivalence}.  Then the natural transformation 
$$\Phi_{\qvb'}^{\qvb''}\circ\Phi_{\qvb}^{\qvb'}\to\Phi_{\qvb}^{\qvb''}$$
provided by Equation \eqref{bimodule map} is a natural isomorphism. 
\end{lemma}

\begin{proof}
For each $\a\in \cP_{\qvb,\xi}=\cP_{\qvb'\!\!,\xi}=\cP_{\qvb''\!\!,\xi}$, consider the projective cover $P_\a^\qvb$ of $L_\a^\qvb$ in $\cO(\QPA_{\qvb, \xi})$.  
The $\D$-module $\D\otimes_{U} P_\a^\qvb$ is the quotient of the module $\Pk_\a$ from Section \ref{truncation} by the 
submodule generated by the weight spaces corresponding to all weights 
$v \in \DDelta_\b \cap \qvb$ for all $\b \in \cF_\qvb\smallsetminus\cP_{\qvb,\xi}$.
This submodule is independent (up to isomorphism) of the choice of $\qvb$ in a fixed equivalence class, so it follows that
\[{_{\la'} T_{\la}}\otimes_{U_\la} P_\a^\qvb \cong \left(\D \otimes_U P_\a^\qvb\right)^{\qvb'}
\cong \left(\D \otimes_U P_\a^{\qvb'}\right)^{\qvb'} \cong P_\a^{\qvb'},\] 
and similarly for $\qvb''$ and $\la''$.

Thus the functors given by tensor product with ${_{\la''} T_{\la'}}\otimes{_{\la'} T_{\la}}$ and ${_{\la''} T_{\la}}$
are both equivalences.  To show that they are the {\em same} equivalence, it is sufficient to show that,
for all $\a,\b\in\cP_{\qvb,\xi}=\cP_{\qvb''\!\!,\xi}$, they induce the same isomorphism between
$\Hom_{U_\la}\!\!\left(P_\a^\qvb,P_\b^\qvb\right)$ and $\Hom_{U_{\la''}}\!\!\left(P_\a^{\qvb''},P_\b^{\qvb''}\right)$.
This follows from the fact that, in both cases, the isomorphism is compatible with the surjections
from $\Hom_\D\!\!\left(\Pk_\a, \Pk_\b\right)$ to the two groups.
\end{proof}

\begin{remark}\label{arkhipov}
These functors are analogous to the derived functors of Arkhipov's twisting functors
\cite{Arktwist,AS}, which give an action of the braid group of $\mg$ on the derived category of 
BGG category $\cO(\mg)$ (see \cite[2.1]{AS} or \cite{KM}). 
The pure twisting functors on hypertoric category $\cO$ are
analogous to the endofunctors defined by pure braids, i.e.\ braids where each strand is
sent to itself.

For example, consider the arrangement of $n$ points in a line. 
More precisely, let $\vb_0$ be as in Example \ref{next example}.
Then for any regular integral parameters $\qvb$ and $\xi$, the category 
$\cO = \cO(\QPA_{\qvb,\xi})$ is equivalent to a singular integral block
$\cO_\la \subset \cO(\mathfrak{sl}_n)$, where the stabilizer of 
$\la$ under the dot-action is an $S_{n-1}$ generated by $n-2$ 
adjacent transpositions in $S_n$.  It is equivalent to the category of
representations of the quiver
\begin{equation}\label{singular sln quiver}
\xy
(-35,0)*++{V_1}="1"; (-15,0)*++{V_2}="2";(5,0)*++{\;}="3l";(10,0)*++{...}="3";(15,0)*++{\;}="3r";(35, 0)*++{V_n}="4";
{\ar@/^/^{c_1} "1";"2"}; {\ar@/^/^{d_1} "2";"1"};
{\ar@/^/^{c_2} "2";"3l"}; {\ar@/^/^{d_2} "3l";"2"};
{\ar@/^/^{c_{n-1}} "3r";"4"}; {\ar@/^/^{d_{n-1}} "4";"3r"};
\endxy
\end{equation}
satisfying the relations $d_1c_1 = 0$, and $c_id_i = d_{i+1}c_{i+1}$ for $1 \le i \le n-2$.

There are $n!$ equivalence classes of regular integral parameters $\qvb$, and 
the Weyl group $\rWeyl$ acts transitively on them.  For each $1 \le i \le n-1$, 
we get an impure twisting functor by allowing the $i$th and $(i+1)$st points of
the arrangement to swap places, and using the appropriate $\Phi^s$ to 
get back to the original category.  It is easy to see using 
Theorem \ref{Translation and quivers} that in terms of the quiver \eqref{singular sln quiver}
this twisting functor is the derived functor of the functor $F_i$ which sends
$(\{V_j\}, \{c_j\}, \{d_j\})$ to $(\{V'_j\}, \{c'_j\}, \{d'_j\})$ given as follows.
Put $V'_j = V_j$ if $j \ne i$ and let $c'_j = c_j$ and $d'_j = d_j$ for 
$j\notin\{i-1, i\}$.  Let $V_i$ be the cokernel of 
$d_{i-1}\oplus (-c_i)\colon V_i \to V_{i-1}\oplus V_{i+1}$, 
let $c'_{i-1}$ and $d'_i$ be the natural quotient maps, and
let $d'_{i-1} \oplus c'_i \colon V'_i \to V_{i-1} \oplus V_{i+1}$ 
be induced from the automorphism $d_ic_i \oplus c_{i+1}d_{i+1}$
of $V_{i-1} \oplus V_{i+1}$.  If $i=1$, the same formula applies
if we put $V_0 = 0$.

There is a natural transformation $F_i \to \id_\cO$ which is
the identity map $V'_j = V_j \to V_j$ if $j \ne i$ and 
in the $i$th place is induced by the map $V_{i-1}\oplus V_{i+1} \to V_i$ 
given by $c_{i-1}$ and $d_i$.  If $P_j$ denotes the projective 
cover of the simple object $L_j$ supported at the
$i$th vertex, then this natural transformation is an isomorphism  
when applied to $P_j$ for $j\ne i$.  For $j=i$ it 
 gives an exact sequence
$0 \to F_iP_i \to P_i \to L_i \to 0$.
The argument of \cite[Theorem 10]{KM} now shows that
$F_i$ can be identified with Arkhipov's twisting functor 
on $\cO_\la$ for the simple reflection $s_i \in S_n$.
\end{remark}

\begin{remark}
Note that,
unlike in the BGG case, not all equivalence classes of parameters $\qvb$ 
need be conjugate under $\rWeyl$.  
In fact, for many hyperplane arrangements, the groups $\rWeyl$ and
$\Weyl$ are trivial, so the only twisting functors are pure ones, 
but there are still many different equivalence classes
of quantized polarized arrangements.
\end{remark}

We now define functors analogous to Irving's shuffling functors on BGG category $\cO$.  
Fix a regular integral parameter $\qvb$ and let $\xi$ and $\xi'$ be regular.  Let $\proWM$ be the category of topologically finitely generated $U_\la$-modules with profinite dimensional weight spaces, that is, the category of $U_\la$-modules which are inverse limits of weight modules such that all modules in the inverse limit have a consistent finite generating set. This category contains the objects representing each weight space, and thus has enough projectives.  Furthermore, the functor $F$ arising from Theorem \ref{first-equiv} gives an equivalence from $\Ulmod_{pro,\qvb}$ to
the category of finitely generated modules over $e_{\qvb}Re_{\qvb}$.
As we will prove in Lemma \ref{Akoszul}, this category has finite global dimension.\footnote{
That theorem concerns an uncompleted version of $e_{\qvb}Re_{\qvb}$, which is sufficient for our purposes because the completion of a projective resolution is still a projective resolution.}

Consider the functor $$\Psi_\xi^{\xi'}:D(\QPA_{\qvb,\xi})\to D(\QPA_{\qvb,\xi'})$$ obtained as the composition
of the derived functors 
$$D(\cO(\QPA_{\qvb,\xi}))\overset{\iota}\longrightarrow D(\Ulmod_{pro,\qvb})\overset{L\pi_{\xi'}}\longrightarrow D(\cO(\QPA_{\qvb,\xi'})),$$
where $\iota$ is the inclusion and $\pi_{\xi'}$ is the projection
functor from Section \ref{truncation}.  The functor $\iota$ is exact,
but it does not send projectives to projectives.  As a result, the
composition of the above derived functors is not the same as the
derived functor of the composition.  The left derived functor
$L\pi_{\xi'}$ preserves bounded derived categories since all the
categories involved have finite global dimension.

In addition, any $s\in \Weyl$ induces an equivalence of categories
$\Psi^s:D(\QPA_{\qvb,\xi})\to D(\QPA_{\qvb,s\cdot\xi})$, 
as we described in Section \ref{sec:weyl-group-symm}.

\begin{definition}\label{def:shuffling}
A {\bf pure shuffling functor} is any functor $D(\QPA_{\qvb,\xi})\to D(\QPA_{\qvb,\xi'})$ obtained as a composition
$$\Psi_{\xi_r}^{\xi'}\circ\Psi_{\xi_{r-1}}^{\xi_r}\circ\ldots\circ\Psi_{\xi_1}^{\xi_2}\circ\Psi_{\xi}^{\xi_1}$$
for any finite sequence $\xi_1,\ldots,\xi_r$ of regular parameters. A
{\bf shuffling functor} is a composition of pure shuffling functors and
functors $\Psi^s$ for $s\in \Weyl$.
\end{definition}

If $\QPA_{\qvb,\xi}$ and $\QPA_{\qvb,\xi'}$ are equivalent in the
sense of Definition \ref{def:equivalence}, then the categories $\cO(\QPA_{\qvb,\xi})$ and $\cO(\QPA_{\qvb,\xi'})$
are in fact equal and $\Psi_\xi^{\xi'}$ is the identity functor.

\begin{remark}\label{twisting remark} These functors
are roughly analogous to Irving's shuffling functors on BGG category $\cO(\mg)$
\cite{Irvshuf}.  More precisely, they are analogous to shifts of
derived functors of \emph{co}shuffling functors, which 
are the right adjoints of shuffling functors. 

For example, let $\vb_0$ be as in Example \ref{big example} 
for $n=2$.  For regular integral parameters $\qvb$ and
$\xi$ the category $\cO(\QPA_{\qvb,\xi})$ is equivalent
to a regular block of BGG category $\cO$ for $\mathfrak{sl}_2$.
There are only two equivalence classes of regular covectors in $\vb_0^*$, 
represented, say, by $\xi$ and $-\xi$.  There is a unique non-identity element 
$s \in \Weyl$, and we have $s\cdot \xi = -\xi$.
Then the functor 
$H^{-1}(\Psi^s\Psi_{\xi}^{-\xi})$ is identified with the coshuffling 
functor associated to the unique simple reflection $s$, which is
the kernel of the adjunction $\theta_s \to \id$ between 
the wall-crossing functor $\theta_s$ and the identity.
\end{remark}

\begin{remark}
The fact that shuffling and twisting functors are equivalences is non-trivial; it 
will follow from Proposition \ref{twisting-comb} 
(for shuffling) and Theorem \ref{TwistShuf Koszul} (for twisting).
\end{remark}

\begin{remark}  It may seem surprising that our twisting functors involve changing
the parameter $\qvb$ while shuffling involves changing $\xi$, 
since Arkhipov's twisting functors are related to
changing the Borel subalgebra and shuffling functors are related
to projective functors, which change the central character.
The resolution to this apparent paradox comes from the fact that
our category is analogous to $\mg$-modules which are $U(\mh)$-locally
finite and have an honest central character.  Soergel's equivalence
\cite{Soe86}
between a regular integral block of this category and a regular 
integral block of the usual BGG category $\cO$ involves switching the
left and right actions on a suitable category of $U(\mg)$-$U(\mg)$ 
bimodules.  Well-known functors such as wall-crossing functors
acting on one side of these bimodules can have unexpected
effects under this isomorphism.  See \cite[5.26]{BLPWquant} for a discussion of this phenomenon
in a more general context.
\end{remark}

\subsection{Combinatorial interpretations}\label{sec:comb}
In this section we give combinatorial interpretations
of shuffling and twisting functors, and we use them to prove that shuffling and twisting commute.  
For several reasons it will be easier to index our categories by polarized arrangements
rather than by quantized polarized arrangements.  
We lose no information with this choice, since shuffling and twisting functors that stay
within an equivalence class are trivial and linkage gives a bijection between equivalence classes
of regular, integral, quantized polarized arrangements and regular polarized arrangements with the same $\vb_0$.
Thus we let $D(\PA_{\eta,\xi})$ denote the bounded derived category of modules over the algebra 
$A(\eta,\xi):=A(\PA_{\eta,\xi})$ from \cite{GDKD}, and use the notation
$$\Phi_{\eta}^{\eta'}:D(\PA_{\eta,\xi})\to D(\PA_{\eta',\xi})\and
\Psi_{\xi}^{\xi'}:D(\PA_{\eta,\xi})\to D(\PA_{\eta,\xi'}),$$
for the functors obtained from those in the previous section via Theorems \ref{alg=comb} and \ref{blpw}.

Following the notation of Theorem \ref{Translation and quivers},
let $R := \wh{Q}_n/ \langle \vart(\fk)\rangle$ (since everything in sight
is integral we write $\wh{Q}_n$ rather than $\wh{Q}_\qvb$), and put 
\[e_\eta := \sum_{\a \in \cF_\eta} e_\a \and e_\xi := \sum_{\b \notin \cB_\xi} e_\b.\]
We define algebras
$$A(\eta, -) := e_\eta R e_\eta\and A(-,\xi) := R/R e_\xi R,$$ so that
we have
\[e_\eta A(-, \xi) e_\eta = A(\eta,\xi) = A(\eta,-)/A(\eta,-)e_\xi A(\eta,-).\]

\begin{lemma} \label{Cartesian} Multiplication induces an isomorphism
\[ R e_\eta \otimes_{A(\eta,-)} A(\eta,\xi) \stackrel\sim\longrightarrow A(-,\xi)e_\eta.\]
\end{lemma}
\begin{proof}
Clearly the multiplication map is surjective.  To show injectivity it
is enough to show that \[R e_\xi R e_\eta \otimes _{A(\eta,-)} A(\eta,\xi) = 0.\]
This is an immediate consequence of the following fact, which says that any path 
in the quiver algebra $\wh{Q}_n$ from
a feasible sign vector to an unbounded one is equivalent to 
one which first goes through feasible sign vectors to one which is
both feasible and unbounded:   for any $\a \in \cF_\eta$ and
any $\b \in \{+,-\}^n \setminus \cB_\xi$, there exists $\gamma \in \cF_\eta$
so that $\gamma\notin \cB_\xi$ and 
\begin{equation*}\tag{*}
\text{for any $1\le i \le n$, either $\gamma(i) = \a(i)$ or $\gamma(i) = \b(i)$.}
\end{equation*}

To see that this fact holds, we drop the integrality assumption on $\eta$
and allow $\eta \in \SgrH_\R/V_{0, \R}$.
The regularity of $\eta$ implies that there is an open ball $B$ with center
$\eta$ so that $\cF_{\eta'} = \cF_\eta$ for any $\eta'\in B$.  
Take $\eta''$ for which $\b \in \cF_{\eta''}$, and choose points $p \in \Int(\Delta_{\eta,\a})$
and $p''\in \Int(\Delta_{\eta'',\b})$.  If a point $p' \in \Delta_{\eta',\gamma}$ lies on the
line segment joining $p$ and $p''$, then $\gamma$ satisfies (*), since the line segment 
cannot cross a hyperplane more than once.  We can take $p''$ so that $\xi(p'')$ is arbitrarily large (keeping $\eta''$ fixed),
so we can find $p'$ so that $\eta' \in B$ but $\xi(p')$ is as large as we like,
giving $\gamma \in \cF_{\eta'} = \cF_\eta$ with $\gamma \notin \cB_\xi$.
\end{proof}

\begin{proposition}\label{twisting-comb}
The functor $\Phi_{\eta}^{\eta'}:D(\PA_{\eta,\xi})\to D(\PA_{\eta',\xi})$ is given by 
$$\Phi_{\eta}^{\eta'}(M) = e_{\eta'} A(-,\xi) e^{}_\eta\, \stackrel{L}\otimes_{A(\eta,\xi)} M.$$
In particular, $\Phi_{\eta}^{\eta'}$ is an equivalence.
\end{proposition}

\begin{proof}
Theorem \ref{Translation and quivers} and Lemma \ref{Cartesian}
imply that the diagram
\[\tikz[xscale=3.5,yscale=2, thick]{
\node (aa)  at (0,1) {$ \cO(\QPA) $};
\node (ab) at (1,1) {$ A(\eta,\xi)\mmod $};
\node (ba) at (0,0) {$ \cO(\QPA') $};
\node (bb) at (1,0) {$ A(\eta', \xi)\mmod $};
\draw[->] (aa) -- (ba) node [left,midway] {${}_{\la'}T_\la \,\otimes_{U_\la} - $} ;
\draw[->] (ab) -- (bb) node [right,midway] {$e_{\eta'} A(-,\xi) e^{}_\eta\, \otimes_{A(\eta,\xi)} -$};
\draw[->] (aa) -- (ab) node [above,midway] {$F$};
\draw[->] (ba) -- (bb) node [below,midway] {$F$};
}
\]
commutes up to natural isomorphism, where $F$ and $F'$ are the equivalences of Theorems \ref{alg=comb} and \ref{blpw}.
The Proposition then follows from passing to derived functors.

These functors are exactly those studied in \cite[\S 6]{GDKD}, which were proven to be equivalences in Theorem 6.13 of that paper.  In that paper we considered the derived category of graded modules, but the corresponding functors (given by the same bimodules) on ungraded modules are still triangulated and fully faithful on gradable modules, and these generate the category.  Thus these functors are equivalences on the ungraded categories.  
\end{proof}

We have a similar description of the functor $\Phi^s$: the element $s$
induces an automorphism of the set of sign vectors by permuting the
appropriate coordinates.  This preserves boundedness for every $\xi$
and sends $\cF_\eta$ to $\cF_{s\cdot \eta}$.  Furthermore, it
preserves the complement of the hyperplane arrangement (as a set) and
thus induces an isomorphism $\phi^s\colon
A(\eta,\xi)\overset{\sim}\longrightarrow A(\eta',\xi)$.

\begin{proposition}\label{twisting-permute}
  The functor $\Phi^s$ is given by pushforward by the isomorphism $\phi^s$.
\end{proposition}

Next we turn to the shuffling functors.

\begin{proposition}\label{shuffling-comb}
The functor $\Psi_{\xi}^{\xi'}:D(\PA_{\eta,\xi})\to D(\PA_{\eta,\xi'})$ is given by
$$\Psi_{\xi}^{\xi'}(M) = A(\eta,\xi') \stackrel{L}\otimes_{A(\eta,-)} M.$$
\end{proposition}

\begin{proof}
As before, we need only draw the commutative diagram
\[\tikz[xscale=4.5,yscale=2, thick]{
\node (aa)  at (0,1) {$ D(\PA_{\eta,\xi}) $};
\node (ab) at (1,1) {$ D(A(\eta,\xi)\mmod) $};
\node (ba) at (0,0) {$ D(\Ulmod_{pro,\qvb}) $};
\node (bb) at (1,0) {$ D(A(\eta,-)\mmod )$};
\node (ca) at (0,-1) {$ D(\PA_{\eta,\xi'}) $};
\node (cb) at (1,-1){$ D(A(\eta, \xi')\mmod) $};
\draw[->] (aa) -- (ba) node [left,midway] {$\iota $} ;
\draw[->] (ba) -- (ca) node [left,midway] {$L\pi_{\xi'} $} ;
\draw[->] (ab) -- (bb) node [right,midway] {$\operatorname{inf}_{A(\eta,\xi)}^{A(\eta,-)}$};
\draw[->] (bb) -- (cb) node [right,midway] {$A(\eta,\xi') \, \stackrel{L}\otimes_{A(\eta,-)} -$};
\draw[->] (aa) -- (ab) node [above,midway] {$F$};
\draw[->] (ba) -- (bb) node [below,midway] {$F$};
\draw[->] (ca) -- (cb) node [below,midway] {$F$};
}
\]
where $\operatorname{inf}_{A(\eta,\xi)}^{A(\eta,-)}$ denotes the inflation (or restriction) of $A(\eta,\xi)$ modules to modules over $A(\eta,-)$ by the obvious homomorphism.
\end{proof}
Exactly as with the twisting functors, the elements of $\Weyl$ induce isomorphisms
$\psi^s\colon A(\eta,\xi)\overset{\sim}\longrightarrow
A(\eta,s\cdot\xi)$.

\begin{proposition}\label{shuffling-permute}
  The functor $\Psi^s$ is given by pushforward by the isomorphism $\psi^s$.
\end{proposition}

As a corollary of Propositions \ref{twisting-comb} and \ref{shuffling-comb}, we can prove that shuffling and twisting commute.

\begin{corollary}\label{commute} There are natural isomorphisms of functors
  \begin{align}
    \Phi_{\eta}^{\eta'} \circ \Psi_{\xi}^{\xi'} & \cong
    \Psi_{\xi}^{\xi'}\circ \Phi_{\eta}^{\eta'} \colon
    D(\PA_{\eta,\xi}) \to D(\PA_{\eta',\xi'}),\label{pp}\\
    \Phi^{s} \circ \Psi_{\xi}^{\xi'} & \cong
    \Psi_{\xi}^{\xi'}\circ \Phi^{s} \colon
    D(\PA_{\eta,\xi}) \to D(\PA_{s\cdot\eta,\xi'}),\label{sp}\\
    \Phi_{\eta}^{\eta'} \circ \Psi^s & \cong
    \Psi^s\circ \Phi_{\eta}^{\eta'} \colon
    D(\PA_{\eta,\xi}) \to D(\PA_{\eta',s\cdot \xi}),\label{ps}\\
 \Phi^{s'} \circ \Psi^s & \cong
    \Psi^s\circ \Phi^{s'} \colon
    D(\PA_{\eta,\xi}) \to D(\PA_{s'\cdot \eta,s\cdot \xi}),\label{ss}
  \end{align}
\end{corollary}

\begin{proof}
First, we show \eqref{pp}.  Consider the commutative diagram
\[\xymatrix{
A(\eta,\xi) \ar[d] & A(\eta,-) \ar[l]\ar[r]\ar[d] & A(\eta,\xi') \ar[d] \\
A(-, \xi) & R \ar[l]\ar[r] & A(-, \xi') \\
A(\eta', \xi)\ar[u] & A(\eta', -) \ar[u]\ar[l]\ar[r] & A(\eta',\xi') \ar[u] \\
}\]
where all the maps are the obvious inclusions or projections.  
By Proposition \ref{shuffling-comb}, the functor $\Psi_{\xi}^{\xi'}$ is given 
by following across the top row or the bottom row of this diagram,
first pulling back, then pushing forward.
By Proposition \ref{twisting-comb},
$\Phi_{\eta}^{\eta'}$ is given by pushing forward then
pulling back along the left or right column.  Note that since the 
map $A(\eta', \xi) \to A(-,\xi)$ does not take the unit to
the unit, pulling back by this map involves first multiplying
by the idempotent $e_{\eta'}$ and then taking the induced module.

It is sufficient to show that going along the top and right sides
of each small square is naturally isomorphic to going along the
left and bottom sides.  For the upper right and lower left squares
this is simply the fact that these squares commute.  For the other
two squares, this follows from Lemma \ref{Cartesian}.

The isomorphism \eqref{sp} follows from the fact that
the maps $\phi^s\colon A(\eta,\xi)\overset{\sim}\longrightarrow
A(s\cdot\eta,\xi)$ and $\phi^s\colon A(\eta,\xi')\overset{\sim}\longrightarrow
A(s\cdot\eta,\xi')$ are induced by an isomorphism
$\phi^s\colon A(\eta,-)\to A(s\cdot \eta,-)$.  The proof for
\eqref{ps} is the same.
Finally \eqref{ss} follows from the fact that the actions of $\rWeyl$ and
$\Weyl$ on sign vectors commute.
\end{proof}

\subsection{Fundamental group action via twisting functors}\label{sec:pione}
This section is about twisting functors only, so we fix $\vb_0$ and $\xi$ and vary $\eta$ among regular 
parameters. The set of regular $\eta$ is the intersection of $\fk^*_\Z \cong \SgrHZ/\vb_0$ with the complement of
a central hyperplane arrangement in $\fk^*_\R\cong W_\R/V_{0,\R}$.
This arrangement, known as the {\bf discriminantal
arrangement}, was first considered for generic arrangements in \cite{ManSch89} and for
general arrangements in \cite{BaBr97}. 
Its hyperplanes are indexed by the circuits of $\cH_0$,
which are the minimal subsets of hyperplanes with dependent normal vectors.
Given a circuit, the corresponding hyperplane of the discriminantal arrangement
is the locus of $\eta$ such that the hyperplanes of $\cH = \cH_\eta$ indexed by the circuit intersect 
non-transversely.

Let $\Upsilon\subset\fk_\R^*$ be the complement of the discriminantal arrangement, and let
$\Upsilon_\C$ be the complexification of $\Upsilon$,
that is, the set of elements in $\fk^*$ whose real or imaginary part
lies in $\Upsilon$.  If $\eta, \eta'$ are regular, then
$X_{\eta,\xi}$ is equivalent to $X_{\eta',\xi}$ if and only if $\eta$ and $\eta'$ lie in the same connected component of $\Upsilon$.
Choose a subset $B\subset \Upsilon\cap\fk^*_\Z$ consisting of one integral basepoint 
from each connected component of $\Upsilon$.

Since the group $\rWeyl$ acts on $\SgrHZ$ preserving $\vb_0$ and permutes the coordinate 
hyperplanes, it induces an action on $\fk^*_\R\cong W_\R/V_{0,\R}$ which preserves
the discriminantal arrangement, and so it acts on 
$\Upsilon$ and $\Upsilon_\C$.  In fact, $\rWeyl$ 
is generated by reflections in the discriminant hyperplanes 
corresponding to pairs of hyperplanes in $\cH_0$ whose 
normal vectors $h_i|_{\vb_0}$ are equal up to a sign.
We choose the set $B$ to be $\rWeyl$-invariant.

\begin{definition}
The {\bf Deligne groupoid} of the discriminantal arrangement
is the subgroupoid of the fundamental groupoid of $\Upsilon_\C$ 
consisting of paths that begin and end in $B$.  The {\bf Weyl-Deligne}
groupoid is the subgroupoid of the fundamental groupoid of $\Upsilon_\C/\rWeyl$
consisting of paths beginning and ending at $B/\rWeyl$.  Note that $\rWeyl$
acts freely on $B$.
\end{definition}

Since each connected component of $\Upsilon$ is simply connected, the Deligne groupoid is independent
(up to canonical isomorphism) of our choice of $B$.
Furthermore, it has an entirely combinatorial interpretation.
Let the {\bf Deligne quiver} be the quiver with
with vertex set $B$ and edges in
both directions between two base points
if and only if the components of $\Upsilon$ in which they lie are separated by a single
hyperplane.

\begin{theorem}[\cite{Paris}]
The Deligne groupoid is isomorphic to the quotient of the 
fundamental groupoid of the Deligne quiver by the identification of any pair of paths of minimal length between
the same two points.  
\end{theorem}

The group $\rWeyl$ acts on the Deligne quiver in the obvious way on
vertices, and with its action on edges preserving orientations
(i.e. if it interchanges adjacent vertices, it interchanges the edges
between them as well).
\begin{proposition}
  The Weyl-Deligne groupoid is the quotient of the Deligne
groupoid by the action of $\rWeyl$.
\end{proposition}
\begin{proof}
  We need only see that the combinatorial action we have described
  coincides with the action of transformations on the lifting of paths
  from $\Upsilon_\C/\rWeyl$ to $\Upsilon_\C$.  This is clear from the
  realization of the Deligne groupoid using Van Kampen's theorem.
\end{proof}

\begin{theorem}\label{action}
By assigning the equivalence $\Phi_{\eta}^{\eta'}$ to the shortest oriented path in the Deligne
quiver from $\eta$ to $\eta'$, we obtain an action of the Deligne groupoid on
the categories $\{D(\PA_{\eta,\xi})\mid \eta\in B\}$.  The functors
$\Phi^s$ perserve the collection of these functors compatibly with the
action of $\rWeyl$ on the Deligne groupoid.
In particular, for each $\eta\in B$, we obtain an action of
$\pi_1(\Upsilon_\C/\rWeyl, \eta)$ on $D(\PA_{\eta,\xi})$ by auto-equivalences.
\end{theorem}

\begin{proof}
For any $\eta, \eta', \eta''\in B$, we defined in \cite[\S 6.3]{GDKD} a natural transformation from
$\Phi_{\eta'}^{\eta''}\circ\Phi_{\eta}^{\eta'}$ to $\Phi_{\eta}^{\eta''}$.  We need to show that
if $\eta'$ lies on a path of minimal length from $\eta$ to $\eta''$ in the Deligne quiver, this
natural transformation is an isomorphism.  

In the proof of \cite[6.12]{GDKD}, we showed that our natural transformation is a natural isomorphism
if and only if the following condition is satisfied for each $S\subset\{1,\ldots,n\}$:
\begin{itemize}
\item[] For every $\a\in\cP_{\eta,\xi}$ and $\a''\in \cP_{\eta'',\xi}$ such that
$\a\vert_S = \a''\vert_S$ is bounded\footnote{By this we mean that $\xi$ is bounded above
on the chamber of the subarrangement determined by $\a$.}
for the subarrangement $\{H_{i,0}\mid i\in S\}\subset \cH_0$,
there is a sign vector $\a'\in\cP_{\eta',\xi}$ such that $\a\vert_S = \a'\vert_S = \a''\vert_S$.
\end{itemize}

We prove this statement by induction on the size of the complement of $S$.  
If $S = \{1,\ldots,n\}$,
the statement says that $\cP_{\eta,\xi}\cap \cP_{\eta'',\xi}\subset \cP_{\eta',\xi}$; for this it is sufficient
to show that $\cF_{\eta}\cap\cF_{\eta''}\subset\cF_{\eta'}$.
Indeed, a sign vector $\a \in \cF_\eta \cap \cF_{\eta''}$ 
fails to lie in $\cF_{\eta'}$ if and only if there is a circuit $C$ such
that $\a\vert_C$ is infeasible for the sub-arrangement 
$\{H_{i,0}\mid i\in C\}\subset\cH_0$. 
This would imply that the hyperplane in the discriminantal arrangement indexed by $C$
separates $\eta$ and $\eta''$ from $\eta'$, which contradicts the fact that $\eta'$ lies on a path of minimal length from $\eta$ to $\eta''$ in the Deligne quiver.

Now consider the general case.  Choose some $i\notin S$; 
by our inductive hypothesis,
there exists a sign vector on $\{1,\ldots,n\}\smallsetminus\{i\}$ that agrees with $\a$ and $\a''$
on $S$ and is bounded and feasible for the polarized arrangement obtained by deleting the 
$i^\text{th}$
hyperplane of $\cH_{\eta'}$.  Both of the extensions of this sign vector to $\{1,\ldots,n\}$ 
will be bounded, and at least one of them will lie in $\cF_{\eta'}$,
so we have shown the relations in the Deligne groupoid.

For the compatibility with the $\rWeyl$-action, we need only note that
the isomorphism $\phi^s$ sends the bimodules
$e_{\eta}A(-,\xi)e_{\eta'}$ to $e_{s\cdot \eta}A(-,\xi)e_{s\cdot
  \eta'}$, since it simply acts by permuting sign vectors.  Thus, we
are done.
\end{proof}

For any $\eta\in B$, let
$\zeta_\eta\in\pi_1(\Upsilon_\C,\eta)$ be the central element represented by the path $[0,1]\to \Upsilon_\C$
taking $t$ to $e^{2\pi it}\eta$.

\begin{theorem}[\mbox{\cite[6.11]{GDKD}}]\label{serre}
The element $\zeta_\eta$ acts as the Serre functor on $D(\PA_{\eta,\xi})$.
\end{theorem}

\begin{remark} For a regular integral block of BGG category $\cO$, the derived shuffling functor
corresponding to the full twist braid is the Serre functor \cite[4.1]{MS}.  
Thus Theorem \ref{serre} provides further evidence for the analogy between pure
shuffling functors and Irving's shuffling functors for pure braids
(Remark \ref{twisting remark}).
\end{remark}

\subsection{Koszul duality}\label{sec:koszul}
Our last goal is to prove that shuffling and twisting functors are exchanged by Koszul 
duality.  First we recall the basic features of the theory of Koszul duality that we will need.  
Throughout this section let $A = \bigoplus_{j \ge 0} A_j$ be a
finite dimensional nonnegatively graded ring such that $A_0 \cong \bigoplus_{\a \in \cP} \C e_\a$ is
a commutative semisimple ring with primitive idempotents indexed by a finite
set $\cP$.   

Let $A\!\!-\!\!\gr$ denote the category of finitely generated graded $A$-modules.
The ring $A$ is called \textbf{Koszul} if $\Ext^i_{A-\gr}(A_0, A_0[-j]) = 0$
for all $i \ne j$.  If we let  
\[E(A) := \bigoplus_{i \ge 0} \Ext^i_A(A_0, A_0) = \bigoplus_{i, j \ge 0} \Ext^i_{A-\gr}(A_0, A_0[-j])\]
be the Yoneda algebra of $A$, then $A$ is Koszul if and only if
the two gradings on $E(A)$ coincide.  
If $A$ is Koszul, then it is \textbf{quadratic}, that is, it is generated by 
over $A_0$ by $A_1$, with relations are generated in degree $2$.  Furthermore, 
\begin{itemize}
\item $E(A)$ is also Koszul,
\item there is a canonical isomorphism $A \cong E(E(A))$,
\item $E(A)$ is isomorphic to the opposite
algebra of the {\bf quadratic dual} of $A$, which is the 
quadratic ring generated by $A_1^*$ over $A_0$ with relations
orthogonal to the relations of $A$.
\end{itemize}

Suppose that $A$ is Koszul, and suppose in addition that $E(A)$ is
left Noetherian.  Then \cite[2.12.6]{BGS96} gives
an equivalence $D_\sgr(A) \to D_\sgr(E(A)^{\text{op}})$ between the bounded
derived categories of finitely generated graded modules.
Since $A^{\text{op}}$ is also Koszul \cite[2.2.1]{BGS96} and $E(A^{\text{op}})^{\text{op}})\cong E(A)$,
we obtain an equivalence $D_\sgr(A^{\text{op}}) \to D_\sgr(E(A))$.
It will be more convenient for us to consider the contravariant
equivalence
\[K \colon D_\sgr(A) \to D_\sgr(E(A))\]
given by composing this functor with the duality functor $D(A) \to D(A^{\text{op}})$
induced by $M \mapsto M^*$.
The functor $K$ takes (shifted) 
simple objects to (shifted) 
projective objects and vice-versa:
we have $K(A_0e) = E(A)e$ and $K(Ae) = E(A)_0e$ for any $e \in A_0 \cong E(A)_0$.

If $A$ is quadratic (but not necessarily Koszul), we will denote by $A^!$ the opposite algebra
of the quadratic dual of $A$.  Thus if $A$ is Koszul, we have $A^! \cong E(A)$ and
\[K \colon D_\sgr(A) \to D_\sgr(A^!).\]

\begin{remark}
We warn the reader that our notation conflicts with the notation in \cite{BGS96}.  In that paper $A^!$ is defined to be
the quadratic dual of $A$, which is {\em opposite} to the Yoneda algebra $E(A)$.
The reason that we make our definition is that we want an equivalence
$D_\sgr(A) \to D_\sgr(A^!)$ that swaps simples and projectives.  Our equivalence $K$ has this property,
but the basic BGS equivalence $D_\sgr(A) \to D_\sgr(E(A)^{\text{op}})$ takes simples to projectives
and {\em injectives} to simples.  

We also note that our algebras $A(\QPA)$ and $A(\PA)$ are isomorphic
to their own opposites, which makes this conflict academic.  However, it is still important to note
that our equivalence differs from the BGS equivalence by an application of the duality functor.
\end{remark}

We wish to study the interaction of Koszul duality with the inclusion $i\colon eAe \to A$ and
the projection $q\colon A^! \to A^!/A^!\bar e A^!$, where $e \in A_0$ is an idempotent
and $\bar e = 1 - e$ is the complementary idempotent.  The following lemma follows immediately
from the definition of quadratic duality.

\begin{lemma}\label{quadratic duality} If $A$ and the subring $eAe$ are both quadratic, then
$(eAe)^! \cong A^!/A^!\bar e A^!$.
\end{lemma}

The following theorem is the main result of this section;
in the next section we will apply it to the ring $A(\PA)$ to prove the duality of shuffling and twisting.

\begin{theorem}\label{Koszul functoriality} Suppose that $A$ and the subring $eAe$ are both Koszul.  
Then $A^!/A^!\bar e A^!$ is Koszul, and in the diagram
\[
\xymatrix@+15pt{ D_\sgr(eAe) \ar[r]^{K_e}\ar@<.5ex>[d]^{i_*} & D_\sgr(A^!/A^!\bar e A^!) \ar@<.5ex>[d]^{q^*} \\
D_\sgr(A) \ar@<.5ex>[u]^{i^*}\ar[r]^{K} & D_\sgr(A^!) \ar@<.5ex>[u]^{q_*}
}\] 
both squares commute up to natural isomorphism.  Here
the horizontal maps are the appropriate Koszul duality equivalences and the vertical
maps are either pushing forward (tensoring) or pulling back (taking the induced module)
by the homomorphisms $i$ and $q$.
\end{theorem}

\begin{proof} The Koszulity of $A^!/A^!\bar e A^!$ follows from Lemma \ref{quadratic duality}.
To prove the rest of the theorem, we must look closely at how 
the functor $K$ is defined.  On a complex $(M^i, \partial)$ it is given by 
\[ (KM)^p_q := \bigoplus_{\substack{p = -i-j \\ q = l + j}} A^!_l \otimes_{A_0} (M^i_j)^*\]
with differential
\[\partial(a \otimes f) = (-1)^{i+j} \sum_c a \check v_c \otimes v_cf + a \otimes \partial f\]
for $a \in A^!_l$ and $f \in (M^i_j)^*$, where $\{v_c\}$ is a basis of $A_1$, $\{\check v_c\}$ is
the dual basis of $A^!_1$, and $v_cf$ and $\partial f$ are defined by dualizing the actions of
$v_c$ and $\partial$ on $M$.  

Let us show that $K_e  i^* \simeq q_*  K$.
For $M \in D_\sgr(A)$, the underlying vector space of $K_e(i^*M)$ is 
\[A^!/A^!\bar eA^! \otimes_{A_0/\bar e A_0} eM^*,\]
while $q_*(KM)$ is 
\[A^!/A^!\bar eA^! \otimes_{A^!} \left( A^! \otimes_{A_0} M^*\right) = A^!/A^!\bar eA^! \otimes_{A_0} M^*.\]
It is easy to see that these are isomorphic and that the gradings agree.  
To check that the differentials are the same, use a basis $\{v_c\}$ of 
$A_1$ obtained by combining bases of $A_1e$ and $A_1\bar e$.

Next note that $i^* i_*$ is naturally equivalent to the identity functor
on $D_\sgr(eAe)$.  From this it follows that
\[K_e \simeq K_e  i^*  i_* \simeq q_*  K  i^*.\]
It is also easy to see that the restriction of $q_*  q^*$ to the 
full subcategory of $D_\sgr(A^!)$ given by complexes $M^\udot$ 
with $\bar e M^i = 0$ for all $i$ is naturally
equivalent to the identity.  But since $$Ki_*(eAe) = K(Ae) = A^!_0e,$$ the
functor $Ki_*$ lands in this 
subcategory, so we have
\[q^*  K_e \simeq q^*  q_*  K  i^* \simeq K  i^*.\qedhere\]
\end{proof}

\subsection{Duality of twisting and shuffling}\label{sec:duality}
Let $\PA_{\eta_1,\xi} = (\vb_0, \eta_1,\xi)$ and $\PA_{\eta_2,\xi} = (\vb_0, \eta_2,\xi)$
be regular polarized arrangements, and let 
$\PA^!_{\eta^!,\xi_1^!}$, $\PA^!_{\eta^!, \xi^!_2}$ be their Gale duals,
as defined in Section \ref{sec:gd}.  Our twisting and shuffling functors 
have graded versions, given by putting the natural grading on the bimodules
that appear in Propositions \ref{twisting-comb} and \ref{shuffling-comb}.
We denote the graded functors by the same symbols:
\[\Phi_{\eta_1}^{\eta_2} \colon D_\sgr(\PA_{\eta_1,\xi}) \to D_\sgr(\PA_{\eta_2, \xi})\and
\Psi_{\xi^!_1}^{\xi^!_2} \colon D_\sgr(\PA^!_{\eta^!,\xi^!_1}) \to D_\sgr(\PA^!_{\eta^!,\xi^!_2}).\]

\begin{theorem}\label{TwistShuf Koszul} There is a natural equivalence
\[\Psi_{\xi^!_1}^{\xi^!_2} \circ K_1 \simeq K_2 \circ \Phi_{\eta_1}^{\eta_2} ,\]
where $K_i \colon D_\sgr(\PA_{\eta_i,\xi}) \to D_\sgr(\PA^!_{\eta^!,\xi_i^!})$
is the contravariant Koszul equivalence defined in Section \ref{sec:koszul}.
\end{theorem} 

\begin{proof}
We denote by
\[A_\spol^!(\eta^!,-) := e_{\eta^!}\left(Q_n/\langle \vart(\fk^!)\rangle\right)e_{\eta^!},\]
the graded ring whose completion is $A^!(\eta^!,-)$.
The theorem will follow from applying Theorem \ref{Koszul functoriality} to the 
diagram
\[\xymatrix{
D_\sgr(A(\eta_1,\xi)) \ar[r]^{K_1}\ar[d] & D_\sgr(A^!(\eta^!,\xi^!_1)) \ar[d] \\
D_\sgr(A(-,\xi)) \ar[r]^K\ar[d] & D_\sgr(A^!_\spol(\eta^!,-)) \ar[d] \\
D_\sgr(A(\eta_2,\xi)) \ar[r]^{K_2} & D_\sgr(A^!(\eta^!,\xi^!_2))
}\]
where the vertical maps are the respective derived pushforwards and pullbacks, and
the horizontal maps are the Koszul duality functors.  
Proposition \ref{twisting-comb}
tells us that the composition of the left-hand vertical functors is $\Phi_{\eta_1}^{\eta_2}$,
and Proposition \ref{shuffling-comb} tells us that the 
composition of the right-hand vertical functors is 
$\Psi_{\xi^!_1}^{\xi^!_2}$, since the finite dimensionality of $A^!(\eta^!,\xi^!_i)$
implies that the images of $A^!_\spol(\eta^!,-)$ and $A^!(\eta^!,-)$ in this ring are the same.

In order to apply Theorem \ref{Koszul functoriality}, we must verify that $A(-,\xi)$ and $A^!_\spol(\eta^!,-)$ are Koszul
and dual to each other, 
$A(-, \xi)$ is finite dimensional, and $A^!_\spol(\eta^!, -)$ is left Noetherian.
Let
$R := Q_n/\langle \vart(\fk)\rangle$ and
$R^! :=  Q_n/\langle \vart(\fk^!)\rangle$.  
(Note that unlike in Section \ref{sec:comb}, here $R$ is a quotient of the quiver algebra $Q_n$ 
rather than its completion $\wh{Q}_n$.)
Put $e := e_{\eta^!}$ and $\bar{e} := e_{\xi} = 1 - e$, so that
$$A^!_\spol(\eta^!, -) = eR^!e\and A(-, \xi) = R/R\bar eR.$$
From this description it is clear that $A(-,\xi)$ is isomorphic to the ring $A_{\mathrm{ext}}(\PA)$ defined in 
\cite[\S 6.1]{GDKD}, where $\PA = \PA_{\eta_i,\xi}$ for $i = 1$ or $2$
(the definition does not use the parameter $\eta$).  The proposition \cite[6.1]{GDKD} says that this ring
is isomorphic as a vector space to the direct sum of cohomology groups 
\[\bigoplus_{\a,\b\in \cB_\xi} H^*(C_\a \cap C_\b; \C),\]
$\a$ and $\b$ range over the set $\cB_\xi = \cF_{\eta^!}$ and
$C_\a$ is the relative core component indexed by $\a$ in
the hypertoric variety $\fM(\PA^!_{\eta^!, \xi^!_i})$.
Thus $A(-,\xi)$ is manifestly finite dimensional.
To see that $A^!_\spol(\eta^!, -)$ is left Noetherian, simply note that for any
$\a, \b\in \cF_\eta$,
\[e_\a A^!_\spol(\eta^!, -) e_\b = e_\a R^! e_\b\]
is a free module of rank one over $Z(R^!)$, which is isomorphic to
the polynomial ring $\Sym(\C^n/\fk^!)$ via the map $\vart$.
Using \cite[3.8]{GDKD}, we observe
that our description of $A^!_\spol(\eta^!, -)$ coincides with that of \cite[3.1]{GDKD} for the polarized arrangement
$\PA^!_{\eta^!, \xi^!_i}$, with the relation A1 deleted.
The proof of \cite[3.2]{GDKD} then adapts immediately to show that $A^!_\spol(\eta^!, -)$ is quadratic.

At this point we simplify notation by putting $$A := A^!_\spol(\eta^!, -)\and B := A(-,\xi).$$

The rings $R$ and $R^!$ are quadratic dual, where we let the
natural basis of $R_1 = (Q_n)_1 = R^!_1$ given by making length
one paths self-dual up to a 
sign that makes the commutation relations around each
square in the cube quiver dual on the two sides (see
\cite[\S 3.3]{GDKD} for one way to produce these signs).
It now follows from Lemma \ref{quadratic duality} that $A$ and $B$ 
are quadratic dual rings.  Since both rings are isomorphic to their own opposites,
we have $A(-, \xi)\cong A^!_\spol(\eta^!, -)^!$.  It thus remains only to show that the following lemma holds. 
\begin{lemma}\label{Akoszul}
$A$ is Koszul and $B$ is its Koszul dual.  In particular, $A$ has finite global dimension.
\end{lemma}

By \cite[2.6.1]{BGS96}, $A$ is Koszul if and only if its 
Koszul complex is a resolution of $A_0 = A/A_{>0}$.  This 
complex can be defined as follows.  As a vector space
it is $A \otimes B^*$, where we put $\otimes = \otimes_{A_0}$
for the remainder of the proof.  This complex is bigraded, with
$A_i \otimes B^*_j$ in degree $(-j,i + j)$.  It is a graded 
left $A$-module using the second grading, and a complex
using the first grading, where the differential $\partial$ is
the composition
\[A \otimes B^* \to A \otimes (B^*_1 \otimes B^*) \to (A \otimes A_1) \otimes B^* \to A \otimes B^* \]
using the comultiplication on $B^*$, the identification $A_1 = B^*_1$, and the multiplication on $A$.
The map $A \otimes B^* \to A_0$ sends $A_0 \otimes B^*_0$ isomorphically to $A_0$ and kills
all higher degree terms.  As a result, what we need to show is that for any $\a, \b \in \cF_{\eta^!} = \cB_\xi$
the complex $e_\a A \otimes B^* e_\b$ is a resolution of $\C$ if $\a = \b$ and is exact if $\a \ne \b$.

First consider the case $\a \ne \b$.  Since $e_\a A \otimes B^* e_\b$ is a complex of free 
graded modules over $Z := Z(R^!)$, it is exact if and only if the reduced complex
$\C \otimes_Z e_\a A \otimes B^* e_\b$ is exact.
It will be more convenient to work with the dual complex, which we denote by $\Sigma^\udot_{\a\b}$.
To describe $\Sigma^\udot_{\a\b}$ more precisely, first note that for any $\gamma \in \cF_{\eta^!}$,
$\C \otimes_Z e_\a A e_\gamma$ is a copy of $\C$ placed in 
degree
\[d_{\a\gamma} := |\{ 1 \le i \le n \mid \a(i) \ne \gamma(i)\}|.\]
Following the notation of \cite[\S 4.1]{GDKD}, we have an isomorphism of vector spaces
\begin{equation}\label{eqn:chain complex}
\Sigma^\udot_{\a\b} \cong \bigoplus_{\gamma \in \cB_\xi} H^*(C_{\gamma}\cap C_{\b}; \C)[-d_{\a\gamma} - d_{\gamma\b}] = 
\bigoplus_{\gamma \in \cB_\xi} {R}_{\gamma\b}[-d_{\a\gamma} - d_{\gamma\b}],
\end{equation}
where ${R}_{\gamma\b}$ is the reduced face ring of the simplicial complex associated 
to the polyhedron $\Delta_{\gamma\b} := \Delta_\gamma \cap \Delta_\b$.  More precisely, 
if we let $\tilde{R}_{\gamma\b}$ be the quotient of
the polynomial ring $\Sym(\SgrH)\cong \C[e_1,\dots, e_n]$, with generators $e_i$ in degree $2$, by the ideal 
\[\Big\langle \prod_{i\in S} e_i \,\,\Big{|}\,\,  \Delta_{\gamma\b} \cap \bigcap_{i\in S} H_i = \emptyset\Big\rangle,\]
then ${R}_{\gamma\b}$ is the quotient of $\tilde{R}_{\gamma\b}$ by the ideal generated by 
$\sum_i a_ie_i$ for all $a = (a_1, \dots, a_n) \in \fk^!\subset\SgrH$.
Notice that the sum \eqref{eqn:chain complex} can actually be taken over all sign vectors $\gamma \in \{+,-\}^n$, because if $\gamma \notin \cB_\xi = \cF_{\eta^!}$, then $\Delta_{\gamma\beta} = \emptyset$, and so the ideal defining $\tilde{R}_{\gamma\beta}$ contains the product of $e_i$ for $i \in \emptyset$, which is $1$.

Consider the component
\[\partial_{\gamma\gamma'}\colon {R}_{\gamma\b}[-d_{\a\gamma} - d_{\gamma\b}] \to {R}_{\gamma'\b}[-d_{\a\gamma'} - d_{\gamma'\b}]\] of the boundary map.  
It is easy to see for degree reasons that $\partial_{\gamma\gamma'} = 0$ 
unless $d_{\gamma\gamma'} = 1$ and $d_{\a\gamma'} = d_{\a\gamma} - 1$, which means that
$\gamma'$ is one step closer to $\a$ than $\gamma$ is.
If these conditions hold there are two possible cases for the map $\partial_{\gamma\gamma'}$. If $d_{\beta\gamma} = d_{\beta\gamma'} + 1$, then 
$\Delta_{\gamma'\b} \subset \Delta_{\gamma\b}$ and up to a sign and grading shift $\partial_{\gamma\gamma'}$
is the natural quotient ${R}_{\gamma\b} \to {R}_{\gamma'\b}$.  
If instead $d_{\beta\gamma} = d_{\beta\gamma'} - 1$,
then $\Delta_{\gamma\b} \subset \Delta_{\gamma'\b}$, and up to a sign and grading shift $\partial_{\gamma\gamma'}$ is 
induced by multiplication by $e_i$, where $\gamma$ and $\gamma'$ differ in the $i^\mathrm{th}$ place.
The signs come from the choice of signs in the quadratic duality between $R$ and $R^!$, and
they are arranged so that the two paths around any square given by $\gamma_1, \dots, \gamma_4 \in \cF_{\eta^!}$ 
which differ in exactly two places have opposite signs.

To show that $\Sigma^\udot_{\a\b}$ is exact, we consider the related ``equivariant" complex
$\tilde{\Sigma}^\udot_{\a\b}$ which is obtained by replacing each summand ${R}_{\gamma\b}$ with $\tilde{R}_{\gamma\b}$
and using the same formula for the boundary map.  This is a complex of free graded $\C[\fk^!]$-modules, and so it is exact if and only if $\Sigma^\udot_{\a\b} = \tilde{\Sigma}^\udot_{\a\b} \otimes_{\C[\fk^!]} \C$ is exact. The advantage of using this new complex is that each term $\tilde R_{\gamma\beta}$ has a basis of monomials, so we can study exactness one monomial at a time.  For a monomial $m$ in $\C[e_1,\dots,e_n]$ we let its support $\supp m$ be the set of $i$ such that $e_i$ divides $m$.  Then a basis for $\tilde R_{\gamma\beta}$ is given by the monomials $m$ such that $\Delta_{\gamma\beta} \cap \bigcap_{i\in \supp m} H_i$ is nonempty. 

Consider the set $I := \{1 \le i \le n \mid 
\alpha(i) = \beta(i)\}$ indexing hyperplanes which do not separate $\Delta_\alpha$ and $\Delta_\beta$.
For any $\gamma\in \{+,-\}^n$ we define another sign vector $\hat\gamma$ by
\[\hat\gamma(i) =
	\begin{cases}
	\alpha(i) & i \in I \\
	\gamma(i) & i \notin I.
	\end{cases}\]
Let $I_\gamma = \{i \in I \mid \hat\gamma(i) \ne \gamma(i)\}$, a set with $d_{\gamma\hat\gamma}$ elements.  It is easy to see that $\Delta_{\gamma\beta} = \Delta_{\hat\gamma\beta} \cap \bigcap_{i\in I_\gamma} H_i$.  It follows that we have an isomorphism
\[\tilde R_{\gamma\beta}[-d_{\alpha\gamma}-d_{\beta\gamma}] \cong e^{}_{I_\gamma}  \cdot \tilde R_{\hat\gamma\beta}[-d_{\alpha\gamma}-d_{\beta\gamma}+2|I_\gamma|] = e^{}_{I_\gamma}  \cdot \tilde R_{\hat\gamma\beta}[-d_{\a\b}],\] where for a subset $S \subset \{1,\dots, n\}$ we let $e_S$ denote the product of $e_i$, $i\in S$. The 
last equality comes from the easy identities $d_{\a\gamma} = d_{\a\hat\gamma} + |I_\gamma|$, $d_{\a\gamma} = d_{\a\hat\gamma} + |I_\gamma|$, and $d_{\a\b} = d_{\a\hat\gamma}+d_{\b\hat\gamma}$. 

We can therefore rewrite our complex as 
\begin{equation}\label{eq:new presentation of complex}
\tilde\Sigma^\bullet_{\a\b} = \bigoplus_{\gamma\in \{+,-\}^n}e^{}_{I_\gamma}  \cdot \tilde R_{\hat\gamma\beta}[-d_{\alpha\beta}].
\end{equation}
In this formulation the components $\partial_{\gamma\gamma'}$ of the differential are given as follows.  Suppose $\gamma$ and $\gamma'$ differ in the place $i$.  If $i\in I$, then we have $\hat\gamma = \hat\gamma'$ and $I_{\gamma'} = I_{\gamma} \setminus \{i\}$, and up to a sign and shift $\partial_{\gamma\gamma'}$ is the inclusion of ideals of $\tilde R_{\hat\gamma\beta}$.  If $i\notin I$, then $I_\gamma = I_{\gamma'}$ and $\Delta_{\hat\gamma'\beta} \subset \Delta_{\hat\gamma\beta}$, and up to a sign and shift $\partial_{\gamma\gamma'}$ is induced by the natural quotient map $\tilde R_{\hat\gamma'\beta} \to \tilde R_{\hat\gamma\beta}$.

Take any $\delta\in \{+,-\}^n$ such that $\delta_I = \alpha_I = \beta_I$, and let $m$ be a monomial whose image in $\tilde R_{\hat\delta\beta}$ is nonzero.  If 
$\gamma \in \{+,-\}^n$ satisfies $\hat\gamma = \delta$, then $m$ lies in the ideal 
$e^{}_{I_\gamma} \tilde R_{\delta\beta}$ if and only if $I_\gamma \subset \supp m$.  There is one such $\gamma$ for each subset of $I \cap \supp m$, and the differentials between the spaces spanned by $m$ in these terms are plus or minus the identity for each pair of subsets of $I_\gamma \subset \supp m$ differing in one element.  This forms an acylic complex unless $\supp m$ is disjoint from $I$, which implies that $\tilde\Sigma_{\a\b}^\bullet$ is quasi-isomorphic to the complex
\[\hat\Sigma_{\a\b}^\bullet := \bigoplus_{\delta|_I = \alpha|_I = \beta|_I} \tilde R_{\delta\beta}/\langle e_i \mid i \in I\rangle[-d_{\a\b}],\]
where all components of the boundary are (up to sign) the natural quotient homomorphisms.

We can assume without loss of generality that all of the hyperplanes $H_i$ meet $\Delta_\beta$, since if any do not, the complex will be unchanged if we remove them.  Look at the subcomplex $\hat\Sigma_{\a\b}^\bullet(1) \subset \hat\Sigma_{\a\b}^\bullet$ spanned by the image of $1$ in each summand.  The nonzero terms occur for $\delta$ such that $\delta|_I = \alpha|_I = \beta|_I$  and $\Delta_{\delta\beta} \ne \emptyset$.  The map $\delta \mapsto \Delta_{\delta\beta}$ is a bijection between these sign vectors and the set of faces of $\Delta_\beta$  (including $\Delta_\beta$ itself) which are not contained in any $H_i, i\in I$.  Equivalently, these are all faces of $\Delta_\beta$ which are not contained in a facet which is ``invisible" from $\Delta_\alpha$, meaning that a ray from a point of $\Delta_\alpha$ to a point in the interior of the facet meets the interior of $\Delta_\beta$ before it reaches the facet.  The complex $\hat\Sigma_{\a\b}^\bullet(1)$ is a cellular chain complex for the Borel-Moore homology of $\Delta_\beta$ with the invisible facets removed.  This space is the union of the interior of $\Delta_\beta$ with an open disk in the boundary, and so the Borel-Moore homology vanishes.  

Now take any monomial $m$ with support $S$ disjoint from $I$ and consider the subcomplex  $\hat\Sigma_{\a\b}^\bullet(m) \subset \hat\Sigma_{\a\b}^\bullet$ spanned by the images of $m$ in each term.  We can assume that $\Delta_m := \Delta_\beta\cap \bigcap_{i\in \supp m} H_i$ is nonempty, since otherwise $\hat\Sigma_{\a\b}^\bullet(m)$ is zero.  The nonzero terms of this complex  come from sign vectors $\delta$ for which $\Delta_{\delta\beta}$ meets $\Delta_m$ and is not contained in an invisible facet.  In other words our complex computes the Borel-Moore homology of the open star of $\Delta_m$ inside $\Delta_\beta$ minus the invisible facets.  This again is the interior of $\Delta_\beta$ union an open disk in the boundary, so the complex for $m$ is again acyclic.

Finally, consider the case $\a = \b$.  Since $I=\{1,\dots, n\}$, it is immediate that $\hat\Sigma^\bullet_{\a\a}$ has only a single nonzero term, coming from $\gamma = \alpha$, and this term is just $\C$.  So our complex $\tilde{\Sigma}^\udot_{\a\a}$ is a resolution of $\C$.  However, we haven't
proved a direct connection between $e_\a A \otimes B^* e_\a$ and $\tilde{\Sigma}^\udot_{\a\a}$,
but only between their reductions:
\begin{equation}\label{complex iso}
\Sigma^\udot_{\a\a} = \C \otimes_{\Sym(\fk^!)} \tilde{\Sigma}^\udot_{\a\a} \cong (\C\otimes_Z e_\a A \otimes B^* e_\a)^*.
\end{equation}
Let $p$ be a vertex of $\Delta_\a$, and set $\Gamma_p := \{\gamma \in \cF_{\eta^!} \mid p \in \Delta_\gamma\}.$
This set has $2^k$ elements, $k = \dim(\fk^!)$, namely all $\gamma \in \{+,-\}^n$ for which 
$\gamma(i) = \a(i)$ if $p \notin H_i$.

Consider the subspace 
$D = \bigoplus_{\gamma \in \Gamma_p} Z(a_{\a\gamma} \otimes b^*_{\gamma\a})$ of $e_\a A \otimes B^* e_\a$,
where $a_{\a\gamma}$ and $b^*_{\gamma\a}$ are minimal degree elements of $e_\a A e_\gamma$ and $e_\gamma B^* e_\a$, respectively.
It is a subcomplex of graded $Z$-submodules isomorphic to $\bigoplus_{\gamma \in \Gamma_p} Z[-2d_{\a\gamma}]$, and it is easy to check that it is a resolution of $\C$.
On the other hand, let $E$ be the kernel of multiplication by the monomial $\prod_{p \in H_i} e_i$ on the complex 
$\tilde{\Sigma}^\udot_{\a\a}$. Then $\tilde{\Sigma}^\udot_{\a\a}/E$ is isomorphic as a graded $S = \Sym(\fk^!)$-module to $\bigoplus_{\gamma \in \Gamma_p} S[-2d_{\a\gamma}]$,
and one can check that it is a resolution of $\C$.  It follows that $E$ is acyclic, and therefore so is
\[\C \otimes_S E \cong \left[\C \otimes_Z (e_\a A \otimes B^* e_\a / D)\right]^*.\]
It follows that $e_\a A \otimes B^* e_\a$ is a resolution of $\C$, as desired.
\end{proof}

Recall that Gale duality interchanges the two types of ``Weyl groups'' considered in 
Section \ref{sec:weyl-group-symm}: in particular,
we have a natural isomorphism $\rWeyl\cong \Weyl^!$.

\begin{theorem}\label{TwistShuf Koszul2} Under this isomorphism, there is a natural equivalence
\[\Psi^s \circ K_1 \simeq K_2 \circ \Phi^s ,\]
where $K_i \colon D_\sgr(\PA_{\eta_i,\xi}) \to D_\sgr(\PA^!_{\eta^!,\xi_i^!})$
is the contravariant Koszul equivalence defined in Section \ref{sec:koszul}.
\end{theorem} 
\begin{proof}
  If we have an isomorphism of algebras, then there is a naturally
  induced isomorphism of their Koszul duals (given by applying the
  pushforward by the original isomorphism to the sum of the simple
  modules).  Thus, we need only check that the Koszul dual of the isomorphism $\phi^s$
  is $\psi^s$.  For this, it suffices to show that the bijection between sign
  vectors of Gale dual arrangements intertwines the actions of
  $\rWeyl$ and $\Weyl^!$, which we have already noted.
\end{proof}

\begin{remark}
There is an ungraded version of this theorem as well; Koszul duality gives an equivalence between the ungraded derived category $D(\PA_{\eta,\xi})$ and the triangulated category of finitely generated DG-modules over $A^!(\xi^!,\eta^!)$.  There is a version of shuffling and twisting functors in the DG setting as well, given by tensor product with the same bimodules, and the proof of Theorem \ref{TwistShuf Koszul} shows that Koszul duality sends DG shuffling to ungraded twisting and vice versa.
By standard homological algebra, we can infer from Theorem \ref{TwistShuf Koszul}
that all versions (graded, ungraded, and DG) of shuffling functors are equivalences: they are full and faithful on gradable modules, and these generate the triangulated category.
\end{remark}

 \bibliography{./symplectic}
\bibliographystyle{amsalpha}
\end{document}